\documentclass{memo-l}

\usepackage{verbatim}
\usepackage{enumerate}
\usepackage[usenames,dvipsnames] {color}
\usepackage{makeidx}
\usepackage{graphicx}

\newtheorem{theorem}{Theorem}[chapter]
\newtheorem{lemma}[theorem]{Lemma}
\newtheorem{proposition}[theorem]{Proposition}
\newtheorem{corollary}[theorem]{Corollary}

\theoremstyle{definition}
\newtheorem{definition}[theorem]{Definition}

\newtheorem{claim}[theorem]{Claim}

\theoremstyle{remark}
\newtheorem{remark}[theorem]{Remark}

\numberwithin{section}{chapter}
\numberwithin{equation}{chapter}

\renewcommand\d{\mathbb{D}}
\renewcommand\c{\mathbb{C}}
\renewcommand\t{\mathbb{T}}
\renewcommand\r{\mathbb{R}}

\newcommand\royal{\mathcal{R}}

\newcommand\cald{\mathcal{D}}
\newcommand\calr{\mathcal{R}}

\newcommand\cala{\mathcal{A}}
\newcommand\calf{\mathcal{F}}
\newcommand\car[1]{|#1|_{\rm car}}

\newcommand\kob[1]{|#1|_{\rm kob}}

\newcommand\id[1]{{\rm id}_{#1}}

\newcommand\be{\begin{equation}}
\newcommand\ee{\end{equation}}

\newcommand\set[2]{\{ #1 \, : \, #2\}}
\newcommand\norm[1]{\| #1 \|}

\DeclareMathOperator{\ran}{ran}

\DeclareMathOperator{\Car}{{Car}}
\DeclareMathOperator{\Kob}{{Kob}}
\DeclareMathOperator{\aut}{{Aut}}
\DeclareMathOperator{\tr}{{tr}}
\DeclareMathOperator{\re}{Re}
\DeclareMathOperator{\im}{Im}

\newcommand\bbm{\begin{bmatrix}}
\newcommand\ebm{\end{bmatrix}}
\newcommand\al{\alpha}
\newcommand\ga{\gamma}
\newcommand\Ga{\Gamma}
\newcommand\de{\delta}
\newcommand\e{\mathrm{e}}
\newcommand\eps{\varepsilon}
\newcommand\la{\lambda}
\newcommand\ph{\varphi}
\newcommand\ups{\upsilon}

\newcommand\si{\sigma}
\newcommand\nd{nondegenerate }
\newcommand\npep{norm-preserving extension property}

\newcommand\half{{\tfrac 12}}
\newcommand\df{\stackrel{\rm def}{=}}
\newcommand\sym{\mathrm{sym}}
\renewcommand\phi{\varphi}
\newcommand{\twopartdef}[4]
{
	\left\{
		\begin{array}{ll}
			#1 &  #2 \\  \\
			#3 &  #4
		\end{array}
	\right.
}

\makeindex

\begin{document}

\frontmatter

\title[The norm-preserving extension property in the symmetrized bidisc]{Geodesics, retracts, and the norm-preserving extension property in the symmetrized bidisc}


\author{Jim Agler}
\address{Department of Mathematics, University of California at San Diego, CA \textup{92103}, USA}
\thanks{Partially supported by National Science Foundation Grant
DMS 1361720 and by the London Mathematical Society grant 41527}

\author{Zinaida Lykova}
\address{School of Mathematics and Statistics, Newcastle University, Newcastle upon Tyne
 NE\textup{1} \textup{7}RU, U.K.}
\email{Zinaida.Lykova@ncl.ac.uk}
\thanks{Partially supported by the Engineering and Physical Sciences grant EP/N03242X/1}

\author{Nicholas Young}
\address{School of Mathematics and Statistics, Newcastle University, Newcastle upon Tyne NE1 7RU, U.K.
{\em and} School of Mathematics, Leeds University,  Leeds LS2 9JT, U.K.}
\email{Nicholas.Young@ncl.ac.uk}
\date{ 14th March 2016, revised 20th August 2016}

\subjclass[2010]{Primary: 32A07, 53C22, 54C15, 47A57, 32F45; Secondary: 47A25, 30E05}

\keywords{symmetrized bidisc, complex geodesic,  norm-preserving extension property, holomorphic retract, spectral set,  Kobayashi extremal problem, Carath\'eodory extremal problem, von Neumann  inequality, semialgebraic set}


\maketitle

\tableofcontents

\begin{abstract}
A set $V$ in a domain $U $ in $\c^n$ has the {\em norm-preserving extension property} \  if every bounded holomorphic function on $V$ has a holomorphic extension to $U$  
with the same supremum norm.  We prove that an algebraic subset of the 
{\em symmetrized bidisc}
\[
G\df\{(z+w,zw):|z|<1, \ |w| < 1\}
\]
 has the norm-preserving extension property  if and only if it is either a singleton, $G$ itself, a complex geodesic of $G$, or the union of the set $\{(2z,z^2): |z|<1\}$ and a complex geodesic of degree $1$ in $G$.    We also prove that the complex geodesics in $G$ coincide with the nontrivial holomorphic retracts in $G$. Thus, in contrast to the case of the ball or the bidisc, there are sets in $G$ which have the norm-preserving extension property   but are not holomorphic retracts of $G$.  In the course of the proof we obtain a detailed classification of the complex geodesics in $G$ modulo automorphisms of $G$.  We give applications to von Neumann-type inequalities for $\Gamma$-contractions (that is, commuting pairs of operators for which the closure of $G$ is a spectral set) and for symmetric functions of  commuting pairs of contractive operators. We find three other domains that contain sets with the norm-preserving extension property which are not retracts:  they are the spectral ball of $2\times 2$ matrices, the tetrablock and the pentablock.  We also identify the subsets of the bidisc which have the norm-preserving extension property for  symmetric functions.
\end{abstract}

\chapter*{Preface} \label{preface}
The symmetrized bidisc is just one domain in $\c^2$, but it is a fascinating one which amply repays a detailed study.  In the absence of a Riemann mapping theorem in dimensions greater than one there is a choice between studying broad classes of domains  and choosing some special domains.  From the work of Oka onwards there has developed a beautiful and extensive theory of general domains of holomorphy; however, in such generality there can be no expectation of a detailed analytic function theory along the lines of the spectacularly successful theory of functions on the unit disc, which is one of the pinnacles of twentieth century analysis.  For this reason the general theory is complemented by the study of functions on special domains such as the ball and the polydisc \cite{rudin1,rudin2}, or more generally Cartan domains \cite{hua,fuks}.  Thereafter it seems natural to investigate {\em inhomogeneous} domains for which a detailed function theory is possible.  The symmetrized bidisc $G$ in $\c^2$ is emphatically one such.  It is inhomogeneous, but it has a $3$-parameter group of automorphisms.  It is closely related to the bidisc, being the image of the bidisc under the polynomial map $(\la^1+\la^2, \la^1\la^2)$. Its geometry and function theory have many resemblances to, but also many differences from, those of the bidisc.  It transpires that many quantities of interest for $G$, for example the Kobayashi and Carath\'eodory distances, can be calculated in a fairly explicit form.

We originally started studying the symmetrized bidisc in connection with the {\em spectral Nevanlinna-Pick problem}, an interpolation problem for analytic $2\times 2$-matrix-valued functions subject to a bound on the spectral radius.  This problem is of interest in certain engineering problems.  We can claim only partial success in resolving the interpolation problem, but we happened on a rich vein of function theory, which was of particular interest to specialists in the theories of invariant distances \cite{jp} and models of multioperators \cite{bhatta}.

In this paper we determine the holomorphic retractions of $G$ and the subsets of $G$ which permit the extension of holomorphic functions without increase of supremum norm.  The methods we use are of independent interest: we analyse the complex geodesics of $G$, and show that there are {\em five} qualitatively different types thereof.  The type of a complex geodesic $\cald$ can be characterized in two quite different ways.
One way is in terms of the intersection of $\cald$ with the special variety
\[
\calr \df \{(2z,z^2):|z|<1\}.
\]
$\calr$ is itself a complex geodesic of $G$, uniquely specified by the property that it is invariant under every automorphism of $G$.  Remarkably, the type of $\cald$ can also be characterized by the number of Carath\'eodory extremal functions for $\cald$ from a certain one-parameter family of extremals.
In contrast, for the bidisc, one can define types of complex geodesics in an analogous way, and one finds that there are only {\em two} types.

\mainmatter

\chapter{Introduction}\label{intro}

A celebrated theorem of H. Cartan states the following \cite{cartan}.
\begin{theorem}\label{cartext}
If $V$ is an analytic variety in a domain of holomorphy $U$ and if $f$ is an analytic function on $V$ then there exists an analytic function on $U$ whose restriction to $V$ is $f$.
\end{theorem}
In the case that the function $f$ is bounded, Cartan's theorem gives no information on the supremum norm of any holomorphic extension of $f$.
We are interested in the case that an extension exists with the {\em same} supremum norm as $f$.
\begin{definition}\label{npep}
A subset $V$ of a domain $U$ has the {\em norm-preserving extension property} if, for every bounded analytic function $f$ on $V$, there exists an analytic function $g$ on $U$ such that
\[
g|V= f \quad \mbox{ and }\quad \sup_U |g| = \sup_V |f|.
\]
\end{definition}
Here $V$ is not assumed to be a variety; to say that $f$ is analytic on $V$ simply means that, for every $z\in V$ there exists a neighborhood $W$ of $z$ in $U$ and an analytic function $h$ on $W$ such that  $h$  agrees with $f$ on $W\cap  V$.
\index{extension property!norm-preserving }
\index{function!analytic on a subset of a domain}

There are at least two reasons to be interested in the \npep, one function-theoretic and one operator-theoretic.  In the case that $U$ is the bidisc, the description of the solutions of an interpolation problem for holomorphic functions on $U$ leads to norm-preserving extensions, while knowledge of the sets having the \npep\ in the bidisc also leads to an improvement of Ando's inequality for pairs of commuting contractions  \cite[Section 1]{agmc_vn}.

The subsets of the bidisc $\d^2$ having the \npep\ were identified in \cite{agmc_vn}.  In this memoir we shall do the same for the symmetrized bidisc.
Let $\c$ denote the complex field and $\d$ the unit disc $\set{z \in \c}{|z|<1}$. 
\index{$\c$}
\index{$\d$}
The {\em symmetrization map} is the map $\pi:\c^2\to \c^2$  given by
\[
\pi(\lambda^1,\lambda^2) = (\lambda^1+\lambda^2,\lambda^1\lambda^2),\qquad\lambda^1,\lambda^2\in\c,
\]
and  the \emph{symmetrized bidisc} is the set 
\begin{align*}
G &= \pi(\d^2) \\
	&= \{(z+w,zw):z,w\in\d\} \subset \c^2.
\end{align*}
\index{$G$}

Retracts of a domain have the \npep.  A (holomorphic) {\em retraction} of $U$ is a holomorphic map of $\rho:U\to U$ such that $\rho\circ\rho=\rho$, and a {\em retract} in $U$ is a set which is the range of a retraction of $U$.  Evidently, for any retraction $\rho$, a bounded analytic function $f$ on $\ran \rho$ has the analytic norm-preserving extension $f\circ\rho$ to $U$.
\index{retraction}
\index{retract}
Two simple types of retracts of $U$ are singleton sets and the whole of $U$; these retracts are {\em trivial}.
Hyperbolic geometry in the sense of Kobayashi \cite{Koba} provides an approach to the construction of  nontrivial retracts.  A {\em complex geodesic} \label{comge} of $U$ is the range $V$ of a holomorphic map $k:\d\to U$ which has a holomorphic left inverse $C$.  Here, since $C\circ k$ is the identity map $\id{\d}$ on $\d$, clearly $k \circ C$ is a retraction of $U$ with range $V$.

To summarize, for a subset $V$ of a domain $U$, the statements
\begin{enumerate}[(1)]
\item $V$  is a singleton, a complex geodesic or all of $U$,
\item $V$ is a retract in $U$,
\item  $V$  has the \npep \ in $U$
\end{enumerate}
satisfy (1) implies (2) implies (3).
It was shown in \cite{agmc_vn} that, in the case that $U$ is the bidisc $\d^2$, the converse implications also hold.  One might conjecture that the same is true for a general domain in $\d^2$; however, the two main results of the paper show that, for the symmetrized bidisc,  (2) implies (1) but (3) does not imply (2).  
\begin{theorem}\label{main2}
A subset $V$ of $G$ is a nontrivial retract of $G$ if and only if $V$ is a complex geodesic of $G$.
\end{theorem}
\index{theorem!\ref{main2}}
Since complex geodesics in $G$ are algebraic sets, as we prove in Theorem \ref{geosvariet} below, it follows that retracts in $G$ are algebraic sets.

A {\em flat geodesic} of $G$ is a complex geodesic of $G$ which is the intersection of $G$ with a complex line: see Proposition \ref{deg1flat}.

\begin{theorem}\label{main}
 $V$ is an algebraic subset of $G$ having the norm-preserving extension property if and only if  either $V$ is a retract in $G$ or $V=\calr\cup \cald$, where $\calr=\{(2z,z^2): z\in\d\}$ and $\cald$ is a flat geodesic in $G$. 
\end{theorem}
\index{theorem!main}
Sets of the form $\calr\cup \cald$ are not retracts of $G$, but nevertheless have the \npep (see Theorem \ref{anom.thm10}).

 Theorems \ref{main2} and \ref{main} have applications to the function theory of the bidisc, the theory of von Neumann inequalities for commuting pairs of operators and the complex geometry of certain domains in dimensions  three and four.  These applications are given in Chapters \ref{appD2} to \ref{SpecBall}.  They are briefly described in the next chapter.

It should be mentioned that there is a substantial theory of holomorphic extensions that refines Cartan's Theorem \ref{cartext}.  However, as far as we know the results in the literature barely overlap the theorems in this paper.  Previous authors (with the exception of \cite{agmc_vn}) do not consider such a stringent condition as the preservation of the supremum norm, and they typically consider either other special domains, such as polydiscs or polyhedra  \cite{alex, PolHenk,dey}, or strictly pseudoconvex domains \cite{henkin}  (note that $G$ is not smoothly embedded in $\c^2$).

\chapter{An overview}\label{overview}
Much of this memoir is devoted to the proof of the main result, Theorem \ref {main} from the introduction.

Neither necessity nor sufficiency is trivial.  Sufficiency is proved in Chapters  \ref{retractsG} and \ref{anomalous}.  In the former it is shown that complex geodesics in $G$ are algebraic sets and that nontrivial retracts are complex geodesics.  In the latter the Herglotz Representation Theorem  is used to show that $\calr\cup \cald$ has the \npep. The proof of necessity in Theorem \ref{main} takes up the remainder of Chapters \ref{extprobG} to \ref{mainproof}.

  The principal tool is the theory of hyperbolic complex spaces in the sense of Kobayashi \cite{Koba}.  The Carath\'eodory and Kobayashi extremal problems are described in Section \ref{CarKob}.  There is a substantial theory of these two problems in the special case of the symmetrized bidisc; the relevant parts of it are
described in Chapters \ref{extprobG} and \ref{cgeosG}.  More of the known function theory of $G$, with citations, is given in Appendix A.

There are two versions of the Carath\'eodory and Kobayashi extremal problems, one for pairs of points and one for tangent vectors.  In order to treat both versions of the problems simultaneously we introduce the terminology of {\em datums}.  A datum in a domain $U$ is either a pair of points in $U$ or an element of the tangent bundle of $U$; the former is a {\em discrete datum}, the latter an {\em infinitesimal datum}.
If $V$ is a subset of $U$ and $\delta$ is a datum in $U$ then $\delta$ \emph{contacts} $V$ if either 
\index{datum}
\index{contact}
\begin{enumerate}[\rm (1)]
\item $\delta=(s_1,s_2)$ is discrete and  both $s_1$ and $s_2$ lie in $V$, or 
\item $\delta =(s,v)$ is infinitesimal and  either $v=0$ and $s\in V$, or  there exist two sequences of points $\{s_n\}$  and $\{t_n\}$ in $V$ such that $s_n \neq t_n$ for all $n$, $s_n \to s$, \, $t_n \to s$ and
 \[
 \frac{t_n-s_n}{\norm{t_n-s_n}} \to v_0,
 \]
for  some unit vector $v_0$ collinear with $v$. 
\end{enumerate} 
A datum $\de$ is {\em degenerate} if either $\delta=(s_1,s_2)$ is discrete and  $s_1=s_2$ or $\de=(s,v)$ is infinitesimal and $v=0$.
\index{datum!degenerate}
An  important property of $G$ is that every \nd\ datum $\de$ in $G$ contacts a unique complex geodesic of $G$, denoted by $\cald_\de$. 

The proof of  necessity in Theorem \ref{main} is based on an analysis of the datums that contact a subset $V$ of $G$ having the \npep.
We divide the \nd  datums in $G$ into five types, called {\em purely unbalanced, exceptional, purely balanced, royal} and {\em flat}.  The definitions of these types (Definition \ref{extdef10} in Section \ref{5types}) for a datum $\de$ are in terms of the Carath\'eodory extremal functions for the datum $\de$.  We use another special feature of $G$: for every \nd datum in $G$ there is a solution $\Phi_\omega$ of the Carath\'eodory extremal problem having the form
\[
\Phi_\omega(s)= \frac{2\omega s^2-s^1}{2-\omega s^1}, \qquad \mbox{ for } s=(s^1,s^2),
\]
for some $\omega\in\t$, where $\t$ denotes the unit circle in $\c$ (Theorem \ref{Phiunivl}).
\index{$\Phi_\omega$}
For royal and flat datums, {\em every} $\Phi_\omega$ is a Carath\'eodory extremal for $\de$.  If there are exactly two values of $\omega\in\t$ for which $\Phi_\omega$ is extremal for $\de$, then $\de$ is said to be {\em purely balanced}.  If there is a unique $\omega\in\t$ for which $\Phi_\omega$ is extremal for $\de$, then we say that $\de$ is either {\em purely unbalanced} or {\em exceptional}, depending on whether a certain second derivative is zero.  The strategy is to consider the consequences for $V$ of the hypothesis that a datum of each of the five types in turn contacts $V$.

To put this plan into effect we have to establish several results about types of datum.  Firstly, according to Theorem \ref{extprop10}, every \nd datum in $G$ is of exactly one of the five types.  A second important fact is that the type of a \nd datum $\de$ can be characterized in terms of the geometry of the complex geodesic $\cald_\de$; this is the gist of Chapters \ref{delicate}  and \ref{classifyG}, and it requires some detailed calculations.  The upshot is (Theorem \ref{geothm30}) that the type of a datum $\de$ is determined by the number and location of the points of intersection of $\cald_\de^-$ and the royal variety $\calr^-=\{(2z,z^2):|z|\leq 1\}$ (that is, whether these points belong to $G$ or $\partial G$).  This result has two significant consequences: (1) the type of a datum is preserved by automorphisms of $G$, and (2) if two \nd datums contact the same complex geodesic then they have the same type.    We may therefore define the type of a complex geodesic $\cald$ to be the type of any \nd datum that contacts $\cald$, whereupon the complex geodesics of $G$ are themselves partitioned into five types.  In particular, $\calr$ is the sole geodesic of royal type, and the geodesics of flat type are precisely the complex lines
\be\label{flatbe}
\calf_\beta \df\{ (\beta+\bar\beta z,z): z\in\d\}
\ee
for $\beta\in\d$.  These flat geodesics constitute a foliation of $G$.
\index{$\calf_\beta$}
\index{complex geodesic!type of}
Theorem \ref{formgeos} gives concrete formulae for each of the five types of geodesics modulo automorphisms of $G$.

Here is a rough outline of the proof of necessity in Theorem \ref{main}. 
Let $V$ be an algebraic set in $ G$ having the norm-preserving extension property. We may assume that $V$ is neither a singleton set nor $G$.
 We must show that  either $V$ is a complex geodesic in $G$ or $V=\calr\cup \cald$ for some flat geodesic $\cald$. 

 Chapter \ref{Gnpep} proves some consequences of the assumption that $V$ has contact with certain types of datum.  Lemma \ref{extlem20} states that if some flat, royal or purely balanced datum $\de$ contacts $V$   then $\cald_\de \subseteq V$.  According to
Lemma \ref{extlem40}, if a flat datum $\de$ contacts $V$  then $V$ is the flat geodesic $\cald_\de$ (which is a retract in $G$) or $V=\calr\cup \cald_\de$, while Lemma \ref{extlem45} states that if a royal datum contacts $V$ then either $V=\calr$ or $V=\calr\cup\cald$ for some flat geodesic $\cald$.  Hence Chapter \ref{Gnpep} reduces the problem to the case that neither a flat nor a royal datum contacts $V$.  By Lemma \ref{purelybal}, if a purely balanced datum $\de$ contacts $V$ then $V$ is the purely balanced geodesic $\cald_\de$.

In Chapter \ref{Vflat} and in Chapter \ref{mainproof} it is shown that if no flat datum contacts $V$ then $V$ is a properly embedded planar Riemann surface of finite type, meaning that the boundary of $V$ has finitely many connected components. Finite connectedness is proved with the aid of real algebraic geometry, specifically, from the permanence properties of semialgebraic sets, which follow from the Tarski-Seidenberg Theorem. By a classical theorem on finitely connected domains \cite[Theorem 7.9]{conway}, $V$ is conformally equivalent to a finitely-circled domain $R$ in the complex plane, meaning that the boundary $\partial R$ is a union of finitely many disjoint circles.  In Lemma \ref{Vrtr} we prove that if $\partial R$ has a single component then $V$ is a nontrivial retract in $G$ and so is a complex geodesic of $G$. 
 To complete the proof it remains to show that, if $V$ has no contact with any flat, royal or purely balanced datum, then $\partial R$ has a single component. This task is achieved by a winding number argument that depends on known properties of Ahlfors functions for the domain $R$; this argument takes up Sections \ref{prelim} and \ref{analytic}.

Finally, some applications of the main results are presented in the last three chapters.  In Chapter \ref{appD2} we identify the symmetric subsets of the bidisc which have the {\em symmetric} extension property (roughly, the \npep \ for symmetric functions). 
In Theorem \ref{allsymexprop} we show that  there are six mutually exclusive classes of such sets and we give explicit formulae for them.

 In Chapter \ref{AvonN} we give two applications of Theorem \ref{main} to the theory of spectral sets.
Theorem \ref{allsymexprop-appl}  states that the sets $V$ of Theorem \ref{allsymexprop} are the only symmetric algebraic sets of $\d^2$ for which the inequality
\[
\| f(T)\| \le \sup_{V} |f |
\]
holds for all bounded symmetric holomorphic functions $f $ on $V$ and all pairs of commuting contractions $T$  subordinate to $V$ (subordination is the condition needed to ensure that $f(T)$ be well defined). 

For any set $A$ of bounded holomorphic functions on a set $V\subseteq \c^2$,  we say that $V$ is an  {\em $A$-von Neumann set} if
\[
\|f(T)\| \leq \sup_{V} |f |
\]
for all $f \in A$ and all pairs $T$ of commuting contractions which are subordinate to $V$.  Here subordination is the natural notion that ensures that the operator $f(T)$ be well defined.   In Theorem \ref{allsymexprop-appl} we show that if $V$ is a symmetric algebraic subset of $\d^2$ and $A$ is the algebra of all bounded symmetric holomorphic functions on $V$ then $V$ is an $A$-von Neumann set if and only if $V$ belongs to one of the six classes of sets which are described by explicit formulae in Theorem \ref{allsymexprop}.

For the second application we need to introduce a variant of the notion of $A$-von Neumann set that is adapted to $G$.
Consider a subset $V$ of $G$ and let $A$ be a set of bounded holomorphic functions on $V$.  We say that $V$ is a {\em $(G, A)$-von Neumann set} if $V$ is a spectral set for every $\Ga$-contraction $T$ subordinate to $V$.  Here a $\Gamma$-contraction is a pair of commuting operators for which $\Ga$ is a spectral set.  Let $A$ be the algebra of all bounded holomorphic functions on $V$.
In Theorem \ref{spectral_sets_G} we show that  if $V$ is algebraic then  $V$ is a $(G,A)$-von Neumann set  if and only if  either $V$ is a retract in $G$ or $V=\calr\cup \cald$, where $\cald$ is a flat geodesic in $G$ (as in equation \eqref{flatbe} above).

In Chapter  \ref{SpecBall} we observe that in any domain which contains $G$ as a holomorphic retract there are sets that have the \npep\  but are not retracts.  In particular this observation applies to the $2\times 2$ spectral ball (which comprises the $2\times 2$ matrices of spectral radius less than one) and two domains in $\c^3$ known as the tetrablock and the pentablock.

  In a digression in Section \ref{caseD2} we show that our geometric method based on uniqueness of solutions of certain Carath\'eodory problems can be used to give an alternative proof of a result of Heath and Suffridge \cite{HS} on the retracts of the bidisc (they all have the form of either a singleton, $\d^2$ itself or a complex geodesic).

Appendix A gathers together the relevant known facts about the geometry and function theory of $G$, with references and a few proofs.   Appendix B is intended to help the reader absorb the notions of the five types of complex geodesic by means of a table of their main properties and a series of cartoons.

\chapter{Extremal problems in the symmetrized bidisc $G$}\label{extprobG}

In this chapter we shall set out our notation and give a brief exposition of some known properties of the Carath\'eodory and Kobayashi extremal problems on the symmetrized bidisc.   We introduce the notion of a {\em datum}, which permits the simultaneous analysis of the discrete and infinitesimal versions of the Kobayashi and Carath\'eodory problems.

If $U$ is an open set in $\c^n$, then by a \emph{datum in $U$} \index{datum} we mean an ordered  pair $\delta$ where either $\delta$ is \emph{discrete}, that is, has the form \index{datum!discrete}
\[
\delta =(s_1,s_2)
\]
where $s_1,s_2 \in U$, or $\delta$ is \emph{infinitesimal}, that is, $\delta$ is an element of the tangent bundle $TU$ of $U$, and so has the form
\index{datum!infinitesimal}
\[
\delta = (s,v)
\]
where $s \in U$ and $v$ is a tangent vector to $U$ at $s$. We identify the tangent vector 
\[
v^1\frac{\partial}{\partial s^1}+\dots+ v^n \frac{\partial}{\partial s^n} \quad \mbox{ with } \quad (v^1,\dots,v^n) \in\c^n.
\]
If $\delta$ is a datum, we say that $\delta$ is \emph{degenerate} if either $\delta$ is discrete and $s_1=s_2$ or $\delta$ is infinitesimal and $v=0$, and {\em nondegenerate} otherwise.
\index{datum!nondegenerate}
If $U \subseteq \c^{n_1}$ and $\Omega \subseteq \c^{n_2}$ are two open sets we denote by $\Omega(U)$ the set of holomorphic mappings from $U$ into $\Omega$. 
\index{$\Omega(U)$}
Furthermore, if $F\in \Omega(U)$, $s\in U$, and $v \in \c^{n_1}$, we denote by $D_v F(s)$  the directional derivative at $s$ of $F$ in the direction $v$.

For any set $X$ we denote by $\id{X}$ the identity map on $X$.
\index{$\id{X}$}

 If $U$ and $\Omega$ are domains, $F\in \Omega(U)$, and $\delta$ is a datum in $U$, we define a datum $F(\delta)$ in $\Omega$ by
 \[
 F(\delta) = (F(s_1),F(s_2))
 \]
 when $\delta$ is discrete and by
 \[
 F(\delta)  =F_*\de  =(F(s),D_v F (s))
 \]
 when $ \de$ is infinitesimal.

Note that if $U, V$ and $W$ are domains,  $\de$ is a datum in $U$ and $F\in V(U), \, G\in W(V)$, then
\be\label{chainrule}
(G\circ F)(\de)= G(F(\de)).
\ee

Let $\delta$ be a datum in $\d$.  We define the {\em modulus} of $\de$, denoted by $|\delta|$, by
\[
|\delta| = \left|\frac{z_1 -z_2}{1-\bar{z_2}z_1}\right|,
\]
when $\delta=(z_1,z_2)$ is discrete,
and by
\[
|\delta|=\frac{|c|}{1-|z|^2}
\]
when $\delta = (z,c)$ is infinitesimal.  \index{datum!modulus of}

Thus $|\cdot|$ is the pseudohyperbolic distance on $\d$ in the discrete case and the Poincar\'e metric on the tangent bundle of $\d$ in the infinitesimal case.

For any domain $U$, the group of automorphisms of $U$ will be denoted by $\aut U$.  
\index{$\aut U$}
We shall make frequent use of the Blaschke factor $B_\al$, defined for $\al\in\d$ by
\[
B_\al(z)= \frac{z-\al}{1-\bar\al z}\quad\mbox{ for } z\in\d.
\]
\index{$B_\al$}
\index{Blaschke factor}

\section{The Carath\'eodory and Kobayashi extremal problems}\label{CarKob}

If $U$ is a domain (an open set in $\c^n$ for some $n$) and $\delta$ is a datum in $U$, then we may consider the Carath\'eodory Optimization Problem, that is, to compute the quantity $\car{\delta}$ defined by
\[
\car{\delta}=\sup_{F\in \d(U)} |F(\delta)|.
\]
\index{Carath\'eodory extremal problem}
\index{$\car{\de}$}
In the sequel we shall refer to this problem as $\Car(\delta)$. 
\index{$\Car(\de)$}
The notation suppresses the dependence of both $\car{\delta}$ and $\Car (\delta)$ on $U$. However, this will not be a cause for confusion, as the domain $U$ will always be clear from the context. We say that \emph{C solves} $\Car (\delta)$ if $C\in \d(U)$ and
\[
\car{\delta} = |C(\delta)|.
\]
In \cite[Section 2.1]{jp}, the function $\car{\cdot}$, when applied to discrete datums, is called the M\"obius pseudodistance for $U$, while, when applied to infinitesimal datums, $\car{\cdot}$ is called the Carath\'eodory-Reiffen pseudometric for $U$. For any \nd datum $\delta$ in $U$, there exists  $C \in \d(U)$ that solve $\Car(\delta)$  \cite{Koba}.

If $U$ is a domain and $\delta$ is a datum in $U$, then we may also consider the Kobayashi Extremal Problem for $\de$, that is, to compute the quantity $\kob{\delta}$ defined by
\be \label{defkob}
\kob{\delta}=\inf \{ |\zeta| : \zeta \mbox{ is a datum } \mbox{ in }\d \mbox{ and there exists }f\in U(\d)\mbox{ such that } f(\zeta) =\delta\}.
\ee
\index{Kobayashi extremal problem}
\index{$\kob{\de}$}
In the sequel we shall refer to this problem as $\Kob(\delta)$.  
\index{$\Kob(\de)$}
We say that \emph{$k$ solves} $\Kob(\delta)$ if $k\in U(\d)$ and there is a datum $\zeta$ in $\d$ such that 
\[
k(\zeta) = \delta \quad \mbox{ and }\quad
\kob{\delta} = |\zeta|.
\]
In \cite[Section 3.1]{jp}, for discrete datums, the function $\tanh^{-1}\kob{\cdot}$ is called the Lempert function for $U$.  In \cite[Section 3.5]{jp}, for infinitesimal datums, the function $\kob{\cdot}$ is called the Kobayashi-Royden pseudometric for $U$.

A domain $U$ is said to be {\em taut} if the space $U(\d)$ is normal (for example, \cite[Section 3.2]{jp}).
If $U$ is a taut domain then, for any \nd datum $\delta$ in $U$, there exists $k \in U(\d)$  that solve $\Kob(\delta)$  \cite[Section 3.2]{jp}.

$G$ is a hyperconvex domain \cite[Remark 7.1.6]{jp} and therefore taut \cite[Remark 3.2.3]{jp}.  Hence $\kob{\de}$ is attained for any datum $\de$ in $G$.

 It is immediate from the Schwarz-Pick Lemma that, for any \nd datum $\de$ in a domain $U$, $\car{\de}\leq\kob{\de}$.
\begin{proposition}\label{Cksolve}
Let $h\in U(\d), \ C\in\d(U)$ be such that $C\circ h=\id{\d}$.  If $\zeta$ is a datum in $\d$ and $\de=h(\zeta)$ then $h$ solves $\Kob(\de)$.
\end{proposition}
For then
\[
\car{\de}\geq |C(\de)|=|C\circ h(\zeta)|=|\zeta|\geq \kob{\de} \geq \car{\de}
\]
and so $|\zeta|=\kob{\de}$.

\section{The Carath\'eodory extremal problem $\mathrm{Car}(\delta)$ for $G$}\label{carG}

We say that a set $\mathcal{C} \subseteq \d(U)$ is a \emph{universal set for the Carath\'eodory problem on $U$} if, whenever $\delta$ is a \nd datum in $U$, there exists a function $C \in \mathcal{C}$ such that $C$ solves $\Car (\delta)$.
\index{Carath\'eodory extremal problem!universal set for}

On the bidisc, the two co-ordinate functions constitute a universal set for the Carath\'eodory problem.
On $G$, there is a one-parameter family that constitutes a universal set.
\begin{definition}\label{defPhi}
The function $\Phi$ is defined for $(z,s^1, s^2) \in \c^3$ such that $z s^1 \neq 2$ by 
\be\label{ext10}
\Phi(z,s^1, s^2) = \frac{2 z s^2 -s^1}{2-z  s^1}. 
\ee
We shall write $\Phi_z(s)$ as a synonym for $\Phi(z,s^1, s^2)$
where $s=(s^1,s^2)$.
\end{definition}
\index{$\Phi_\omega(s)$}
 $\Phi_\omega$ does belong to $\d(G)$ for every $\omega\in\t$: see Proposition \ref{elG} in Appendix A.  The following statement is in Theorem 1.1 (for discrete datums) and Corollary 4.3 (for infinitesimal datums) in \cite{AY04}.
\begin{theorem}\label{extthm10}
The set $\mathcal{C} = \set{\Phi_\omega}{ \omega \in \t}$ is a universal set for the Carath\'eodory problem on $G$.
\end{theorem}

 Theorem \ref{extthm10} provides a concrete way of analyzing the Carath\'eodory problem on $G$.
Indeed, if, 
for a \nd datum $\delta$, we define a function $\rho_\delta$ on the circle by the formula
\be\label{defrhode}
\rho_\delta(\omega) = |\Phi_\omega (\delta)|^2, \qquad \omega \in \t,
\ee
\index{$\rho_\de$}
then Theorem \ref {extthm10} implies that
\be\label{ext20}
\car{\delta}^2=\max_{\omega \in \t} \rho_\delta(\omega).
\ee

\section{Five types of datum $\delta$ in $G$}\label{5types}

The following definition formalizes five qualitative cases for the behavior of the Carath\'eodory extremal problem for $G$.
\begin{definition}\label{extdef10}
Let $\delta$ be a \nd datum in $G$.  We say that $\delta$ is
\begin{enumerate}[\rm (1)]
\item  \emph{purely unbalanced} if $\rho_\delta$ has a unique maximizer $\omega_0 = e^{it_0}$ and
    \[
    \frac{d^2}{dt^2} \rho_\delta(e^{it})\big|_{t=t_0} < 0;
    \]
\item  \emph{exceptional} if $\rho_\delta$ has a unique maximizer $\omega_0 = e^{it_0}$ and
    \[
    \frac{d^2}{dt^2} \rho_\delta(e^{it})\big|_{t=t_0} = 0;
    \]
\item  \emph{purely balanced} if $\rho_\delta$ has exactly two maximizers in $\t$;
\item  \emph{royal} if $\rho_\delta$ is constant and $C(s) = \tfrac12 s^1$ solves $\Car (\delta)$;
\item  \emph{flat} if $\rho_\delta$ is constant and $C(s) = s^2$ solves $\Car (\delta)$.
\end{enumerate}
\index{datum!purely unbalanced}
\index{datum!exceptional}
\index{datum!purely balanced}
\index{datum!royal}
\index{datum!flat}
\index{datum!balanced}
\index{datum!unbalanced}
\index{datum!type of}
 Finally, we say that $\delta$ is \emph{balanced} if $\delta$ is either exceptional, purely balanced, or royal and we say that $\delta$ is \emph{unbalanced} if $\delta$ is either purely unbalanced or flat.
\end{definition}

It may be helpful to look at the cartoons in Appendix B at this point; they provide visual images of geodesics of the five types.

\newtheorem{pentathm}[equation]{Pentachotomy Theorem}
\begin{pentathm}\label{extprop10} \em
If $\delta$ is a \nd datum in $G$, then exactly one of the cases $(1)$ to $(5)$ in Definition {\rm \ref{extdef10}} obtains. In particular, either $\delta$ is balanced or $\delta$ is unbalanced and not both.
\end{pentathm}
\index{pentachotomy}
\begin{proof}
Case 1: $\de$ is discrete.\\
(i)  Suppose that $\rho_\de$ is constant on $\t$.  By Theorem \ref{flatandroyal} either (a) $(s_j^1)^2=4s_j^2$ for $j=1,2$ or (b) there exists $\beta\in\d$ such that $s_j^1= \beta+\bar\beta s_j^2$ for $j=1,2$.  In case (a) we may write $s_j=(2z_j,z_j^2)$ where $z_j=\half s_j^1$.  We then find that
\[
\Phi_\omega(s_j)=-z_j
\]
independently of $\omega$, and so, by equations \eqref{defrhode} and  \eqref{ext20},
\[
\car{\de}^2=\rho_\de(\omega)=|\Phi_\omega(\de)|^2= |(\Phi_\omega(s_1),\Phi_\omega(s_2))|^2=|(-z_1,-z_2)|^2=|(z_1,z_2)|^2= |(\half s_1^1, \half s_2^1)|^2.
\]
Since the function $C(s)=\half s^1$ belongs to $\d(G)$, it follows that $C$ solves $\Car(\de)$, and so, according to Definition \ref{extdef10}, $\de$  is a royal datum.

In case (b) we have $s_j=(\beta+\bar\beta z_j, z_j)$, where $z_j=s_j^2$.      A simple calculation shows that, for any $\omega \in\t$,
\[
\Phi_\omega(s_j)=m_\omega(z_j) \quad \mbox{ for } j=1,2,
\]
for some $m_\omega \in\aut\d$.  Again  by equations \eqref{defrhode} and  \eqref{ext20},
\[
\car{\de}^2=\sup_\omega|\Phi_\omega(\de)|^2= |(\Phi_\omega(s_1),\Phi_\omega(s_2))|^2=|(m_\omega(z_1),m_\omega(z_2))|^2=|(z_1,z_2)|^2= |( s_1^2,  s_2^2)|^2.
\]
Hence $C(s)=s^2$ solves $\Car(\de)$, and so $\de$ is a flat datum.

(ii) Now suppose that  $\rho_\delta$ is not constant on $\t$. It is clear from the definition of  
$\rho_\delta$,
\[
\rho_\delta(e^{it}) =\left|\Phi_{e^{it}}(s_1,s_2) \right|^2 = \left| 
\frac{\Phi_{e^{it}}(s_1) - \Phi_{e^{it}}(s_2) }{1 - \overline{\Phi_{e^{it}}(s_2)} \Phi_{e^{it}}(s_1)}  \right|^2,
\]
that the function $t \mapsto \rho_\delta(e^{it})$ is smooth and $2 \pi$-periodic, and so attains its maximum over $[0, 2 \pi)$ at some point $t_0$, which is a critical point satisfying
\[
\frac{d^2}{dt^2} \rho_\delta(e^{it})\big|_{t = t_0} \le 0.
\]
In fact, by Theorem \ref{trichotomy},  $\rho_\delta$ attains its supremum over $\t$ at either one or two points. In the former case $\delta$ is either purely unbalanced or exceptional, and in the latter case $\delta$ is  purely balanced. We have shown that, for any \nd discrete datum $\delta$ in $G$, at least one of (1) to (5) is true.

Conditions (1), (2) and (3) are clearly mutually exclusive. Moreover, since they together correspond to $\rho_\delta$ being nonconstant on $\t$, they are inconsistent with (4) and (5). It remains to show that $\delta$ cannot be both flat and royal. If it is, then by the calculations above there exists $\beta \in \d$ such that both $s_1$ and $s_2$ lie in the intersection of the sets
\[
\{(\beta +\bar{\beta} z, z): z \in \d \} \;\; \text{and} \;\; \{ s \in G: (s^1)^2 = 4 s^2 \}.
\]
Thus the equation
\[
( \beta + \bar{\beta} z)^2 = 4 z
\]
has two distinct roots for $z \in \d$; this is easily seen to be false. Hence a \nd discrete datum cannot be both royal and flat. We have shown that exactly one of the cases (1) to (5) obtains for  any discrete datum $\delta$.

Case 2: $\delta= (s_1, v)$ is a \nd infinitesimal datum in $G$.\\
(i)  Suppose that $\rho_\de$ is constant on $\t$. By Theorem \ref{flatroyalesimal},
 $\rho_\delta$ is constant on $\t$ if and only if one of the following conditions holds:\\
{\rm (a)} there exists $z \in \d$ such that $s_1= (2z, z^2)$ and $v$ is collinear with $(1,z)$,\\
{\rm (b)} there exist $\beta, z \in \d$ such that $s_1= (\beta + \bar{\beta}z, z)$ and $v$ is collinear with $(\bar{\beta},1)$.\\
Thus, according to Definition \ref{extdef10}, $\de$  is a royal datum in case (a) and 
 a flat datum in case (b). \\
(ii) Now suppose that  $\rho_\delta$ is not constant on $\t$. 
Let $\delta=(s_1,v)$.  By the definition of  
$\rho_\delta$, for $\omega \in \t$, 
\[
\rho_\delta(e^{it}) =\left|\Phi_{e^{it}}(s_1,v) \right|^2 = \left|(\Phi_{e^{it}}(s_1),D_v\Phi_{e^{it}}(s_1)) \right|^2 = \frac{|D_v\Phi_{e^{it}}(s_1)|^2}{(1-|\Phi_{e^{it}}(s_1)|^2)^2}.
\]
It is clear that the function $t \mapsto \rho_\delta(e^{it})$ is smooth and $2 \pi$-periodic, and so attains its maximum over $[0, 2 \pi)$ at some point $t_0$, which is a critical point satisfying
\[
\frac{d^2}{dt^2} \rho_\delta(e^{it})\big|_{t = t_0} \le 0.
\]
By Theorem \ref{trichotomy},  $\rho_\delta$ attains its supremum over $\t$ at either one or two points. In the former case $\delta$ is either purely unbalanced or exceptional, and in the latter case $\delta$ is  purely balanced. We have shown that, for any \nd infinitesimal datum $\delta$ in $G$, at least one of (1) to (5) is true.

Conditions (1), (2) and (3) are clearly mutually exclusive. Moreover, since they together correspond to $\rho_\delta$ being nonconstant on $\t$, they are inconsistent with (4) and (5). It remains to show that $\delta$ cannot be both flat and royal.
 If it is, then 
there exists $z \in \d$ such that 
\[
s_1= (2z, z^2) \;\;\text{and} \;\;v \;\;\text{is collinear with} \;\; (1,z),
\]
and there exists $\beta \in \d$ such that
\[
s_1=(\beta +\bar{\beta} z^2, z^2)\;\;\text{and} \;\;v \;\;\text{is collinear with} \;\; (\bar\beta,1).
\]
Thus $(1,z)$ is collinear with $ (\bar\beta,1)$, and so $1=\bar\beta z$, which is impossible for $z, \beta \in \d$.
Hence a \nd infinitesimal datum cannot be both royal and flat. We have shown that exactly one of the cases (1) to (5) obtains for $\delta$.
\end{proof}

We shall have more to say about the above classification scheme of datums using the relation \eqref{ext20}. Later, in Theorem \ref{geothm30}   we shall give necessary and sufficient geometrical conditions for a datum to be one of the five types in Definition \ref{extdef10}. Additionally, the balanced/unbalanced dichotomy is characterized in geometric terms in Section \ref{BalGeo}  (Theorem \ref{geothm30} below).

Let $\delta$ be a datum in $G$. Observe that if $C$ solves $\Car (\delta)$, then, as automorphisms of $\d$ act as isometries in the pseudohyperbolic metric,  $m \circ C$ also solves $\Car (\delta)$. Thus the solution to $\Car (\delta)$ is never unique. This motivates the following definition.
\begin{definition}\label{extdef20}
If $\delta$ is a \nd datum in $G$, we say that \emph{the solution to $\Car (\delta)$ is essentially unique}, if whenever $C_1$ and $C_2$ solve $\Car (\delta)$ there exists  $m\in\aut\d$ such that $C_2=m \circ C_1$.
\end{definition}

\section{The Kobayashi extremal problem $\mathrm{Kob} (\delta)$ for $G$}\label{kobG}

Let $\delta$ be a datum in $G$. Just as solutions to $\Car (\delta)$ are never unique, nor are solutions to $\Kob (\delta)$: if $m\in\aut\d$ and $f$ solves $\Kob(\delta)$, then $f \circ m$ solves $\Kob(\delta)$. This suggests the following analog of Definition \ref{extdef20}.
\begin{definition}\label{extdef30}
If $\delta$ is a \nd datum in a domain $U$, we say that \emph{the solution to $\Kob (\delta)$ is essentially unique}, if, whenever $f_1$ and $f_2$ solve 
$\Kob (\delta)$, there exists $m\in\aut\d$ of $\d$ such that $f_2=f_1 \circ m$.
\end{definition}

It is easy to see that if $U$ is strictly convex and the boundary of $U$ is smooth then the solution of any Kobayashi extremal problem in $U$ is essentially unique \cite[Proposition 11.3.3]{jp}. On the other hand, if $U$ is merely assumed to be convex, it does not follow that Kobayashi extremals are essentially unique -- witness the example $U=\d^2$.   It is noteworthy that for $G$, despite its non-convexity, solutions of Kobayashi problems are nevertheless essentially unique.  For discrete datums, 
the following result is contained in \cite[Theorem 0.3]{AY06}.  For infinitesimal datums it is proved in Appendix A -- see Theorem \ref{Kob_ess_unique}.  
\begin{theorem}\label{extprop30}
If $\delta$ is a \nd  datum in $G$ then the solution to $\Kob (\delta)$ is essentially unique.
\end{theorem}

\begin{remark}\label{essun} \rm
This theorem implies that if, for some \nd datum $\de$ in $G$, two solutions to $\Kob (\delta)$ agree at some \nd datum in $\d$ then they coincide.
\end{remark}

\newpage

\chapter{Complex geodesics in $G$} \label{cgeosG}
 
In Section \ref{cgeos} we  establish some facts required for the proof of Theorem \ref{geothm30}.  We begin with some generalities.
In the present context it is convenient to use the term `geodesic' to denote a {\em set}, rather than a holomorphic map on $\d$ (as for example in \cite{Koba}).  For the latter notion we reserve the term `complex C-geodesic'.

\section{Complex geodesics and datums in $G$}\label{cgeos}

\begin{definition}\label{defcompgeo}
Let $U$ be a domain and let $\cald\subset U$.  We say that $\cald$ is a {\em complex geodesic} in $U$ if there exists a function $k\in U(\d)$ and a function $C\in \d(U)$ such that $C\circ k=\id{\d}$ and $\cald=k(\d)$.
\end{definition}
\index{complex geodesic}

Kobayashi \cite[Chapter 4, Section 6]{Koba}, following \cite{ves}, defines a {\em complex C-geodesic} in a domain $U$  to be a map $k\in U(\d)$ such that, for every \nd datum $\zeta$ in $\d$,
\[
 \car{k(\zeta)} =|\zeta|.
\]
\index{complex C-geodesic}
It follows easily from the Schwarz-Pick Lemma that $k\in U(\d)$ is a complex C-geodesic if and only if $k$ has a holomorphic left inverse.  Thus a map $k$ is a complex C-geodesic if and only if $k(\d)$ is a complex geodesic in the sense of Definition \ref{defcompgeo}\footnote{Kobayashi also uses the term `complex geodesic', in a sense which differs slightly from ours, though the difference is not significant for the domains studied in this paper.}.

A notable advance in the hyperbolic geometry of domains in  $\c^n$ was the seminal theorem of Lempert \cite{lem81} to the effect that the Carath\'eodory and Kobayashi metrics agree on any bounded convex domain $U$: $\kob{\cdot}=\car{\cdot}$ in $U$.
\index{theorem!Lempert's}
 It implies that if $\delta$ is a \nd datum in a bounded convex domain $U \subseteq \c^n$, then there exists a solution $C$ to $\Car(\delta)$ and a solution $k$ to $\Kob (\delta)$ such that
\be\label{geo10}
C \circ k =\id{\d}.
\ee
This result  guarantees a plentiful supply of complex geodesics in bounded convex domains.

The relationship between datums in $U$ and complex geodesics in $U$ can be cleanly described using the notion of contact. 
\begin{definition}\label{defcontact}
Let $U$ be a domain,  let $V$ be a subset of $U$ and let $\delta$ be a datum in $U$.  Then $\delta$ \emph{contacts} $V$ if either 
\begin{enumerate}[\rm (1)]
\item $\delta=(s_1,s_2)$ is discrete and  $s_1\in V$ and $s_2 \in V$, or 
\item $\delta =(s,v)$ is infinitesimal and  either $v=0$ and $s\in V$, or  there exist two sequences of points $\{s_n\}$  and $\{t_n\}$ in $V$ such that $s_n \not=t_n$ for all $n$, $s_n \to s$, \, $t_n \to s$ and
 \[
 \frac{t_n-s_n}{\norm{t_n-s_n}} \to v_0,
 \]
for  some unit vector $v_0$ collinear with $v$. 
\end{enumerate} 
\end{definition}
\index{contact}
\begin{remark} \label{F1=F2} \rm If a datum $\de$ in $U$ contacts a set $V\subseteq U$, then, for any holomorphic map $F$ on $U$, the value of $F(\de)$ only depends on $F|V$.
It is obvious for discrete $\de$ and for the infinitesimal datum 
$\delta =(s,0)$.  Consider a \nd infinitesimal datum
$\delta =(s,v)$ that contacts a subset $V$ of $U$.  Then there exist two sequences of points 
$\{s_n\}$  and $\{t_n\}$ in $V$ such that $s_n \not=t_n$ for all $n$, $s_n \to s$, \, $t_n \to s$ and
 \[
 \frac{t_n-s_n}{\norm{t_n-s_n}} \to v_0,
 \]
for  some unit vector $v_0$ collinear with $v$. Note that
\[
t_n = s_n + \frac{t_n-s_n}{\norm{t_n-s_n}} \norm{t_n-s_n},
\]
and so, for any domain $\Omega$ and $F \in \Omega(U)$,
\[
F(s) =\lim_{n\to\infty} F(s_n) \quad\mbox{ and }\quad
D_v F(s) = \lim_{n \to \infty} \frac{F(t_n) -F(s_n)}{\norm{t_n-s_n}}.
\]
Thus $F(\de)$ is determined by $F|V$.
\end{remark}

The notion of contact permits Lempert's discovery to be expressed in the following geometric fashion: 

{\em For every \nd datum $\de$ in a bounded convex domain $U$ there is a complex geodesic in $U$ that $\de$ contacts.}

The following proposition relates the concept of contact to holomorphic maps on $\d$.  
\begin{proposition}\label{geoprop10}
Let $U$ be a domain in $\c^n$.
\begin{enumerate}[\rm (1)]
    \item   If $f\in U(\d)$ and $\zeta$ is a datum in $\d$ then $f(\zeta)$ contacts $f(\d)$; 
 \item if $k$ is a complex C-geodesic in $U$ and $\delta$ is a datum in $U$ that contacts $k(\d)$, then there exists a datum $\zeta$ in $\d$ such that $\delta =k(\zeta)$.
\end{enumerate}
\end{proposition}
\begin{proof}
(1)  The statement is trivial for discrete datums.  Let $\zeta=(z,v)$ be an infinitesimal datum in $\d$.  If $f'(z)v=0$ then $f(\zeta)=(f(z),0)$, which contacts $f(\d)$.  Consider the case that $f'(z)v\neq 0$.  Let 
\[
u_n=z+ \frac {1}{n}v, \qquad  v_n= z+\frac {1}{2n}v.
\]
For $n$ a sufficiently large integer, $u_n$ and $v_n$ are in $\d$, and we may define $s_n=f(u_n), \ t_n=f(v_n)$ in $U$. Clearly $s_n \to f(z), \ t_n \to f(z)$. Since
\[
s_n-t_n= \frac{1}{2n}f'(z)v+o(1/n)
\]
and $f'(z)v$ is nonzero, $s_n\neq t_n$ for sufficiently large $n$, and
\begin{align*}
\frac{s_n-t_n}{\|s_n-t_n\|} &= \frac{(2n)^{-1}f'(z)v + o(\tfrac 1n)}{\|(2n)^{-1}f'(z)v + o(\tfrac 1n)\|} \\
	&\to \frac{f'(z)v}{\|f'(z)v\|},
\end{align*}
which is collinear with $f'(z)v$.  Thus $f(\zeta)$ contacts $f(\d)$.

(2)  Let $k$ be a complex C-geodesic.  Then $k$ has a left inverse $C\in \d(U)$. Let $V=k(\d)$. We have
\[
k \circ C{|V}= {\id{U}}{|V}
\]
Since $\delta$ contacts $V$, by Remark \ref{F1=F2},
\[
(k \circ C)(\delta) = \id{U}(\delta) = \delta.
\]
Let $\zeta=C(\delta)$. Then $\zeta$ is a datum in $\d$ and, by equation \eqref{chainrule},  $k(\zeta)=\de$.
\end{proof}

\section{Uniqueness of  complex geodesics for each datum in $G$} \label{cgeosdtmsG}

An important fact in the function theory of $G$ is that Lempert's conclusion ($\kob{\cdot}=\car{\cdot}$) holds for $G$, even though $G$ does not satisfy any of the various convexity hypotheses of versions of Lempert's theorem:  $G$ is not convex (nor even biholomorphic to a convex domain  \cite{cos04}).
\begin{theorem}\label{lempprop}
For any datum $\de$ in $G$,  $\kob{\de}=\car{\de}$.
\end{theorem}
For discrete datums this is \cite[Corollary 5.7]{AY04}.  In fact, since $G$ is taut, it follows that equality also holds for infinitesimal datums \cite[Proposition 11.1.7]{jp}.  See also \cite[Theorem 7.1.16]{jp}.

The following property of $G$ plays a fundamental role in this paper.
\begin{theorem}\label{geothm10}
For every \nd datum $\delta$ in $G$ there exists a unique complex geodesic $\cald$ in $G$ such that $\delta$ contacts $\cald$. Moreover $\cald = k(\d)$ where $k$ solves $\Kob(\de)$.
\end{theorem}
\index{theorem!uniqueness of complex geodesics for each datum}
\begin{proof}
Let $\de$ be a \nd datum in $G$.  Let $k$ solve $\Kob(\de)$ and $C$ solve $\Car(\de)$.  Then there exists a datum $\zeta$ in $\d$ such that $k(\zeta)=\de$ and
\[
|\zeta|= \kob{\de}=\car{\de}=|C(\de)| >0.
\]
Thus $C\circ k \in \d(\d)$ and $|C\circ k(\zeta)|=|\zeta|$.  By the Schwarz-Pick Lemma, $C\circ k \in\aut\d$.  Let $\cald=k(\d)$.  Then $\cald$ is a complex geodesic in $G$ and $\de=k(\zeta)$ contacts $\cald$ as desired.

To prove uniqueness, let $\cald'$ be any complex geodesic of $G$ that is contacted by $\de$.  By Definition \ref{defcompgeo} there exists $h\in G(\d)$ and $F\in\d(G)$ such that $F\circ h=\id{\d}$ and $\cald'=h(\d)$.
Since $\de$ contacts $h(\d)$, by Proposition \ref{geoprop10}, there exists a datum $\zeta'$ in $\d$ such that $h(\zeta')=\de$.  Since $F\circ h=\id{\d}$ we have $|\zeta'|=|\de|=|\zeta|$, and so $h,k$ are both solutions of $\Kob(\de)$.  By Theorem \ref{extprop30}, $h=k\circ m$ for some $m\in\aut\d$.  Thus $\cald'=h(\d)=k(\d)=\cald$.
\end{proof}

As a consequence of Theorem \ref{geothm10} we may unambiguously attach to each datum in $G$ a unique complex geodesic.
\begin{definition}\label{d}
For any \nd datum $\delta$ in $G$ the unique complex geodesic in $G$ that is contacted by $\delta$ is denoted by  $\cald_\delta$.
\end{definition}
\index{$\cald_\de$}

\section{Flat C-geodesics}\label{flatcgeos}
The symmetrized bidisc has the remarkable property that it is foliated by the ranges of its complex C-geodesics of degree one.
For any $\beta\in\d$ we introduce the function $f_\beta:\d\to \c^2$ given by
\be\label{defFbeta}
f_\beta(z)= (\beta + \bar\beta z, z)
\ee
and the set 
\be\label{defcalfbeta}
\calf_\beta=f_\beta(\d)= \{(\beta + \bar\beta z, z):z\in\d\}.
\ee
\index{$f_\beta$}
\index{$\calf_\beta$}
\begin{proposition}\label{fsubbeta}
For any $\beta\in\d$ the map $f_\beta$ is a complex C-geodesic of $G$.
\end{proposition}
\begin{proof}
By Proposition \ref{elG}, $f_\beta$ maps $\d$ into $G$.  Thus $f_\beta\in G(\d)$, and clearly $f_\beta$ has a holomorphic left inverse, to wit the co-ordinate function $s^2$.  Hence $f_\beta$ is a complex C-geodesic of $G$.
\end{proof}
It follows of course that $\calf_\beta$ is a complex geodesic of $G$.

The following proposition implies that $G$ is foliated by the flat geodesics $\calf_\beta, \, \beta\in\d$, in $G$.  The proof is a simple calculation. 
\begin{proposition}\label{propbeta}
If $s\in G$ then there is a unique $\beta\in\d$ such that $s\in \calf_\beta$. Moreover $\beta$ is given by
\[
\beta=\frac{s^1-\bar{s^1}s^2}{1-|s^2|^2}.
\]
\end{proposition}
The fact that $\beta$ so defined lies in $\d$ is contained in Proposition \ref{elG}.
\begin{proposition}\label{deg1flat}
The following statements are equivalent for a complex {\rm C}-geodesic $k$ of $G$.
\begin{enumerate}[\rm (1)]
\item $\deg(k)=1$;
\item $k=f_\beta\circ m$ for some $\beta\in\d$ and $m\in\aut\d$;
\item  $k(\d)$ is the intersection of $G$ with a complex line;

\item $k(\d)$ is contained in the intersection of $G$ with a complex line.

\end{enumerate}
\end{proposition}
\begin{proof}
(1)$\Leftrightarrow$(2) The complex C-geodesics of $G$ are explicitly described in \cite{AY06,pz05}: they are $\Gamma$-inner functions of degree one or two, and those of degree one are precisely of the form in (2).\\
(2)$\Rightarrow$(3)$\Rightarrow$(4) is trivial.
Conversely, suppose that 
$k(\d)$ is contained in the intersection of a complex line with $G$.
 Let $\omega\in\t$ be such that $\Phi_\omega$ is a left inverse of $k$.  Then $m\df\Phi_\omega\circ k\in\aut\d$.  By Theorem \ref{descgeos} and Proposition \ref{cancels}, $k^2$ is a non-constant Blaschke product, $\ph$ say, of degree $1$ or $2$, and 
\be\label{kkk}
k^1=\overline {k^1}k^2 \quad \mbox{ on } \t.
\ee
Since $k(\d)$ lies in a complex line we can write $k=(a+b\ph,\ph)$ for some $a,b\in\c$.  Moreover the relation \eqref{kkk}
implies that $b=\bar a$.  Hence 
\be\label{kaph}
k=(a+\bar a\ph,\ph).
\ee
   Since $\ph$ is nonconstant it has a zero $z_0$  in $\d$. Since $k(z_0)=(a,0)$ and $k(z_0) \in G$, we have $a\in\d$.   Differentiate the identity
\[
(2\omega\ph-a-\bar a \ph)(z)=m(z)(2-\omega a-\omega \bar a\ph(z))
\]
and substitute $z=z_0$ to show that $ \ph'(z_0) \neq 0$ and thus that $z_0$ is a simple zero of $\ph$. 
Since $k$ has a holomorphic left inverse, $k$ is injective, and therefore $z_0$ is the only zero of $\ph$ in $\d$.
Hence $\ph$ has degree $1$ and so $\ph\in\aut\d$.  Equation \eqref{kaph} shows that statement (2) holds.  Thus (4) implies (2).
\end{proof}
In the light of Proposition \ref{deg1flat} we shall call any complex C-geodesic of degree one  a {\em flat } C-{\em geodesic} of $G$.  
\index{flat C-geodesic}
There is no conflict with the notion of flat datum, as introduced in Definition \ref{extdef10}, for the following reason.
\begin{proposition} \label{twoflats}
A \nd datum $\de$ in $G$ is a flat datum if and only if there is a flat {\rm C}-geodesic $k$ that solves $\Kob(\de)$.
\end{proposition}
\begin{proof}
Suppose that $\de$ is a flat datum.  Since $\rho_\de$ is constant on $\t$, Theorems \ref{flatandroyal} and \ref{flatroyalesimal} imply that $\de$ contacts either $\calr$ or $\calf_\beta$ for some $\beta\in\d$.  In the former case the function $C(s)=\half s^1$ solves $\Car(\de)$, which implies that the datum $\de$ is both royal and flat, contrary to the Pentachotomy Theorem.  Hence $\de$ contacts $\calf_\beta$, and the flat C-geodesic $k=f_\beta$ solves $\Kob(\de)$.

Conversely, suppose that $f_\beta\circ m$ solves $\Kob(\de)$ for some $\beta\in\d, \ m\in\aut\d$; then so does $f_\beta$.  Since $C\circ f_\beta=\id{\d}$, where $C(s)=s^2$, it follows that $C$ solves $\Car(\de)$. A simple calculation shows that, for any $\omega\in\t$, 
\[
\Phi_\omega\circ f_\beta=m_\omega \quad \mbox{ for some }m_\omega \in \aut\d.
\] 
Since $\delta$ contacts $f_\beta(\d)$, by Proposition \ref{geoprop10}, $\delta=f_\beta(\zeta)$ for some \nd datum $\zeta$ in $\d$.  We have, in view of the definition \eqref{defrhode},
\[
\rho_\de(\omega)=|\Phi_\omega(\delta)|^2= |\Phi_\omega\circ f_\beta(\zeta)|^2= |m_\omega(\zeta)|^2=|\zeta|^2.
\]
Hence $\rho_\delta$ is constant on $\t$. 
Thus $\delta$ is a flat datum.
\end{proof}

\section{Rational $\Ga$-inner functions}\label{RatGaInner}
Let us recall some notions and results from \cite{ALY15}.  The closure of $G$ in $\c^2$ will be denoted by $\Ga$; thus
\[
\Ga=\{(z+w,zw): z,w\in\d^-\}.
\]
\index{$\Ga$}
The distinguished boundary of $\Ga$ is denoted by $b\Ga$.
\index{$b\Ga$}
\index{distinguished boundary}
\index{function!rational $\Ga$-inner}
\begin{definition}\label{Gainner}
A {\em rational $\Gamma$-inner function} is a rational analytic map $h:\d\to \Gamma$ with the property that $h$ maps $\t$ into the distinguished boundary $b\Gamma$ of $\Gamma$.  
\end{definition}
\index{function!$\Ga$-inner}
All complex C-geodesics in $G$ are rational $\Ga$-inner functions (Theorem \ref{descgeos}). 

 For any rational $\Gamma$-inner function $h$, the second component $h^2$ is a finite Blaschke product.  The degree $\deg (h)$ is equal to the degree of $h^2$ (in the usual sense).   The following notion will play an important role.
\begin{definition}\label{defroynode}
Let $h$ be a rational $\Gamma$-inner function.  A {\em royal node} of $h$ is a point  $\la\in\d^-$ such that $h(\la)\in\royal^-$. A {\em royal point} of $h$ is a point of $h(\d^-)\cap \calr^-$. 
\end{definition}
\index{royal node}
\index{royal point}
The following result is \cite[Theorem 1.1]{ALY15}.  Observe that $\la\in\d^-$ is a royal node of $h$ if and only if $h^1(\la)^2=4h^2(\la)$.
\begin{theorem}\label{number_royal} 
If $h$ is a  nonconstant  rational $\Gamma$-inner  function then either $h(\d^-) = \royal^-$ or $h$ has exactly $\deg(h)$ royal nodes, counted with multiplicity.
\end{theorem}
Here multiplicity is defined as follows.
\begin{definition}\label{multiplicity}
Let $h$ be a rational $\Gamma$ inner function.  If $\sigma$ is a zero of $(h^1)^2-4h^2$ of order $\ell$, we define the \emph{multiplicity $\#\sigma$} of $\sigma$ (as a royal node of $h$) by 
\[
 \#\sigma \quad = \quad \twopartdef{\ell} {\mbox{ if }\sigma \in \d}{\half \ell} {\mbox{ if }\sigma \in \t.}
\]
\end{definition}
\index{royal node!multiplicity of}

\begin{lemma}\label{cancel_maxim}
Let $\delta$ be a \nd datum in $G$, let  $k$ solve $\Kob(\de)$ and suppose that $k^2$ is a Blaschke product of degree $2$. Let $\omega \in \t$.
The following conditions are equivalent:

{\rm (i)} $\omega$ is a maximiser of $\rho_\delta$ on $\t$;

{\rm (ii)} $(2 \bar\omega, \bar\omega^2)$ is in $k(\d)^- \cap \partial \calr$;

{\rm (iii)} $\Phi_{\omega}$ solves $\Car(\delta)$.
\end{lemma}
\begin{proof}
By the definition of  $\rho_\delta$, (i) $\Leftrightarrow$ (iii).

(ii) $\Rightarrow$  (iii).  Suppose (ii), and let $\tau\in\d^-$ be such that $k(\tau)=(2\bar\omega,\bar\omega^2)$.  By the Maximum Modulus Principle, $\tau\in\t$.  Consider the rational function
\be\label{seecancel}
\Phi_\omega\circ k(\la) = \frac{2\omega k^2(\la)-k^1(\la)}{2-\omega s(\la)}
\ee
which is inner, by Proposition \ref{phiok}.  By \cite[Corollary 6.10]{ALY13} $k^1, k^2$ are rational functions having the same denominator and therefore $\Phi_\omega\circ k$ has apparent degree $2$, the degree of $k$.  However, as $\la\to\tau$, both the numerator and the denominator of the right hand side of equation \eqref{seecancel} tend to zero.  Hence both numerator and denominator are divisible by $\la-\tau$, and so $\Phi_\omega\circ k$ is a rational inner function of degree at most one.

Suppose that $\Phi_\omega\circ k$ is constant.  Then $k(\d)$ is contained in a complex line, and so, by Proposition \ref{deg1flat}, $k$ has degree one, contrary to hypothesis.  Thus $\Phi_\omega\circ k$ is a rational inner function of degree exactly one, which is to say that $\Phi_\omega \circ k\in\aut\d$.

Since $k$ solves $\Kob(\de)$ there exists a datum $\zeta \in \d$ such that $k(\zeta) = \delta$ and $|\zeta| = \kob{\delta}$. By  the invariance of $|\cdot |$ under automorphisms,
\[
| \Phi_{\omega}\circ k (\zeta)| = |\zeta|.
\]
Therefore
\[
|\Phi_{\omega}(\delta)|= | \Phi_{\omega}\circ k (\zeta)| = |\zeta|
=  \kob{\delta}= \car{\delta}.
\]
Hence $\Phi_{\omega}$ solves $\Car(\delta)$.  Thus (ii) implies (iii).

(iii) $\Rightarrow$  (ii). Let $\Phi_{\omega}$ solve $\Car(\delta)$. Then $\Phi_{\omega}\circ k \in \aut\d$ has degree one. Thus 
\[
\Phi_{\omega}\circ k = \frac{2\omega k^2 -k^1}{2- \omega k^1}
\]
has a cancellation at some point $\tau\in\d^-$ (that is, $\la-\tau$ divides both numerator and denominator). By Theorem \ref{crit-canc=royal},  $\tau \in \t$,  $\tau$ is a royal node for $k$ and $k(\tau) = (2 \bar\omega, \bar\omega^2)$. Therefore $(2 \bar\omega, \bar\omega^2)$ is in $k(\d)^- \cap \partial \calr$.
\end{proof}
\begin{proposition}\label{georoyal}
Let $k$ be a rational $\Ga$-inner function of degree $2$.  Then $k$ is a complex C-geodesic of $G$ if and only if $k$ has a royal node in $\t$.
\end{proposition}
\begin{proof}
Let $k$ be a complex C-geodesic, so that $k$ has a holomorphic left inverse $C$.  Pick a \nd datum $\zeta$ in $\d$.  By Proposition \ref{Cksolve}, $k$ solves $\Kob(\de)$, where $\de=k(\zeta)$.  Let $\omega$ be a maximizer of $\rho_\de$ in $\t$.  By Lemma \ref{cancel_maxim}, there exists $z\in\t$ such that $k(z)=(2\bar\omega,\bar\omega^2)$.  Then $z$ is a royal node of $k$ in $\t$.

Conversely, suppose that $k$ has a royal node $z\in\t$, so that $k(z)=(2\bar\omega,\bar\omega^2)$ for some $\omega\in\t$.  By Lemma \ref{cancel_maxim}, $\Phi_\omega$ solves $\Car(\de)$.  Since $\kob{\cdot}=\car{\cdot}$, it follows that $\Phi_\omega$ is a holomorphic left inverse of $k$ modulo $\aut\d$.  Hence $k$ is a complex C-geodesic of $G$.
\end{proof}

\chapter{The retracts of $G$ and the bidisc $\d^2$}\label{retractsG}

As was observed in the introduction, there are inclusions between the collections of complex geodesics, holomorphic retracts and sets with the \npep\ in a general domain.  In this chapter we prove Theorem \ref{main2}, to the effect that the nontrivial retracts in $G$ are the complex geodesics.  By way of a digression, we show that a similar geometric argument can be used to establish the (known) description of the retracts in $\d^2$.  Furthermore, we show by an explicit construction that complex geodesics in $G$ are zero sets of polynomials of degree at most $2$.

\section{Retracts and geodesics of $G$}\label{retracts&geos}

Let $U$ be a domain. We say that $\rho$ is a \emph{retraction of $U$} if $\rho \in U(U)$ and $\rho \circ \rho = \rho$. \index{retraction} A set $R\subseteq U$ is a \emph{retract in $U$} if there exists a retraction $\rho$ of $U$ such that $R=\ran \rho$. \index{retract}
We say that a retraction $\rho$ and the corresponding retract $R$ are {\em trivial} if either $\rho$ is constant (and $R$ is a singleton set) or $\rho=\id{U}$ (and $R=U$).

A particularly simple way that a nontrivial retraction can arise is if $C\in \d(U)$, $k \in U(\d)$ and $C\circ k= \id{\d}$. In that event, if we set $\rho=k \circ C$, then $\rho \in U(U)$ and
\be\label{ret10}
\rho \circ \rho = (k \circ C) \circ (k \circ C) = k \circ (C \circ k) \circ C
=k \circ \id{\d} \circ C = k \circ C = \rho.
\ee
Therefore, if $U$ has dimension greater than one,  $\rho$ is a nontrivial retraction and  the complex geodesic $\ran k = \ran \rho$ is a nontrivial retract. Thus, on general domains in dimension greater than one, complex geodesics are nontrivial retracts. The question arises: for which domains is the converse of this statement true? See for example \cite{HS,knese}.  It {\em is} true for the symmetrized bidisc.
\begin{theorem}\label{retthm10}
A map $\rho \in G(G) $ is a nontrivial retraction of $G$ if and only if 
$\rho= k \circ C$  for some $C\in \d(G)$ and $k \in G(\d)$ such that $C\circ k=\id{\d}$.
 Thus  $R$ is a nontrivial retract in $G$ if and only if $R$ is a complex geodesic in $G$.
\end{theorem}
\index{theorem!retracts are complex geodesics in $G$}
\begin{proof}
Let $\rho$ be a retraction of $G$ that is neither constant nor $\id{G}$ and let $R=\ran \rho$.  
As $\rho$ is not constant, there exists a pair $s_1, s_2$ of distinct points in $\rho(G)$.  Let $\de$ be the 
\nd discrete datum $(s_1,s_2)$ in $G$.

Choose a datum $\zeta$ in $\d$ and $k\in G(\d)$ such that $k$ solves $\Kob(\de)$ and $k(\zeta)=\de$.
Since $\rho|R$ is the identity, $\rho(\de)=\de$, and so both $k$ and $\rho\circ k$ solve $\Kob(\de)$.  Moreover $k(\zeta)=\de=\rho\circ k(\zeta)$.  By Remark \ref{essun}, $\rho\circ k =k$.
Therefore 
\[
\mathcal{D}_\de =k(\d) = \rho\circ k(\d)\subseteq R.
\]

If $\mathcal{D}_\de \neq R$ then there exists $t\in R\setminus \mathcal{D}_\de$.  By the argument in the foregoing paragraph, for any $s\in \mathcal{D}_\de$, the complex geodesic through $t$ and $s$ is contained in $R$, that is, $\mathcal{D}_{(t,s)} \subseteq R$.
In particular, as $s_1\neq s_2$, the geodesics $\cald_1\df\mathcal{D}_{(t,s_1)}$ and $\cald_2\df\mathcal{D}_{(t,s_2)}$ are contained in $R$.  Furthermore, Theorem \ref{geothm10} implies that $\cald_1\neq \cald_2$.  

Choose $k_j$ to solve $\Kob(t,s_j)$ and a vector $v_j$ tangent to $\cald_j, \, j=1,2$ at $t$.   Since $\cald_1\neq \cald_2$, it follows from Theorem \ref{geothm10} that $v_1, v_2$ are not collinear.  Since $\rho$ acts as the identity map on $R$,
\[
D\rho(t)v_j=v_j, \qquad j=1,2.
\]
Hence $D\rho(t)$ is an invertible linear transformation.  By the Inverse Function Theorem $\rho(G)$ contains a neighborhood of $t$ in $\c^2$.  Since $\rho$ is the identity on its range, $\rho$ agrees with $\id{G}$ on a nonempty open set.  Thus $\rho=\id{G}$, contrary to our assumption.  Hence $R=\mathcal{D}_\de$.
\end{proof}

\section{Retracts of $\d^2$}\label{caseD2}

The argument used in the preceding section to characterize retractions in $G$ applies in any domain $U$ in $\c^2$ with the property that, for every \nd datum $\de$ in $U$, the solution to $\Kob(\de)$ is essentially unique.  The bidisc does not have this property, but  nevertheless,  a variant of the argument is effective. We shall exploit the fact that, for datums $\la$ in $\d^2$, the solution of $\Kob(\la)$ is not essentially unique precisely when the solution of $\Car(\la)$ {\em is} essentially unique.  Thereby we shall obtain known results on retractions of $\d^2$ (see \cite{HS,guo,knese}) by simple geometric methods. 

Incidentally, the implications of uniqueness and nonuniqueness of solutions of Carath\'eodory and Kobayashi problems in many domains (including $G$) are explored in \cite{kz2013}.

\begin{definition}\label{def_bal_datum_D2}
Let $\la=(\la_1,\la_2)$ be a \nd discrete datum  in $\d^2$. Then $\la$ is 
 {\em balanced},  {\em of type $1$} or {\em of type $2$} according as
\[
|(\la_1^1,\la_2^1)| =|(\la_1^2,\la_2^2)|, \quad
|(\la_1^1,\la_2^1)|  > |(\la_1^2,\la_2^2)| \; \mbox{ or }
|(\la_1^1,\la_2^1)|  < |(\la_1^2,\la_2^2)|
\]
respectively.
$\la$ is {\em unbalanced} if $\la$ is of type $1$ or type $2$.
\end{definition}

If $\la $ is a balanced datum in $\d^2$ then 
the unique geodesic in $\d^2$ that contacts $\la$ is a balanced disc in the sense of Definition \ref{defbalD2}.
It is not hard to see that, for the \nd datum $\la=(\la_1,\la_2)$ in $\d^2$, the Kobayashi extremal problem $\Kob(\la)$ has an essentially unique solution if and only if $\la$ is balanced, whereas the Carath\'eodory extremal problem $\Car(\la)$ has an essentially unique solution if and only if $\la$ is unbalanced. 
\begin{theorem}\label{retractD2}
A subset $R$ of $\d^2$ is a nontrivial retract in $\d^2$ if and only if
$R$ is a complex geodesic in $\d^2$.
\end{theorem}
\begin{proof} Since $R$ is assumed to be nontrivial, in particular $R$ is not a singleton. Choose points $\lambda_1,\lambda_2 \in R$ with $\lambda_1
\not=\lambda_2$.

If $\lambda =(\lambda_1,\lambda_2)$ is balanced, then by \cite[page 293]{agmc_vn} the solution to $\Kob(\lambda)$ is essentially unique. It then follows by the argument in the preceding section that $R = k(\d)$ where $k$ solves $\Kob (\lambda)$.

Now assume that $\lambda$ is not balanced, say
\be\label{t1}
|(\la_1^1,\la_2^1)|  > |(\la_1^2,\la_2^2)|.
\ee
Let $\rho$ be a retraction of $\d^2$ with $\rho (\d^2) = R$.
Let $f_1, f_2:\d^2 \to \d$ be the first and second co-ordinate functions.   Since $\{f_1,f_2\}$ constitutes a universal set for the Carath\'eodory problem on $\d^2$, the inequality \eqref{t1} shows that $f_1$ solves $\Car(\la)$.  Since $\la$ contacts $R$ and $\rho|R$ is the identity map, it is clear that $f_1\circ \rho$ also solves $\Car(\la)$.  By the essential uniqueness of solutions of $\Car(\la)$, there exists $m\in\aut\d$ such that $f_1\circ\rho=m\circ f_1$.  Apply this equation to the points $\la_1, \la_2\in R$ to obtain the equations
\[
\la_1^1= m(\la^1_1) \quad \mbox{ and }\quad \la_2^1 = m(\la_2^1).
\]
The relation \eqref{t1} shows that $\la_1^1\neq \la_2^1$, and thus $m$ fixes a pair of distinct points in $\d$.  Hence $m=\id{\d}$, and so
\be\label{f1rho}
f_1\circ \rho = f_1
\ee
on $\d^2$.  Thus $\rho$ has the form
\[
\rho(\lambda) = (\lambda^1, \phi(\lambda))
\]
for some $\phi \in \d(\d^2)$. 
Consequently, if for each fixed $\alpha \in \d$ we define $\phi_\alpha \in \d(\d)$ by the formula,
\[
\phi_\alpha(w) = \phi(\alpha,w),\qquad w \in \d.
\]
For any $z\in\d^2$, the equation $\rho\circ\rho=\rho$ yields the relation
\[
\rho(z^1, \phi(z))=(z^1, \phi(z)).
\]
Hence 
\be\label{rho2z}
\phi(z)=\phi(z^1,\phi(z)), \qquad z \in \d^2.
\ee
We claim that $\phi_\al$ is a retraction of $\d$.  Indeed, for $w \in \d$,
\begin{align*}
\ph_\al(\ph_\al(w)) &= \phi(\al,\phi(\al,w)) \\
	&= \phi(\al,w) \qquad \mbox{ (by equation \eqref{rho2z})} \\
	&= \ph_\al(w).
\end{align*}
This proves that $\phi_\alpha$ is a retraction of $\d$ for all $\alpha \in \d$.

Now retractions of $\d$ are either constant maps or $\id{\d}$. If there exists $\alpha_0 \in \d$ such that $\phi_{\alpha_0} = \id{\d}$, then by continuity, $\phi_\alpha = \id{\d}$ for all $\alpha$ in a neighborhood of $\alpha_0$. But then $\rho = \id{\d}$, which contradicts the assumption that $R$ is nontrivial. Therefore, $\phi_\alpha$ is constant for all $\alpha \in \d$.

For each $\alpha \in \d$, choose $g(\alpha) \in \d$ where $\phi_\alpha(w) = g(\alpha)$ for all $w \in \d$. Then
\[
\rho(\lambda) = (\lambda_1, \phi(\lambda)) = (\lambda_1, \phi_{\lambda_1}(\lambda_2))=(\lambda_1,g(\lambda_1))
\]
for all $\lambda_1 \in \d$. This proves that if we define  $k\in \d^2(\d)$ by the formula $k(w) = (w,g(w))$, $w \in \d$, then
\[
R=\ran \rho =\ran k.
\]
Since in addition, $f_1 \circ k = \id{\d}$, it follows that $R$ is a geodesic.

If $\lambda$ is not balanced and $|(\lambda_1^2,\lambda_2^2)| >|(\lambda_1^1,\lambda_2^1)|$, then the reasoning as above, with the roles of the coordinates reversed, again yields the conclusion that  $R$ is a geodesic.
\end{proof}

\section{Geodesics in $G$ are varieties}\label{geovar}

The relation in equation \eqref{ret10} has another valuable application, namely to the calculation of a quadratic polynomial that vanishes on a prescribed geodesic of $G$.  

\begin{theorem}\label{geosvariet}
For every complex geodesic $\cald$ in $G$ there exists a polynomial $P$ of total degree at most $2$ such that $\cald=\{s\in G:P(s)=0\}$.
\end{theorem}
\index{theorem!complex geodesics are varieties}
If $\cald$ is flat then we may choose $P(s)=s^1-\beta - \bar\beta s^2$ for some $\beta\in\d$.
The result therefore follows from the following more detailed statement.

{\em Warning: in the next proposition superscript $2$ is used for the square of a number and for the second component in the same setting. For example, $(s^2)^2$ is the square of the the second component
of $s = (s^1, s^2)$.}

\begin{proposition}\label{retprop10}
Let $\delta$ be a nonflat \nd datum in $G$, let $\omega\in\t$ be such that $\Phi_\omega$ solves $\Car(\delta)$ and let $k\in G(\d)$ satisfy  $\Phi_\omega \circ k = \id{\d}$.    Then $k^2$ is expressible in the form
\be\label{ret30}
k^2=  \frac{q^\sim}{q}
\ee
for some polynomial $q$  of degree $2$ that does not vanish on $\d^-$, where 
$q^\sim(z) =z^{2} \overline{q(1/\bar{z})}$.

Let
\be\label{defP}
P(s) =(2-\omega s^1)^2 \; \frac{ \tilde q\circ \Phi_\omega(s) - s^2 q\circ\Phi_\omega(s)}{\omega^2s^2-1}.
\ee
Then 
\begin{enumerate}[\rm (1)]
\item $P$ is a polynomial of total 
degree $2$ that vanishes on $\cald_\de$;
\item $P$ is irreducible;
\item $P(s)=  (s^2)^2\overline{P(\overline{s^1}/\overline{s^2}, 1/\overline{s^2})}$ and
\item $\cald_\delta = \set{s\in G}{P (s)=0}$.
\end{enumerate}
\end{proposition}
\begin{proof}
 Proposition \ref{cancels} in the Appendix contains the statement that $k^2$ is expressible in the form \eqref{ret30} for some polynomial $q$  of degree at most $2$ that does not vanish on $\d^-$. Since $\delta$ is not flat, $\deg k^2 =2$, which is to say that  $\deg(q^\sim)=2$. 

By equation \eqref{ret10}, $k \circ \Phi_\omega$ is a retraction onto $\ran k = \cald_\delta$. Therefore, for $s \in \cald_\delta$,
\[
(k \circ \Phi_\omega) (s) = s.
\]
Writing $k =(k^1, k^2)$ and $s=(s^1,s^2)$, we obtain the equations

\be\label{ret20}
k^1\circ\Phi_\omega(s) = s^1\qquad \text{ and }\qquad k^2\circ\Phi_\omega(s) =s^2
\ee
 for all $s\in \cald_\delta$.

By Proposition \ref{cancels}, $ k^2(\bar{\omega}) = \bar{\omega}^2$, and
therefore

\be\label{ret50}
q^\sim(\bar\omega) =\bar\omega^2 q(\bar\omega).
\ee
Substitute $z=\Phi_\omega(s)$ into equation \eqref{ret30} and use the second equation in \eqref{ret20} to obtain
\be\label{ret60}
q^\sim\circ\Phi_\omega(s) =s^2 q\circ\Phi_\omega(s)
\ee
for all $s\in \cald_\delta$. By virtue of  equation \eqref{ext10}, if  $Q^\sim$ and $Q$ are defined by the formulas
\be\label{ret70}
Q^\sim(s)=(2-\omega s^1)^2 q^\sim\circ\Phi_\omega(s)\qquad
\text{ and }\qquad Q(s) = (2-\omega s^1)^2q\circ\Phi_\omega(s),
\ee
then $Q^\sim$ and $Q$ are quadratic polynomials in two variables, and  equation \eqref{ret60} becomes
\be\label{ret80}
Q^\sim(s) =s^2 Q(s) \qquad \mbox{ for all } s \in \mathcal{D}_\de.
\ee
Also, note that if $s=(s^1,\bar\omega^2)$ and $s^1\not=2\bar\omega$, then
\[
\Phi_\omega(s) = \frac{2\omega \bar\omega^2-s^1}{2-\omega s^1}=\bar\omega,
\]
so that equations  \eqref{ret70} and \eqref{ret50} imply that
\be\label{ret90}
Q^\sim(s) =(2-\omega s^1)^2q^\sim(\bar\omega)=(2-\omega s^1)^2 \bar\omega^2 q(\bar\omega)=\bar\omega^2Q(s).
\ee
Equations \eqref{ret80} and \eqref{ret90} together imply that
\be\label{ret100}
Q^\sim(s) -s^2Q(s)
\ee
is a polynomial that vanishes on both $\cald_\delta$ and $\{s:s^2=\bar\omega^2\}$. It follows that
\be\label{ret110}
P(s) = \frac{Q^\sim(s) -s^2Q(s)}{\omega^2s^2-1}
\ee
is a well-defined polynomial that vanishes on $\cald_\delta$.   
From equation \eqref{ret110} one can derive an explicit formula for $P$ in terms of $q$:
\begin{align}\label{explicP}
P(s)&=  -q(\bar\omega) (s^1)^2+2q'(\bar\omega) s^1s^2-2q''(0) (s^2)^2 \notag  \\
	&\hspace*{1cm}  +2 \overline{q'(\bar\omega)} s^1 +4(2\re q(0) - q(\bar\omega))  s^2 -2 \overline{q''(0)}.
\end{align}
Note that, by equation \eqref{ret50},   $q(\bar\omega)$ is real.

 $P$ has total degree $2$, and so (1) holds.\\

\noindent (2) Suppose $P$ is not irreducible.  Since $P$ has degree 2, the zero set of $P$ is the union of two complex lines. If $\cald_\de$ is contained in only one of the complex lines then, by Proposition \ref{deg1flat}, $\de$ is a flat datum, contrary to hypothesis.  Hence both lines meet $\cald_\de$.  Since $\cald_\de$ is connected, $\cald_\de$ contains the common point, $k(z_0)$ say, of the two lines.    Thus both lines have the complex tangent $k'(z_0)$ at $k(z_0)$, and so the two lines coincide.  Again $\cald_\de$ is contained in a complex line, and so Proposition \ref{deg1flat} applies again to show that $\de$ is a flat datum, contrary to hypothesis.  Thus $P$ is irreducible.\\

\noindent (3)
Note that, for any $s\in G$,
\[
\Phi_\omega\left(\frac{\overline{s^1}}{\overline{s^2}}, \frac{1}{\overline{s^2}}\right)= 1/ \overline{\Phi_\omega(s)}.
\]
Thus 
\begin{align*}
P(\overline{s^1}/\overline{s^2}, 1/\overline{s^2}) & =
(2- \omega \overline{s^1}/\overline{s^2})^2 \; 
\frac{ \tilde q\circ \Phi_\omega\left(\frac{\overline{s^1}}{\overline{s^2}}, \frac{1}{\overline{s^2}}\right)
 - 1/\overline{s^2} \; q\circ
\Phi_\omega\left(\frac{\overline{s^1}}{\overline{s^2}}, \frac{1}{\overline{s^2}}\right)}{\omega^2/\overline{s^2}-1}\\
 & =(2- \omega \overline{s^1}/\overline{s^2})^2 \; 
\frac{\overline{s^2} \;\tilde q (1/ \overline{\Phi_\omega(s)}) - q (1/ \overline{\Phi_\omega(s)})}{\omega^2 -  \overline{s^2}}.
\end{align*}
By definition,
$q^\sim(z) =z^{2} \overline{q(1/\bar{z})}$, hence
\[
\overline{\tilde q (1/ \overline{\Phi_\omega(s)})} = \frac{  q\circ
\Phi_\omega(s)}{\Phi_\omega(s)^2}
\]
and 
\[
\overline{ q (1/ \overline{\Phi_\omega(s)})} = \frac{ \tilde q\circ
\Phi_\omega(s)}{\Phi_\omega(s)^2}.
\]
Therefore
\begin{align*}
(s^2)^2\overline{P(\overline{s^1}/\overline{s^2}, 1/\overline{s^2})} &=
\frac{(2\omega s^2 -  s^1)^2}{1 - \omega^2 s^2} \left( s^2 \; 
\overline{\tilde q (1/ \overline{\Phi_\omega(s)})} - \overline{ q (1/ \overline{\Phi_\omega(s)})} \right)\\
&= \frac{(2\omega s^2 - s^1)^2}{1 - \omega^2 s^2}
 \left( s^2 \; \frac{  q\circ
\Phi_\omega(s)}{\Phi_\omega(s)^2} - \frac{ \tilde q\circ
\Phi_\omega(s)}{\Phi_\omega(s)^2} \right)\\
&=  P(s).
\end{align*}
Statement (3) holds.\\

\noindent (4)  From statement (1), $\cald_\de\subset P^{-1}(\{0\})$.
It is easy to see from the definition \eqref{defP} of $P$ that, for every $s \in G$,
\begin{align*}
P(s) & =(2-\omega s^1)^2 \; \frac{ \tilde q\circ \Phi_\omega(s) - s^2 q\circ\Phi_\omega(s)}{\omega^2s^2-1}\\
& =\frac{(2-\omega s^1)^2}{\omega^2 s^2-1} \; \left[
\frac{\tilde q\circ \Phi_\omega(s)}{q\circ\Phi_\omega(s)} - s^2 \right]\; q\circ\Phi_\omega(s) \\
& =\frac{(2-\omega s^1)^2}{\omega^2 s^2-1} \; q\circ\Phi_\omega(s) \;
\left[ k^2 \circ\Phi_\omega(s) - s^2 \right].
\end{align*}
Recall that the polynomial $q$  does not vanish on $\d^-$, and 
$\frac{(2-\omega s^1)^2}{\omega s^2-1} \neq 0$ for all $s \in G$.

By assumption,  $k\in G(\d)$ satisfies  $\Phi_\omega \circ k = \id{\d}$.
We claim that,
for each $s \in G$, the equation $k^2 \circ\Phi_\omega(s) =s^2$ implies that $k^1 \circ\Phi_\omega(s) =s^1$. For every $z \in \d$,
\[
\frac{2 \omega k^2(z) - k^1(z)}{2-\omega k^2(z)}=z.
\]
Hence 
\[
\frac{2 \omega k^2(\Phi_\omega(s)) - k^1(\Phi_\omega(s))}{2-\omega k^2(\Phi_\omega(s))}=\Phi_\omega(s)
\]
and, since $k^2 \circ\Phi_\omega(s) =s^2$, we have
\[
\frac{2 \omega s^2 - k^1(\Phi_\omega(s))}{2-\omega s^2}=\Phi_\omega(s).
\]
This implies that
\[
2 \omega s^2 - k^1(\Phi_\omega(s))= 2\omega s^2 -s^1,\; \text{and so}\;\;
k^1 \circ\Phi_\omega(s) =s^1.
\]
Therefore 
\begin{align*}
 \set{s\in G}{P (s)=0}  &= \set{s\in G}{k^2 \circ\Phi_\omega(s) =s^2}\\
	 &= \set{s\in G}{k^2 \circ\Phi_\omega(s) =s^2\; \text{and}\; k^1 \circ\Phi_\omega(s) =s^1}.
\end{align*}
Thus, by the equations \eqref{ret20}, $\cald_\delta = \set{s\in G}{P (s)=0}$.
\end{proof}

\begin{lemma}\label{retlem10}
Let $\delta$ be a balanced datum in $G$, let $\Phi_{\omega_1}$ and $\Phi_{\omega_2}$ solve $\Car(\delta)$ with $\omega_1\not=\omega_2$, and let $k$ solve $\Kob(\delta)$. Choose $m_1, \, m_2\in\aut\d$ such that
\[
m_1 \circ \Phi_{\omega_1} \circ k = \id{\d}\qquad  \text{ and }\qquad  m_2 \circ  \Phi_{\omega_2} \circ k = \id{\d},
\]
and then let $C_1$ and $C_2$ be defined by
\[
C_1=m_1 \circ \Phi_{\omega_1}\qquad \text{ and }\qquad C_2=m_2 \circ \Phi_{\omega_2}.
\]
There exists a nonzero constant $c$ such that
\[
C_1(s)-C_2(s) =\frac{cP_\delta(s)}{(2-\omega_1s^1)(2-\omega_2s^1)}
\]
for all $s\in G$,  where $P_\delta$ is defined by equation \eqref{defP}.
\end{lemma}
\begin{proof}
Simply observe that
\[
(2-\omega_1s^1)(2-\omega_2s^1)(C_1(s)-C_2(s))
\]
is a quadratic polynomial that vanishes on $\ran k$ and that
\[
\ran k = \cald_\delta = \set{s\in G}{P_\delta (s)=0}.
\]
\end{proof}

\chapter{Purely unbalanced and exceptional datums in $G$} \label{delicate}

The most delicate aspect of the geometric characterization of types of datum in Theorem \ref{geothm30} below is the distinction between purely unbalanced and exceptional datums, depending as it does on whether or not a certain second derivative is zero.  In this chapter we calculate the relevant second derivatives.

Let us recall the following  elegant parametrization of solutions to $\Kob(\delta)$ due to Pflug and Zwonek \cite[Theorem 2 part (ii)]{pz05}.
\begin{proposition}\label{prop3.10}
Let $\delta$ be a \nd datum in $G$ and assume that $\cald_\delta \cap \calr = \emptyset$. If $k$ solves $\Kob(\delta)$ then there exist  $m_1, \ m_2\in\aut\d$  such that
\be\label{kmm}
k = (m_1+m_2,m_1m_2).
\ee
\end{proposition}

\begin{proposition}\label{secderiv}
Let $\de$ be a \nd datum in $G$ and suppose that $ \cald_\de^- \cap \calr^-$ is the singleton set $\{(2\bar\omega_0, \bar\omega_0^2)\}$ for some $\omega_0\in\t$.  Then
\be\label{goal}
\frac{d^2}{dt^2} \rho_\de(e^{it}) \big|_{e^{it}=\omega_0} = 0.
\ee
\end{proposition}
\begin{proof}
Let $k$ solve $\Kob(\de)$, so that there is a datum $\zeta$ in $\d$ such that 
$k(\zeta)=\de$ and $|\zeta|= \kob{\de}$.    Then $k(\d)=\cald_\de, \; k(\d^-)=\cald_\de^-$ and $\Phi_{\omega_0}$ solves $\Car(\de)$ by Lemma \ref{cancel_maxim}.

Since $k$ has no royal nodes in $\d$, by Proposition \ref{prop3.10}, $k=(m_1+m_2,m_1m_2)$ for some $m_1,m_2\in\aut \d$.  Replace $k$ by $k\circ m_1^{-1}$: then we may write $k(z)=(z+m(z),zm(z))$, where $m=cB_\al$
for some $c\in\t$ and $\al\in\d$.    Consider any royal node $\tau\in\d^-$  of $k$.   The royal nodes of $k$ are the fixed points of $m$ in $\d^-$, and so, since $k(\tau) \in E$, 
\[
k(\tau)=(2\tau,\tau^2)=(2\bar\omega_0,\bar\omega_0^2).
\]
Thus $\tau=\bar\omega_0$.  Hence $k$ has the unique royal node $\bar\omega_0$, and so $m$ has the unique fixed point $\bar\omega_0$, and this fixed point lies in $\t$.

It is elementary that, for the function $m=cB_\al$,
\begin{enumerate}[\rm (1)]
\item $m$ has a unique fixed point which lies in $\t$ if and only if $c\neq 1$ and $|\al| = \half|1-c|$, and
\item  $m$ has two fixed points which lie in $\t$ if and only if $|\al| > \half |1-c|$.
\end{enumerate}
For our chosen $m$, alternative (1) holds, and so $|\al| = \half|1-c|>0$.

For $1\leq r< |\al|^{-1}$ let $m_r= cB_{r\al}$.
Then $m_r\in\aut\d$ and $m_1=m$.  For $r>1$, since $|r\al| > \half|1-c|$, \; $m_r$ has two fixed points lying in $\t$.  These fixed points are easily seen to be
\[
\tau_\pm(r) =  \frac {1-c \pm \sqrt{(1-c)^2+4r^2|\al|^2 c}}{2r\bar\al}.
\]
When $r=1$ the discriminant in this formula is zero, since $m_1$ has a unique fixed point, and so $(1-c)^2=-4|\al|^2c$.  Hence
\be\label{formtaur}
\tau_\pm(r) =\frac{1-c}{2r\bar\al}(1 \pm i\sqrt{r^2-1})\quad \in\t.
\ee
These two fixed points are distinct when $r>1$ and coincide with $\tau$ when $r=1$.

Let
\be\label{kr}
k_r(z)=(z+m_r(z),zm_r(z))
\ee
for $z\in\d, \, 1\leq r<|\al|^{-1}$.  Then $k_r$ is a rational $\Gamma$-inner function of degree $2$ having royal nodes $\tau_\pm(r) \in\t$.
Since $k_r^1$ and $k_r^2$ have the same denominator, $\Phi_\omega\circ k_r$ is also rational of degree at most $2$, for any $\omega$.  Moreover, $\Phi_\omega\circ k_r$ is inner, by Proposition \ref{phiok} in the Appendix.  When $\omega= \overline{\tau_\pm(r)}$, the numerator and denominator of $\Phi_\omega\circ k_r(z)$ both vanish at $z=\tau_\pm(r)$.  Thus cancellation occurs, and so $\Phi_{\overline{\tau_\pm(r)}}\circ k_r$ has degree $1$.  Thus $k_r$ has the left inverse $\Phi_{\overline{\tau_\pm(r)}}$ modulo $\aut\d$, which is to say that $k_r$ is a complex C-geodesic of $G$.

Let $\de_r=k_r(\zeta), \; 1\leq r< |\al|^{-1}$.  Then $k_r$ solves $\Kob(\de_r)$ and $\Phi_{\overline{\tau_\pm(r)}}$ solves $\Car(\de_r)$.  Hence the point $\overline{\tau_\pm(r)}$ is a maximizer for $\rho_{\de_r}$, and consequently
\be\label{maxrhodelr}
\frac{d}{dt} \rho_{\de_r}(e^{it})\big|_{e^{it}=\overline{\tau_\pm(r)}} =0.
\ee

For $1\leq r<|\al|^{-1}$ define $F_r:\t\to\r$ by
\begin{align*}
F_r(\omega)&= \frac{d}{dt} \rho_{\de_r}(e^{it})\big|_{e^{it}=\omega} \\
	&=   \frac{d}{dt} |\Phi_{e^{it}}(\de_r)|^2\big|_{e^{it}=\omega}.
\end{align*}
By equation \eqref{maxrhodelr},
\be\label{zerosFr}
F_r(\overline{\tau_\pm(r)})=0.
\ee

From the formula \eqref{defPhi} for $\Phi_\omega$ it is easy to see that 
\[
 |\Phi_\omega(\de_r)| = |g_r(\omega)| \qquad \mbox{ for } \omega \in\t
\]
for some fractional quadratic function $g_r$ in which the coefficients depend continuously (indeed, polynomially) on $r$.  It follows that $g_r \to g_1$ in $C^\infty(\t)$ as $r\to 1$.  Now
\begin{align*}
F_r(\omega)&= \frac{d}{dt} |g_r(e^{it})|^2\big|_{e^{it}=\omega} \\
	&=   -2\im\left( \omega g_r'(\omega)\overline{g_r(\omega)}\right).
\end{align*}
Consequently $F_r \to F_1$ in $C^\infty(\t)$ as $r\to 1$.

Let
\[
G_r(t)= F_r(e^{it})
\]
for $t\in\r$ and write
\[
\tau= e^{it_0}, \qquad \tau_\pm(r)= e^{i{t_\pm(r)}},
\]
where $t_\pm(r)$ is close to $t_0$ for $r$ close to $1$.
Our goal \eqref{goal} is to show that $G_1'(-t_0)=0$.   We have, for any $t\in\r$,
\[
G_r(t)=G_r(-t_0)+(t+t_0)G_r'(-t_0)+\half(t+t_0)^2 G_r''(\theta)
\]
for some $\theta=\theta(r,t)\in\r$.   By equation \eqref{zerosFr}, $G_r(-t_\pm(r))=0$.  Therefore
\begin{align*}
0&= G_r(-t_+(r))-G_r(-t_-(r)) \\
	&=(t_-(r)-t_+(r))G_r'(-t_0)+A_+(r)(t_0-t_+(r))^2 + A_-(r)(t_0-t_-(r))^2
\end{align*}
for some real numbers $A_\pm(r)$ which are uniformly bounded for all $r,t$, say by $K>0$.  Thus, for $r>1$,
\[
|G_r'(-t_0)| \leq K \frac{(t_0-t_+(r))^2 +(t_0-t_-(r))^2}{|t_-(r)-t_+(r)|}
\]
Now $|t_0-t_\pm(r)|$ and $| t_+(r)-t_-(r)|$ are asymptotic to $|\tau-\tau_\pm(r)|$ and $|\tau_+(r)-\tau_-(r)|$ respectively as $r\to 1$,
and from equation \eqref{formtaur} it is easy to see that all these quantities are asymptotic to $\sqrt{r-1}$.  It follows that $G_r'(-t_0) \to 0$ as $r\to 1$.    Since $G_r \to G_1$ with respect to the $C^1$ norm, it follows that $G_1'(-t_0)=0$, as required.
\end{proof}
We shall need the following calculation.
\begin{lemma}\label{Phioconst}
If a datum $\de$ in $G$ contacts $\calr$ then $\Phi_\omega(\de)$ is independent of $\omega\in\t$.
\end{lemma}
\begin{proof}
If $\de$ is discrete it has the form $(s_1,s_2)$ where $s_j=(2z_j,z_j^2)$ for some $z_j\in\d, \, j=1,2$.  One then finds that $\Phi_\omega(\de)=(-z_1,-z_2)$ for all $\omega\in\t$.  If $\de$ is infinitesimal then $\de=(s_1,v)$ where $s_1=(2z_1,z_1^2)$ for some $z_1\in\d$ and $v=\la(1,z_1)$ for some $\la\neq 0$.  A straightforward calculation (more details are in the proof of Theorem \ref{flatroyalesimal}) yields
\[
\Phi_\omega(\de)=(-z_1, -\half\la).
\]
In either case, $\Phi_\omega(\de)$ does not depend on $\omega$.
\end{proof}
For every $m\in\aut\d$ we define a map $\tilde m\in G(G)$ by 
\be\label{defmtilde-main}
\tilde m(z+w,zw)= (m(z)+m(w), m(z)m(w)) \qquad \mbox{ for all } z,w\in\d.
\ee
\index{$\tilde m$}
Every automorphism of $G$ is of the form $\tilde m$ for some $m\in\aut\d$, by Theorem \ref{autosG}.

We use the notation $f^\vee(z)= \overline{f(\bar z)}$ for any function $f$ on $\d$.
\index{$f^\vee$}
\begin{lemma}\label{mcircPhio}
Let $m=cB_\al \in\aut\d$ for some $c\in\t$ and $\al\in\d$.  For any $\omega\in\t$
\be\label{formmcircPhi}
\Phi_\omega \circ \tilde m = m_1\circ \Phi_{(m^\vee)^{-1}(\omega)}
\ee
where $m_1=cB_{-\al}$.
Consequently, for any \nd datum $\de$ in $G$ and any $m\in\aut\d$,
\be\label{formrhomde}
\rho_{\tilde m(\de)}=\rho_\de\circ (m^\vee)^{-1}.
\ee
\end{lemma}
If $m(z)=cB_\al(z)$ then $(m^\vee)^{-1}(\omega)=B_{-\bar\al}(c\omega)$.
\begin{proof}
The identity \eqref{formmcircPhi} is a matter of verification.  If $\omega_1=(m^\vee)^{-1}(\omega)$ then 
\begin{align*}
\rho_{\tilde m(\de)}(\omega)&= |\Phi_\omega \circ \tilde m(\de)|^2= |m_1\circ \Phi_{\omega_1}(\de)|^2=|\Phi_{\omega_1}(\de)|^2=\rho_\de(\omega_1)=\rho_\de\circ (m^\vee)^{-1}(\omega).
\end{align*}
\end{proof}
\begin{lemma}\label{rhomtild}
Let $\de$ be a \nd datum in $G$ and let $m=cB_\al\in\aut\d$ for some $c\in\t$ and $\al\in\d$.
  If $\rho_\de$ attains its maximum over $\t$ at $\omega_0$ then $\rho_{\tilde m(\de)}$ attains its maximum at $m^\vee(\omega_0)$ and 
\be\label{sectild}
\frac{d^2}{dt^2} \rho_{\tilde m(\de)}(e^{it})\big|_{e^{it}=m^\vee(\omega_0)}= \frac{|1-\al\omega_0|^4}{(1-|\al|^2)^2} \frac{d^2}{dt^2}\rho_\de(e^{it})\big|_{e^{it}=\omega_0}.
\ee
\end{lemma}
\begin{proof}
It is immediate from equation \eqref{formrhomde} that $\rho_{\tilde m(\de)}$ attains its maximum at $m^\vee(\omega_0)$.

Let $\omega_0=e^{it_0}$ and define
\[
r_\de(t)=\rho_\de(e^{it})
\]
for $t\in\r$.  Since $r_\de$ attains its maximum over $\r$ at $t_0$, we have
\be\label{critpt}
r_\de'(t_0)=0.
\ee

Let $\tau(\cdot)$ be defined as a smooth function of $t$  in a neighborhood of $t_0$ by
\[
e^{i\tau(t)}=(m^\vee)^{-1}(e^{it}) = \frac{ce^{it}+\bar\al}{1+c\al e^{it}}.
\]
Then
\be\label{tauprime}
\tau'(t)=\frac{1-|\al|^2}{|1+ c\al e^{it}|^2}.
\ee
From equation \eqref{formrhomde} we have
\begin{align*}
r_{\tilde m(\de)}(t)&=\rho_{\tilde m(\de)}(e^{it})=\rho_\de \circ (m^\vee)^{-1}(e^{it})=\rho_\de(e^{i\tau(t)})=r_\de\circ\tau(t).
\end{align*}
Hence, by equation \eqref{tauprime}, for all $t\in \r$,
\[
r_{\tilde m(\de)}'(t)= r_\de'\circ\tau(t) \tau'(t)= r_\de'\circ\tau(t)  \frac{1-|\al|^2}{|1+ c\al e^{it}|^2}.
\]
Differentiate again to obtain
\be\label{rdblprim}
r_{\tilde m(\de)}''(t)= \frac{1-|\al|^2}{|1+ c\al e^{it}|^4}\left( |1+ c\al e^{it}|^2\frac{d}{dt} r_\de'\circ\tau(t)-r_\de'\circ\tau(t)\frac{d}{dt}|1+ c\al e^{it}|^2\right).
\ee
Note that
\[
e^{it}=m^\vee(\omega_0) \Leftrightarrow (m^\vee)^{-1}(e^{it})=\omega_0 \Leftrightarrow e^{i\tau(t)} = e^{it_0}.
\]
Let $t_1$ be such that $\tau(t_1)=t_0$. Then $e^{it_1}= m^\vee(\omega_0)$ and, by equation \eqref{critpt},
\[
r_\de'\circ\tau(t_1)= r_\de'(t_0)=0.
\]
Hence the final term on the right hand side of equation \eqref{rdblprim} vanishes when $t=t_1$, and we have
\begin{align}\label{211015}
r_{\tilde m(\de)}''(t_1)&= \frac{1-|\al|^2}{|1+ c\al e^{it_1}|^2}\frac{d}{dt} r_\de'\circ \tau(t)\big|_{t=t_1} \notag\\
	&= \frac{1-|\al|^2}{|1+ c\al m^\vee(\omega_0)|^2} r_\de''\circ \tau(t_1)\tau'(t_1)\notag \\
	&=\frac{(1-|\al|^2)^2}{|1+ c\al m^\vee(\omega_0)|^4} r_\de''(t_0).
\end{align}
Now
\[
1+ c\al m^\vee(\omega_0) = \frac{1-|\al|^2}{1-\al\omega_0},
\]
and so equation \eqref{211015} becomes
\[
r_{\tilde m(\de)}''(t_1)= \frac{|1-\al\omega_0|^4}{(1-|\al|^2)^2} r_\de''(t_0).
\]
The last equation is equivalent to equation \eqref{sectild}.
\end{proof}
\begin{lemma}\label{0021}
If $\cald$ is a complex geodesic in $G$ such that
\be\label{sproynodes}
\cald^-\cap\calr^- = \{(0,0),(2,1)\}
\ee
then $\cald= k_r(\d)$ for some $r\in (0,1)$, where
\be\label{formkr}
k_r(z)=  \frac{1}{1-rz}\left(2(1-r)z, z(z-r)\right)\quad\mbox{ for }z\in\d.
\ee
\end{lemma}
\begin{proof}
Let $\de$ be a \nd datum that contacts $\cald$ and let $k$ solve $\Kob(\de)$.  Then $\cald=k(\d)$.
Since $(0,0) \in k(\d)$ and $(2,1)\in k(\d^-)$ we may assume (composing $k$ with an automorphism of $\d$ if necessary) that  $k(0)=(0,0)$ and $k(1)=(2,1)$.  By \cite[ Remark 2]{pz05}, $k$ can be written in the form
\[
k(\la)= \frac{1}{1-\bar\al\la}\left( 2\tau(1-|\al|)\la, \tau^2\la(\la-\al)\right)
\]
for some $\tau\in\t$ and $\al\in\d$.  Since $k(1)=(2,1)$ we have
\[
\frac{\tau(1-|\al|)}{1-\bar\al}=1=\frac{\tau^2(1-\al)}{1-\bar\al}.
\]
It follows that $1-|\al|=|1-\al|$, and so $0\leq\al<1$ and $\tau=1$.  Thus 
\[
k(\la)= \frac{1}{1-\al\la}\left( 2(1-\al)\la, \la(\la-\al)\right).
\]
If $\al=0$ then $k(\d)=\calr$, contrary to the hypothesis \eqref{sproynodes}.  Thus $0 < \al < 1$.
\end{proof}
\begin{proposition}\label{d2non0}
If $\de$ is a \nd datum in $G$ such that $\cald_\de^- \cap \calr^-$ consists of two points, one lying in $\calr$ and one lying in the boundary $\partial\calr$ of $R$,
then
\[
\frac{d^2}{dt^2}\rho_\de(e^{it}) \big |_{e^{it}=1} < 0.
\]
\end{proposition}
\begin{proof}
First consider the case that 
\be\label{specroynodes}
\cald_\de^- \cap \calr^- = \{(0,0), (2,1)\}.
\ee
Let $k$ solve $\Kob(\de)$.   By Lemma \ref{0021}, $k=k_r$ as in equation \eqref{formkr} for some $r\in (0,1)$.
Suppose that $\de$ is a discrete datum.  Then $\de=k(\zeta)$ for some discrete datum $\zeta$ in $\d$, say $\zeta=(z_1,z_2)$ for some distinct points $z_1,z_2 \in\d$.    Let $\omega=e^{it}$ throughout the proof.  Observe that
\[
\Phi_\omega\circ k(z)=\omega z\frac{z-\overline{\ga(\omega)}}{1-\ga(\omega) z}
\]
where 
\be\label{defgam}
\ga(\omega)= r 1+(1-r)\omega.
\ee
  Here $\ga\in\d$ if $\omega\in\t\setminus\{1\}$, while $\ga(1)=1$ and $\Phi_1\circ k(z)=-\omega z$.
Moreover
\be\label{modgasq}
|\ga(\omega)|^2= 1-2r(1-r)(1-\cos t).
\ee

We wish to differentiate with repect to $t$ the function
\begin{align}\label{2.3}
1-\rho_\de(e^{it}) &= 1-|\Phi_\omega(\de)|^2 \notag \\
	&=1-|\Phi_\omega\circ k(\zeta)|^2 \notag \\
	&= 1-\left|\frac{\Phi_\omega\circ k (z_1)- \Phi_\omega \circ k(z_2)}{1-\overline{\Phi_\omega\circ k(z_2)}\Phi_\omega\circ k(z_1)} \right|^2 \notag \\
	&= \frac{ (1-|\Phi_\omega \circ k(z_1)|^2)(1-|\Phi_\omega\circ k(z_2)|^2)}{| 1-\overline{\Phi_\omega\circ k(z_2)}\Phi_\omega\circ k(z_1)|^2}.
\end{align}
We have
\begin{align} 
1-\overline{\Phi_\omega\circ k(z_2)}\Phi_\omega\circ k(z_1) &= (1-\bar z_2 z_1) \left( 1+ \frac{\bar z_2 z_1(1-|\ga|^2)}{(1-\bar\ga\bar z_2)(1-\ga z_1)} \right)   \label{16.1} \\
	&=  \frac{(1-\bar z_2 z_1)(1-\bar\ga \bar z_2 - \ga z_1 + \bar z_2 z_1)}{(1-\bar\ga \bar z_2)(1-\ga z_1)}.    \label{16.1a}
\end{align}
By equation \eqref{2.3},
\begin{align}\label{goodlogform}
\log ( 1-\rho_\de(\omega))&= \half \log (1-|\Phi_\omega\circ k(z_1)|^2)^2 + \half \log (1-|\Phi_\omega\circ k(z_2)|^2)^2 \notag \\
	& \hspace*{1.5cm} - \log | 1-\overline{\Phi_\omega\circ k(z_2)}\Phi_\omega\circ k(z_1)|^2.
\end{align}
Consequently
\be\label{diffgoodform}
\frac{ \frac{d}{dt}(1-\rho_\de(\omega))}{1-\rho_\de(\omega)} =\half \mathrm{I} +\half \mathrm{II}- \mathrm{III}
\ee
where 
\begin{align}\label{IIII}
\mathrm{I} &= \frac{\frac{d}{dt} (1-|\Phi_\omega\circ k(z_1)|^2)^2}{(1-|\Phi_\omega\circ k(z_1)|^2)^2}, \notag \\
\mathrm{II}&= \frac{\frac{d}{dt} (1-|\Phi_\omega\circ k(z_2)|^2)^2}{(1-|\Phi_\omega\circ k(z_2)|^2)^2},  \\
\mathrm{III} &=\frac{\frac{d}{dt}  | 1-\overline{\Phi_\omega\circ k(z_2)}\Phi_\omega\circ k(z_1)|^2}{ | 1-\overline{\Phi_\omega\circ k(z_2)}\Phi_\omega\circ k(z_1)|^2}. \notag
\end{align}
Let us calculate $\mathrm{III}$.  From equations \eqref{defgam} and \eqref{modgasq} we have
\be\label{derivga}
\frac{d\ga}{dt}= i(1-r)\omega, \qquad \frac{d|\ga|^2}{dt} = - 2r(1-r)\sin t.
\ee
Thus, on differentiating equation \eqref{16.1} we obtain
\begin{align}
\frac{d}{dt}  (1-\overline{\Phi_\omega\circ k(z_2)}\Phi_\omega\circ k(z_1)) &= \frac{2r(1-r)\bar z_2 z_1(1-\bar z_2 z_1)}{(1-\bar\ga \bar z_2)^2 (1-\ga z_1)^2)} \times  \label{18.1}\\
	& \left( (1-\bar z_2)(1-z_1)\sin t + i(1-r)(\bar z_2-z_1)(1-\cos t) \right). \notag
\end{align}
On combining equations \eqref{18.1} and \eqref{16.1a} we find that
\be\label{19.1}
\frac{d}{dt}  | 1-\overline{\Phi_\omega\circ k(z_2)}\Phi_\omega\circ k(z_1)|^2= \frac{4r(1-r)|1-\bar z_2 z_1|^2}{|1-\ga z_1|^2 |1-\ga z_2|^2} \re A(t,z_1,z_2)
\ee
where
\begin{align}\label{19.2}
A(t,z_1,z_2)& = \frac{\bar z_2 z_1(1-\bar\ga \bar z_1 - \ga z_2 + \bar z_1 z_2)}{(1-\bar\ga \bar z_2)(1-\ga z_1)} \times \notag \\
	&\hspace*{1cm} \left( (1-\bar z_2)(1-z_1) \sin t + i(1-r)(\bar z_2 -z_1)(1-\cos t) \right).
\end{align}
Now use equations \eqref{16.1a} and \eqref{19.2} to deduce that
\be\label{21.1}
\mathrm{III}= 4r(1-r)\re\left\{ \frac{\bar z_2 z_1 \left((1-\bar z_2)(1-z_1)\sin t +i(1-r)(\bar z_2 -z_1)(1-\cos t)\right)} {(1-\bar\ga \bar z_2 -\ga z_1 + \bar z_2 z_1)(1-\bar\ga \bar z_2)(1-\ga z_1)} \right\}.
\ee
By equation \eqref{diffgoodform},
\[
\frac{d}{dt}\rho_\de(\omega)=-(1-\rho_\de(\omega))(\half\mathrm{I}+\half\mathrm{II}-\mathrm{III}).
\]
We shall differentiate this equation and put $t=0, \ \omega=1$.  Since $\omega=1$ is a maximizer for $\rho_\de$, the point $t=0$ is a critical point of $\rho_\de(e^{it})$, and $\rho_\de(1)=\car{\de}^2$.  Thus
\be\label{thisisit}
\frac{d^2}{dt^2}\rho_\de(\omega) \big|_{t=0} =  -(1-\car{\de}^2)\frac{d}{dt}(\half\mathrm{I}+\half\mathrm{II}-\mathrm{III})\big|_{t=0}.
\ee
In view of equations \eqref{defgam} and \eqref{derivga}, at the point $t=0$,
\[
\ga=1, \quad \frac{d\ga}{dt}= i(1-r) \; \; \mbox{ and } \;\; \frac{d|\ga|^2}{dt}=0.
\]
Hence 
\begin{align*}
\frac{d}{dt}\mathrm{III}\big|_{t=0}&= 4r(1-r)\re\left\{\bar z_2 z_1 \frac{d}{dt} \frac {(1-\bar z_2)(1-z_1)\sin t +i(1-r)(\bar z_2-z_1)(1-\cos t)} {(1-\bar\ga \bar z_2 -\ga z_1 + \bar z_2 z_1)(1-\bar\ga \bar z_2)(1-\ga z_1)}      \right\} \big|_{t=0}    \\
	& = 4r(1-r)\re \frac{\bar z_2 z_1}{(1-\bar z_2)(1-z_1)}.
\end{align*}

Since $\mathrm{I, II}$ are the special cases of $\mathrm{III}$ obtained when $z_2=z_1$ and $z_1=z_2$ respectively, it follows that
\begin{align*}
\frac{d}{dt}\mathrm{I}\big|_{t=0}&=\frac {4r(1-r)|z_1|^2}{|1-z_1|^2}, \\ 
\frac{d}{dt}\mathrm{II}\big|_{t=0}&=\frac {4r(1-r)|z_2|^2}{|1-z_2|^2}.
\end{align*}
By equation \eqref{thisisit},
\begin{align*}
\frac{d^2}{dt^2}\rho_\de(\omega) \big|_{t=0}& =  -(1-\car{\de}^2)2r(1-r)\left(\frac {|z_1|^2}{|1-z_1|^2} + \frac{|z_2|^2}{|1-z_2|^2}
-2\re \frac{\bar z_2 z_1}{(1-\bar z_2)(1-z_1)} \right)\\
	&=  -(1-\car{\de}^2)2r(1-r)\left|\frac{z_1}{1-z_1} - \frac{z_2}{1-z_2}\right|^2 \\
	&< 0.
\end{align*}
We have proved Proposition \ref{d2non0} for discrete datums $\de$ in the case that 
\[
\cald_\de^-\cap \calr^-=\{(0,0),(2,1)\}.
\]

Now consider the case that $\de$ is an infinitesimal datum.  Then $\de=k(\zeta)$ for some infinitesimal datum $\zeta\in T\d$, say $\zeta=(z,v)$ where $v\neq 0$.  We find that
\[
\Phi_\omega\circ k(\zeta)= \left(\omega z B_{\bar\ga}(z), -\omega v \frac{\bar\ga-2z+\ga z^2}{(1-\ga z)^2}\right)
\]
where $\ga$ is given by equation \eqref{defgam} and $B_{\bar\ga}$ is a Blaschke factor.
Therefore
\begin{align*}
|\Phi_\omega\circ k(\zeta)| &= \frac{ |v|\, |\bar\ga-2z + \ga z^2|}{|1-\ga z|^2(1-|z B_{\bar\ga}(z)|^2)} \\
	&= \frac{|v| \, |\bar\ga-2z+\ga z^2|}{(1-|z|^2)(1-2\re(\ga z)+|z|^2)}.
\end{align*}
Hence
\begin{align*}
\log \rho_\de(\omega) &= \log |\Phi_\omega\circ k(\zeta)|^2 \\
	&= 2\log\frac{|v|}{1-|z|^2} + \log |\bar\ga -2z+\ga z^2|^2 - 2\log (1-2\re(\ga z)+|z|^2),
\end{align*}
and so
\be\label{firstderiv}
\frac{d}{dt}\rho_\de(\omega) = \rho_\de(\omega)(\mathrm{I}-2\mathrm{II})
\ee
where
\begin{align*}
\mathrm{I}&= \frac{\frac{d}{dt} |\bar\ga -2z+\ga z^2|^2}{ |\bar\ga -2z+\ga z^2|^2}, \\
\mathrm{II}&= \frac{\frac{d}{dt}(1-2\re(\ga z)+|z|^2)}{1-2\re(\ga z)+|z|^2}.
\end{align*}
Now, by equations \eqref{derivga},
\be\label{another}
\frac{d}{dt}(\bar\ga - 2z+\ga z^2) = -i(1-r)(\bar\omega-\omega z^2),
\ee
and therefore
\be\label{gotI}
\mathrm{I}= 2(1-r)\im\frac{\bar\omega-\omega z^2}{\bar\ga - 2z+ \ga z^2}.
\ee
Likewise
\be\label{gotII}
\mathrm{II} = \frac{2(1-r)\im(\omega z)}{1-2\re(\ga z) +|z|^2}.
\ee

Differentiate equation \eqref{firstderiv} and set $t=0$.  Since $0$ is a critical point of $\rho_\de(e^{it})$ and $\rho_\de(1)=\car{\de}^2$,
\be\label{secder}
\frac{d^2}{dt^2}\rho_\de(\omega)\big|_{t=0} = \car{\de}^2\left( \frac{d}{dt}\mathrm{I} -2 \frac{d}{dt}\mathrm{II}\right)\big|_{t=0}.
\ee
From equations \eqref{gotI}, \eqref{another} and \eqref{derivga} one has
\[
\frac{d}{dt}\mathrm{I}\big|_{t=0}= -2(1-r)\re \frac{-2z+r(1+z)^2}{(1-z)^2},
\]
while
\begin{align*}
\frac{d}{dt}\mathrm{II}\big|_{t=0} &=2(1-r)\im \frac{d}{dt} \frac{\omega z}{1-2\re(\ga z) +|z|^2}\big|_{t=0} \\
	&=  \frac{2(1-r)}{|1-z|^4}  \{ |1-z|^2 \re z - 2(1-r)(\im z)^2\}.
\end{align*}
Hence, by equation \eqref{secder},
\begin{align*}
\frac{d^2}{dt^2}\rho_\de(\omega)\big|_{t=0} &= -2(1-r) \car{\de}^2( \mathrm{III}+ r \mathrm{IV})
\end{align*}
where
\begin{align*}
\mathrm{III}&=\re \left\{\frac{2z}{|1-z|^2}-\frac{2z}{(1-z)^2}\right\}- \frac{4(\im z)^2}{|1-z|^4} \\
	&=0, \\
\mathrm{IV}&=  \re \frac{(1+z)^2}{(1-z)^2} +\frac{4(\im z)^2}{|1-z|^4} \\
	&=\frac{(1-|z|^2)^2}{|1-z|^4}.
\end{align*}
Thus 
\begin{align*}
\frac{d^2}{dt^2}\rho_\de(\omega)\big|_{t=0}&= -2r(1-r) \car{\de}^2\frac{(1-|z|^2)^2}{|1-z|^4} \\
	&<0.
\end{align*}

We have proved the conclusion of the Proposition in the case that $\cald_\de^- \cap \calr^- = \{(0,0), (2,1)\}$.   Now consider the general case: suppose that
\be\label{genroynodes}
\cald_\de^- \cap \calr^- = \{(2\al,\al^2), (2\bar\omega_0,\bar\omega_0^2)\}
\ee
for some $\al\in\d, \, \omega_0\in\t$.  Let
\[
m=\overline{B_\al(\bar\omega_0)} B_\al  \in \aut\d.
\]
Since $\tilde m$ is an automorphism of $G$ and $\de$ contacts $\cald_\de$, the datum $\tilde m$ contacts
$\tilde m(\cald_\de)$, and therefore
\[
\tilde m(\cald_\de^-)= \cald_{\tilde m(\de)}^-, \qquad \tilde m(\calr^-)=\calr^-
\]
and so
\[
\tilde m(\cald_\de^- \cap \calr^-) = \cald_{\tilde m(\de)}^-\cap \calr^-.
\]
Hence, in view of the equation \eqref{genroynodes},
\[
\cald_{\tilde m(\de)}^-\cap \calr^- = \tilde m( \{(2\al,\al^2), (2\bar\omega_0,\bar\omega_0^2)\}) = \{(0,0),(2,1)\}.
\]

By the above special case,
\[
\frac{d^2}{dt^2} \rho_{\tilde m(\de)}(e^{it})\big|_{e^{it}=1} < 0.
\]
It follows from Lemma \ref{rhomtild} that
\[
\frac{d^2}{dt^2} \rho_\de (e^{it})\big|_{e^{it}=\omega_0} < 0.
\]
\end{proof}

\chapter{A geometric classification of geodesics in $G$}\label{classifyG}

All geodesics in a domain $U$ are coequal for Euclidean geometry, in the sense that they are all properly embedded analytic discs, and so conformally equivalent. However they can differ from the point of view of the intrinsic geometry of $U$. Say two geodesics $\cald_1$ and $\cald_2$ are \emph{equivalent} in $U$ (written $\cald_1 \sim \cald_2$) if there exists an automorphism $\tau$ of $U$ such that $\tau(\cald_1) = \cald_2$.
\index{geodesics!equivalence of}
 In the case when $U=\mathbb{B}$,  the unit ball in $\c^n$,   the group of automorphisms acts transitively on the set of complex geodesics, and so there is a single equivalence class.   In $G$, on the other hand, we identify five distinct species of geodesic, and three of these five species consist of one-parameter families of pairwise inequivalent geodesics (Theorem \ref{formgeos}), while the remaining two species comprise a single equivalence class.

 In this chapter we prove a geometric characterization of the type of a datum $\de$ (in the sense of  Definition \ref{extdef10}) in terms of the intersection of the complex geodesic $\cald_\de$ with the royal variety (more precisely, with the intersection of the corresponding closures).  The result is Theorem \ref{geothm30}.  It has two important consequences: firstly, the type of a datum is preserved by automorphisms of $G$ (Corollary \ref{holinv}), and secondly, if two datums contact the same geodesic then they have the same type (Corollary \ref{sametype}).
The correspondence between \nd datums and geodesics therefore permits a classification of geodesics too into five types. Recall the Pentachotomy Theorem \ref{extprop10}, which states
that, for  a \nd datum $\delta$ in $G$,  exactly one of the cases $(1)$ to $(5)$ in Definition \ref{extdef10} holds.  

The definitions of the five types of datum are not, at a first glance, geometric in character: it is not immediately apparent that types are preserved by automorphisms of $G$.  However, the following theorem characterizes types of datums purely in terms of the geometry of $G$.

\begin{theorem}\label{geothm30}
Let $\delta$ be a \nd datum in $G$ and let $E=\cald_\delta^- \cap \calr^-$. Then
\begin{enumerate}[\rm (1)]
    \item\ \  $E$ consists of two points, one lying in $\calr$ and one lying in $\partial \calr$ $\iff \delta $ is purely unbalanced;
   \item\ \   $E$ consists of a single point that lies in $\partial \calr$ $\iff \delta$ is exceptional;
       \item\ \  $E$ consists of two points, both lying in $\partial \calr$  $\iff \delta $ is purely balanced;
\item\ \  $E$ consists of a single point that lies in $\calr$ $\iff \delta$ is flat;
        \item\ \  $E$ = $\calr^-$ $\iff \delta$ is royal.
\end{enumerate}
\end{theorem}
\index{datum!flat}
\index{datum!exceptional}
\index{datum!purely unbalanced}
\index{datum!purely balanced}
\index{datum!royal}
\index{theorem!geometric classification of geodesics}
\begin{proof}   
By Theorem \ref{geothm10},  there exists a solution $k$ of $\Kob(\de)$ and $\cald_\de=k(\d)$.
By Theorem \ref{descgeos},  this $k$ is a rational $\Gamma$-inner function and $k^2$ is a Blaschke product of degree $1$ or $2$. 
Since $k$ is  $\Gamma$-inner, $k(\t) \subseteq b\Gamma $.

If $k(\d^-) \neq \royal^-\cap \Gamma$, by Theorem \ref{number_royal},  $k(\d^-)$ meets $\royal^-$ exactly $\deg(k)$ times,
that is, once or twice, counted with multiplicity.\\

(4)  Suppose that $E$ consists of a single point $s_1=k(z_1)$ that lies in $\calr$ (thus in $G$ rather than $\partial G$). 
Now $z_1$ is a simple zero of the rational function $R=(k^1)^2-4k^2$ in $\d$.  For suppose it has multiplicity greater than $1$.
By \cite[Proposition 3.5]{ALY15}, the zeros of $R$ are symmetric with respect to $\t$, and it follows that $R$ has at least $4$ zeros lying off $\t$. However, since $\deg(k)\leq 2$, $R$ has degree at most $4$, and so $\deg(R)=4$,  $\deg(k)=2$  and $k$ has no royal nodes in $\t$.  This contradicts the fact (see Proposition \ref{georoyal}) that any complex C-geodesic of degree $2$ has a royal node in $\t$.  Hence $z_1$ is the unique royal node of $k$ and is simple.
By Theorem \ref{number_royal}, the complex C-geodesic $k$ has degree $1$.   Thus there is a flat complex C-geodesic that solves $\Kob(\de)$.  By Proposition \ref{twoflats},  $\de$ is a flat datum.   

Conversely, suppose $\delta$ is a flat datum. By Proposition \ref{twoflats}, there is a flat complex C-geodesic, which has degree 1, that solves $\Kob(\de)$.  By Theorem \ref{number_royal}, $k$ has exactly one royal node, counted according to multiplicity.
However, by Proposition \ref{deg1flat}, any flat complex C-geodesic has the form $f_\beta\circ m$ for some $\beta\in\d, \ m\in\aut\d$ and so has a royal node in $\d$, by Proposition \ref{flatcaproyal}.  Hence $k$ has a single royal node, which lies in $\d$, and so  $E$ consists of a single point that lies in $\calr$.   The equivalence (4) is proved.\\


(5) Suppose  $E = \royal^-$.  Then $\royal^-=\cald_\de^-\cap \royal^-$ and so  $\royal^- \subseteq \cald_\de^-=k(\d^-)$.  Thus the holomorphic function $(k^1)^2-4k^2$ has uncountably many zeros in $\d$, and so vanishes identically.  Thus $\cald_\de = \calr$, and so $\de$ contacts $\calr$.  By Lemma \ref{Phioconst}, $\Phi_\omega(\de)$ is independent of $\omega$.  Thus $\rho_\de$ is constant on $\t$, and so $\de$ is either flat or royal.  However, by statement (4), $\delta$ is not flat, and so $\de$ is royal.

Conversely, suppose that $\de$ is royal: $\rho_\de$ is constant on $\t$ and $C(s)=\half s^1$ solves $\Car(\de)$.  Thus $\Phi_\omega$ solves $\Car(\de)$ for all $\omega\in\t$.  

Define $h\in G(\d)$ by $h(z)=(2z,z^2)$.  Then $C\circ h=\id{\d}$.  We claim that there exists a datum $\zeta$ in $\d$ such that $h(\zeta)=\de$.
In the case that $\de$ is discrete, say $\de=(s_1,s_2)$, Theorem \ref{flatandroyal} asserts that  $s_1$ and $s_2$ either both lie in $\calr$ or both lie in $\{(\beta+\bar\beta z,z):z\in\d\}$ for some $\beta\in\d$.  If the second alternative holds then $C_1(s)=s^2$ is easily seen to solve $\Car(\de)$, and so $\de$ is both flat and royal, contrary to Theorem \ref{extprop10}.  Thus $s_1,s_2$ both lie in $\calr$, so that $s_1=h(z_1), \ s_2=h(z_2)$ for some $z_1,z_2\in\d$.  Hence $\de=h(\zeta)$ where $\zeta=(z_1,z_2)$.

In the case that $\de$ is infinitesimal, say $\de=(s_1,v)$, Theorem \ref{flatroyalesimal} asserts that either (a) there exists $z_1\in\d$ such that $s_1=(2z_1,z_1^2)$ and $v$ is collinear with $(1,z_1)$, or (b) there exist $\beta,z_1\in\d$ such that $s_1=(\beta+\bar\beta z_1,z_1)$ and $v$ is collinear with $(\bar\beta,1)$, say $v=c(\bar\beta,1), \, c\neq 0$.  

We claim that (b) does not hold.  For suppose it does. Let $\eta$ be the datum $(z_1, c)$ in $\d$ and let $C_1(s)=s^2$.
Then $h(\eta)=\de$ and $C_1\circ h=\id{\d}$.  It follows that $C_1$ solves $\Car(\de)$, and therefore $\de$ is both royal and flat, again contrary to Theorem \ref{extprop10}.  Hence (a) holds.
Now, if $\zeta$ is the datum $(z_1, \half c)$ in $\d$ then $h(\zeta)=\de$.  

Thus, in either case, there is a datum $\zeta$ in $\d$ such that $h(\zeta)=\de$.  By Lemma \ref{Cksolve}, since $C\circ h=\id{\d}$,
it follows that $h$ solves $\Kob(\de)$.  Hence $\cald_\de=h(\d)=\calr$, and so $E=\calr^-$.  
Thus the equivalence (5) holds.\\

(3)  Suppose $E$ consists of two points, both in $\partial\calr$.  If $\deg(k)=1$ then, as is shown in the proof of (4) above, $\de$ is flat and hence $E\cap\partial\calr=\emptyset$, contrary to hypothesis.  Thus $\deg(k)=2$.  By Lemma \ref{cancel_maxim},  
   $\rho_\delta$ has exactly two maximizers in $\t$, which is to say that $\delta$ is purely balanced.

Conversely, suppose that $\de$ is purely balanced, so that $\rho_\de$ has exactly two maximizers, say $\omega_1\neq \omega_2$ in $\t$.  Then once again  $\de$ is not flat and so $\deg(k)=2$.  By Lemma \ref{cancel_maxim}, 
\[
\cald_\de^-\cap\partial\calr=k(\d)^-\cap\partial\calr=\{(2\bar\omega_1,\bar\omega_1^2), (2\bar\omega_2,\bar\omega_2^2)\}.
\]
Hence $E$ contains two distinct points in $\partial\calr$.  Since $k$ has exactly $2$ royal nodes,  by Theorem \ref{number_royal}, it follows that $k$ has no royal points in $G$, and so $E=\{(2\bar\omega_1,\bar\omega_1^2), (2\bar\omega_2,\bar\omega_2^2)\}$.
Thus the equivalence (3) holds.\\

(1)  Suppose that $E$ consists of two points, one in $\calr$ and one, say $(2\bar\omega_0,\bar\omega_0^2)$, in $\partial\calr$.  By statement (4) above, $\de$ is not flat and hence, by Propositions \ref{twoflats} and  \ref{deg1flat},  $\deg(k)=2$.  By Lemma \ref{cancel_maxim}, $\rho_\de$ has the unique maximizer $\omega_0$ in $\t$, and by Proposition \ref{d2non0},
\be\label{lt0}
\frac{d^2}{dt^2}\rho_\de(e^{it})\big|_{e^{it}=\omega_0} < 0.
\ee
By definition, $\de$ is purely unbalanced.

Conversely, suppose that $\de$ is purely unbalanced, so that $\rho_\de$ has a unique maximizer, say $\omega_0$ in $\t$, and the inequality \eqref{lt0} holds.  By Proposition \ref{secderiv}, $E \neq \{(2\bar\omega_0, \bar\omega_0^2)\}$.   Thus $k$ has a second royal node $z_0\in\d^-$ such that $k(z_0)\neq (2\bar\omega_0, \bar\omega_0^2)$.  If $z_0\in\t$ then $k(z_0)\in\partial\calr$ and, by Lemma  \ref{cancel_maxim}, $\rho_\de$ has a second maximizer in $\t$, contrary to the assumption that $\de$ is purely unbalanced.  Hence $z_0\in\d$, so that $k(z_0)\in\calr$.  Thus $E$ contains two points, one in $\calr$ and one in $\partial\calr$.  Since $k$ has at most two royal nodes, by Theorem \ref{number_royal}, $E$ consists of these two points.
The equivalence (1) is proved.\\

(2)  Since $\deg(k)$ is one or two, by Theorem \ref{number_royal}, either $k(\d)=\calr$ or $k$ has one or two royal nodes.  There are thus five possibilities for $E$, as described in the statements (1)-(5).  Thus $E$ consists of a single point lying in $\partial\calr$ if and only if $E$ is not of the forms described in statements (1), (3), (4) and (5).  Consequently $E$ consists of a single point lying in $\partial\calr$ if and only if $\de$ is not purely unbalanced, purely balanced, flat or royal, which in turn is so if and only if $\de$ is exceptional, by Theorem \ref{extprop10}.
\end{proof}
\begin{corollary}\label{holinv}
The type of a datum in $G$ is holomorphically invariant: if $\de$ is a \nd datum in $G$ and $\tilde m\in\aut G$ then $\de$ and $\tilde m(\de)$ are of the same type.
\end{corollary}
The statement is immediate from Theorem \ref{geothm30} and the fact that every automorphism $\tilde m$ of $G$ extends continuously to a bijective self-map of $\Ga$ (Corollary \ref{corautos}).

\begin{corollary}\label{sametype}  The type of a complex geodesic is unambiguously defined: if $\de_1, \de_2$ are \nd datums in $G$ and $\cald_{\de_1}=\cald_{\de_2}$ then $\de_1$ and $\de_2$ have the same type.
\end{corollary}
This statement too is clear from  Theorem \ref{geothm30}.

The following geometric characterization of balanced geodesics in $G$ is an immediate corollary of Theorem \ref{geothm30}
and Definition \ref{extdef10}.

\begin{proposition}\label{prop3.9}
If $\delta$ is a \nd datum in $G$, then $\delta$ is unbalanced if and only if $\cald_\delta$  meets $\calr$ in a single point. $\delta$ is balanced if and only if either $\cald_\delta \cap \calr = \emptyset$ or $\cald_\delta=\calr$.
\end{proposition}

In the light of Theorem \ref{geothm30}, the terms \emph{flat, exceptional, purely unbalanced, purely balanced}, and \emph{royal} (as introduced in Definition \ref{extdef10}) are all terms that  refer not only to qualitative properties of the Carath\'eodory extremal problem associated with a given datum $\delta$ but also to geometric properties of the unique geodesic associated with $\delta$.

\begin{definition}\label{typesofgeos}
A geodesic $\cald$ in $G$ is said to be {\em flat}, or of {\em flat type}, if there is a \nd flat datum $\de$ in $G$ such that $\cald=\cald_\de$.
Analogous definitions apply for exceptional, purely unbalanced, purely balanced and royal geodesics.
\end{definition}
\index{geodesics!flat}
\index{geodesics!exceptional}
\index{geodesics!purely unbalanced}
\index{geodesics!purely balanced}
\index{geodesics!royal}
\index{geodesics!types of}
\begin{proposition}\label{calfbeta}
The flat geodesics in $G$ are the sets $\calf_\beta$  {\rm (}defined in equation \eqref{defcalfbeta}{\rm )} for $\beta\in\d$.
\end{proposition}
\begin{proof}
Let $\cald$ be a flat geodesic in $G$ and let $\de$ be a flat datum such that $\cald=\cald_\de$.  By Proposition \ref{twoflats} there is a flat C-geodesic $k$ that solves $\Kob(\de)$, and by Proposition \ref{deg1flat}, $k= f_\beta\circ m$ for some $\beta\in\d$ and $m\in\aut\d$.  Hence
\[
\cald=\cald_\de=k(\d)=f_\beta(\d)=\calf_\beta.
\]
\end{proof}
\begin{proposition}\label{flatcaproyal}
The closure of a flat geodesic meets the closure of the royal variety $\calr^-$ exactly once, and the common point lies in $\calr$.
\end{proposition}
\begin{proof}
Consider the flat geodesic $\calf_\beta=f_\beta(\d)$ where $\beta\in\d$.
The royal nodes of $f_\beta$ are the points $z\in\d^-$ such that $(\beta+\bar\beta z,z) \in\calr^-$, that is, such that
\be\label{quadra}
(\beta+\bar\beta z)^2=4z.
\ee
If $\beta=0$ then the unique royal node of $f_\beta$ is $0$ and $\calf_0^-\cap \calr^- =\{(0,0)\}$.  If $\beta\neq 0$ then the equation \eqref{quadra} has two distinct roots in $\c$ which are symmetric with respect to the unit circle.  Hence there is no root in $\t$ and exactly one root in $\d$.  Thus $\calf_\beta^-\cap \calr^-$ consists of a single point, which belongs to $\calr$.
\end{proof}

The type of a complex geodesic is defined in terms of solutions of the Carath\'eodory extremal problem.  It follows, though, that every solution $k$ of a Kobayashi problem has a type, and one can ask for a characterization of types directly in terms of the complex C-geodesic $k\in G(\d)$.
 Recall that two geodesics $\cald_1, \cald_2$ are equivalent (written  $\cald_1\sim \cald_2$) if there exists an automorphism $\tilde m$ of $G$ such that $\tilde m(\cald_1)=\cald_2$.
\begin{theorem}\label{formgeos}
Let $\cald$ be a complex geodesic of $G$.
\begin{enumerate}[\rm (1)]
\item $\cald$ is  purely unbalanced if and only if $\cald\sim k_r(\d)$  for some $r\in (0,1)$ where
\[
k_r(z) = \frac{1}{1-r z}\left(2(1-r)z, z(z-r)\right) \quad \mbox{ for all } z\in\d;
\]
\item $\cald$ is exceptional if and only if $\cald\sim h_r(\d)$  for some real number $r>0$, where
\begin{align}\label{hrz}
h_r(z)&=(z+m_r(z),zm_r(z))
\end{align}
and
\be\label{mrz}
m_r(z) = \frac{(r-i)z+i}{r+i-iz}  \quad \mbox{ for } z\in\d;
\ee
\item  $\cald$ is purely balanced if $\cald\sim g_r(\d)$  for some $r\in (0,1)$, where
\[
g_r(z)=(z+B_r(z), zB_r(z)) \quad \mbox{ for all } z\in\d;
\]
\item $\cald$ is flat if and only if $\cald \sim k(\d)$ where $k(z)=(0,z)$;
\item $\cald$ is royal if and only if $\cald\sim k(\d)$ where $k(z)=(2z,z^2)$.
\end{enumerate}
Moreover, in statements (1), (2) and (3),  the corresponding geodesics are pairwise inequivalent for distinct values of $r$ in the given range.
\end{theorem}
\index{$m_r$}
\index{$k_r$}
\index{$g_r$}
\index{$h_r$}
\index{geodesics!flat}
\index{geodesics!exceptional}
\index{geodesics!purely unbalanced}
\index{geodesics!purely balanced}
\index{geodesics!royal}
\index{theorem!characterization of geodesics in terms of solutions of $\Car(\delta)$}
\begin{proof}
(1) The royal nodes of $k_r$ are $0$ and $1$ and the corresponding royal points are $(0,0)$ and $(2,1)$. Thus $k(\d)^-\cap\calr^- =\{(0,0),(2,1)\}$.  Hence, if $\cald\sim k(\d)$ then $\cald^-\cap \calr^-$ consists of a point in $\calr$ and another in $\partial\calr$.  By Theorem \ref{geothm30}, $\cald$ is purely unbalanced.

Conversely, let $\cald$ be purely unbalanced.  By Theorem \ref{geothm30},
\[
\cald^-\cap \calr^-=\{(2\al,\al^2), (2\tau,\tau^2)\}
\]
 for some $\al\in\d$ and $\tau\in\t$.  Choose $m\in\aut\d$ such that $m(\al)=0$ and $m(\tau)=1$.  Then $\tilde m(\cald^-)\cap\calr^- =\{(0,0),(2,1)\}$.  By Lemma \ref{0021}, $\tilde m(\cald)=k_r(\d)$ for some $r\in (0,1)$, and then $\cald = \tilde m^{-1}(k_r(\d)) \sim k_r(\d)$.

To prove pairwise inequivalence of the geodesics $k_r(\d), \ 0<r<1$, suppose that $k_r(\d)\sim k_{r'}(\d)$ for some pair $r,r'$.  Then there exists $m_1\in\aut\d$ such that $\tilde m_1(k_r(\d))=k_{r'}(\d)$.  According to \cite[Corollary 4.6.4]{Koba}, in a general domain $\Omega$, if $f, g$ are complex C-geodesics in $\Omega$ having the same range then there is an automorphism $\al$ of $\Omega$ such that $g=f\circ \al$.  Hence there exists $m_2\in\aut\d$ such that
\[
k_{r'}=\tilde m_1\circ k_r\circ m_2.
\]
Since $\tilde m_1^{-1}$ preserves the royal variety, $m_2$ maps the royal nodes of $k_{r'}$ to the royal nodes of $k_r$.  These royal nodes are $0$ and $1$ for both functions, and so $m_2$ leaves both $0$ and $1$ fixed.  Hence $m_2=\id{\d}$ and $k_{r'}=\tilde m_1\circ k_r$.  Similarly, $\tilde m_1$ maps the royal points of $k_{r'}$ to the royal points of $k_r$, which implies that $m_1$ also fixes the points $0$ and $1$, and so $m_1=\id{\d}$.  Thus $k_r=k_{r'}$, which easily implies that $r=r'$.\\

(2)  We first show that $h_r(\d)$ is an exceptional geodesic for any $r\in\r\setminus\{0\}$. Let 
\[
\tau = \frac{r-i}{r+i}\quad \mbox{ and }\quad  \al= \half(1-\bar\tau).
\]
Then $\tau\in \t\setminus \{-1,1\}$ and $\al\in\d\setminus \{0\}$, and we have $m_r=\tau B_\al \in\aut\d$. 

Moreover, since $|1-\tau|=2|\al|$, it follows that the equation $m_r(z)=z$ has a unique root in $\t$, which is to say that $h_r$ has a unique royal point in $\partial\calr$, and no royal points in $\calr$.  By Theorem \ref{extprop30}, $h_r(\d)$ is an exceptional geodesic in $G$.

Conversely, suppose that $\cald$ is an exceptional geodesic in $G$.  By Theorem \ref{geothm30}, $\cald^-\cap \calr^-$ consists of a single point in $\partial\calr$, which we may take (replacing $\cald$ by an equivalent geodesic) to be $(2,1)$.  Let $k$ be a complex C-geodesic in $G$ such that $\cald=k(\d)$.  Since $\cald\cap\calr=\emptyset$, Proposition \ref{prop3.10} implies that we can assume that $k(z)=(z+m(z), zm(z))$ for some $m\in \aut\d$, and since the royal points of this $k$ are the points $(2z,z^2)$ where $z\in\d^-$ satisfies $m(z)=z$, we infer that $m$ has the unique fixed point $1$ in $\d^-$.  Any such automorphism has the form $m=\tau B_\al$ for some  $\al\in\d\setminus\{0\}$ and $\tau\in\t\setminus\{-1,1\}$ such that $1-\tau=2\bar\al$.

Let 
\[
r=i\frac{1+\tau}{1-\tau}.
\]
Then $r\in\r\setminus\{0\}$ and
\[
\al=\half (1-\bar\tau)= -\frac{i}{r-i}.
\]
One may then verify that $m=m_r$, 
and so $k=h_r$ for some $r\in\r\setminus\{0\}$.  Now
\[
m_r^{-1}(z)=\frac{(r+i)z-i}{r-i+iz}= m_{-r}(z)
\]
and
\[
h_r\circ m^{-1}(z)= (z+m_r^{-1}(z),zm_r^{-1}(z))=h_{-r}(z).
\]
Thus, for any $r\in\r\setminus\{0\}, \, h_r(\d)=h_{-r}(\d)$.  Hence $\cald\sim h_r(\d)$ for some $r>0$.

Suppose that $h_r(\d)\sim h_{r'}(\d)$ for some $r,r'>0$.  We shall show that $r=r'$.  The equivalence implies that there exist $a_1,a_2\in\aut\d$ such that
\be\label{circcirc}
h_{r'}= \tilde a_2\circ h_r\circ a_1.
\ee
Since $a_1, \tilde a_2$ fix the common royal nodes and points respectively of $h_r, h_{r'}$, necessarily $a_1(1)=1=a_2(1)$.

Equation \eqref{circcirc} can be written
\begin{align*}
(z+m_{r'}(z), zm_{r'}(z)) &= \tilde a_2\circ (a_1(z)+m_r\circ a_1(z), a_1(z)m_r\circ a_1(z)) \\
	&= \left(a_2\circ a_1(z)+a_2\circ m_r\circ a_1(z),  \left(a_2\circ a_1(z)\right)\left(a_2\circ m_r\circ a_1(z)\right) \right)
\end{align*}
for all $z\in\d$.  It follows that, for every $z$, either
\[
a_2\circ a_1(z)=z \quad \mbox{ and }\quad a_2\circ m_r\circ a_1(z) = m_{r'}(z)
\]
or
\[
a_2\circ a_1(z) =m_{r'}(z)\quad \mbox{ and }\quad a_2\circ m_r\circ a_1(z) =z.
\]
Therefore either 
\be\label{alt1}
a_2\circ a_1=\id{\d}\quad \mbox{ and }\quad  a_2\circ m_r\circ a_1 = m_{r'}
\ee
or
\be\label{alt2}
a_2\circ a_1 =m_{r'}\quad \mbox{ and }\quad a_2\circ m_r\circ a_1=\id{\d}.
\ee
If the alternative \eqref{alt1} holds, then $a_2=a_1^{-1}$ and 
\be\label{conjug}
m_r\circ a_1\circ m_{-r'}=a_1.
\ee
Let $a_1=cB_\ga$ for some $c\in\t$ and $\ga\in\d$.  Since $a_1(1)=1$,
\[
c=\frac{1-\bar\ga}{1-\ga}.
\]
In view of equation \eqref{conjug} we have $m_r\circ a_1\circ m_{-r'}(\ga) =0$, and therefore
\[
a_1\circ m_{-r'}(\ga)= -\frac{i}{r-i}.
\]
On writing out this equation in detail we obtain
\[
\frac {1-\bar\ga}{1-\ga} \frac{(r'+i-i\ga)\ga-i-\ga(r'-i)}{(-\bar\ga(r'+i)+i)\ga+\bar\ga i+r'-i}=-\frac{i}{r-i}.
\]
This equation simplifies to $r=r'$.

The remaining possibility is that the alternative \eqref{alt2} holds.  In this case we may write
\[
m_{-r'}\circ a_2\circ a_1=\id{\d}, \qquad m_{-r'}\circ a_2\circ m_r\circ a_1=m_{-r'}.
\]
Let $b=m_{-r'}\circ a_2$.  Then the last equation becomes
\[
b\circ a_1=\id{\d}, \qquad b\circ m_r\circ a_1= m_{-r'}.
\]
Since $m_{-r'}$ and $a_2$ both fix $1$, so does the automorphism $b$, and we may therefore apply the first alternative, with $a_2$ replaced by $b$, to deduce that $r=-r'$  Since $r,r'>0$, this cannot happen.  Thus the first alternative \eqref{alt1} holds and $r=r'$.\\

(3) Suppose that $\cald\sim g_r(\d)$ for some $r\in(0,1)$.  Since $B_r$ has the fixed points $1$ and $-1$, $g_r$ has the royal nodes $1,-1$ and $g(\d)^-$ meets $\calr^-$ in the royal points $(\pm 2,1)$.  By Theorem \ref{geothm30}, $g(\d)$ and consequently $\cald$ are purely balanced geodesics.

Conversely, suppose that $\cald$ is purely balanced.  By Theorem \ref{geothm30}, $\cald^-\cap\calr^-$ consists of a pair of points in $\partial\calr$, and we may take these points to be $(2,1)$ and $(-2,1)$.  Choose a complex C-geodesic $k$ such that $\cald=k(\d)$.  Since $\cald\cap\calr=\emptyset$, by Proposition \ref{prop3.10} we can assume that $k(z)=(z+m(z), zm(z))$ for some $m\in \aut\d$.  The fixed points of $m$ are the royal nodes of $k$, which are $\pm 1$.  It follows that $m=B_r$ and hence $k=g_r$ for some nonzero $r\in (-1,1)$.  Since
\[
g_r\circ B_{-r}=g_{-r},
\]
we conclude that $\cald\sim g_r(\d)$ for some $r$ in $(0,1)$.

Suppose that $g_r(\d) \sim g_{r'}(\d)$.  Then there exist $a_1,a_2\in\aut\d$ such that
\be\label{grr'}
g_{r'}=\tilde a_2\circ g_r\circ a_1,
\ee
and moreover $a_1,a_2$ map the set $\{-1,1\}$ of royal nodes of $g_r$ and $g_{r'}$ to itself.
Hence  either $a_1$ or $-a_1$ fixes both $1$ and $-1$, and so $a_1=\pm B_\ga$ for some $\ga\in (-1,1)$.  Then $a_1(\ga)=0$.

It follows from equation \eqref{grr'} that either
\[
a_2\circ a_1=\id{\d} \quad \mbox{ and } \quad a_2\circ B_r \circ a_1 = B_{r'}
\]
or
\be\label{secalt}
a_2\circ a_1= B_{r'}  \quad \mbox{ and } \quad a_2\circ B_{r}\circ a_1=\id{\d}.
\ee
Consider the first alternative.  Then $a_2=a_1^{-1}$ and $B_{-r}\circ a_1= a_1\circ B_{-r'}$.
Since
\[
B_{-r}\circ a_1(\ga)=B_{-r}(0)= r,
\]
we have
\[
\pm B_\ga\circ B_{-r'}(\ga) = r.
\]
Since $r'$ and $\ga$ are real, $B_{-r'}(\ga)=B_{-\ga}(r')$, and so
\[
r=\pm B_\ga\circ B_{-\ga}(r') = \pm r'.
\]
Since $r,r' >0$, we have $r=r'$.

If the second alternative \eqref{secalt} holds, then
\[
(B_{-r'}\circ a_2)\circ a_1= \id{\d} \quad \mbox{ and }\quad (B_{-r'}\circ a_2)\circ B_r \circ a_1=B_{-r'},
\]
and so, by the first alternative, $r=-r'$.  Again this contradicts $r,r'>0$, and so we deduce that $r=r'$.\\

(4) and (5) follow from Theorem \ref{autosG}. 
\end{proof}
\begin{remark}
By Definition \ref{extdef10}, for a \nd datum $\de$, there is a unique solution of $\Car(\de)$ among the $\Phi_\omega$ if and only if $\de$ is  either purely unbalanced or exceptional.  Hence the complex C-geodesics $k_r$ and $h_r$ described in (1) and (2) of Theorem \ref{formgeos} have unique holomorphic left inverses (modulo $\aut\d$) of the form $\Phi_\omega$.  In fact more is true:  Kosi\'nski and Zwonek \cite[Theorem 5.3]{kz2013} show that the $k_r$ and $h_r$  have unique holomorphic left inverses (modulo $\aut\d$) of any form.
\end{remark}

\chapter{Balanced geodesics in $G$} \label{BalGeo}
In studying sets $V$ with the \npep \ in $G$ we shall exploit intersections of $V$ with balanced and flat geodesics.
Recall that a datum is defined to be balanced if it is exceptional, purely balanced or royal.   We shall say that a complex geodesic is {\em balanced} if it is contacted by a \nd balanced datum.

One characterization of balanced datums is given in Proposition \ref{prop3.9} above; 
another is afforded by Proposition \ref{prop3.12} below.
The following is a simple observation. 
\begin{lemma}\label{m1=m2} Let $\delta$ be a \nd datum in $G$ and let $k$ be a solution to $\Kob(\delta)$ such that
\[
k = (m_1+m_2,m_1m_2)
\]
where $m_1, m_2\in\aut \d$.
Then the royal nodes of $k$ are the solutions of $m_1=m_2$ in $\d^-$ and the roots of this equation in $\c$ are symmetric with respect to $\t$.
\end{lemma}

\begin{proposition}\label{prop3.12}
If $\delta$ is a \nd datum in $G$ and $k$ is a solution to $\Kob(\delta)$, then $\delta$ is balanced if and only if $k$ factors through the bidisc, that is, there exists $f\in \d^2(\d)$ such that $k = \pi \circ f$.
 Furthermore, if $f=(m_1,m_2)$ and $k=\pi\circ f$  then $m_1,m_2\in\aut\d$ and exactly one of the following statements holds:
\begin{enumerate}[\rm (1)]
\item $m_1=m_2$ and $\de$ is royal;
\item  the equation $m_1(z)=m_2(z)$ has exactly two solutions in $\t$ and $\de$ is purely balanced;
\item the equation $m_1(z)=m_2(z)$ has exactly one solution in $\t$ (which is of multiplicity two) and $\de$ is exceptional.
\end{enumerate}
\end{proposition}
\begin{proof}
Fix a datum $\delta$ in $G$ and a solution $k$ to $\Kob(\delta)$.   Assume that $\de$ is  balanced. By Theorem \ref{geothm30} either $\cald_\delta \cap \calr = \emptyset$ or $\cald_\delta=\calr$. If $\cald_\delta=\calr$, then there exists $m\in\aut\d$ such that $k(z) = (2m(z),m(z)^2) = (\pi \circ f) (z)$ where $f:\d\to \d^2$ is defined by $f(z)=(m(z),m(z))$. If $\cald_\delta \cap \calr = \emptyset$ then, by Proposition \ref{prop3.10}, there exist $m_1,m_2\in\aut\d$ such that $k=\pi\circ f$ where $f$ is defined by $f(z)=(m_1(z),m_2(z))$.  Thus, in either case, $k$ factors through the bidisc.

To prove the converse, suppose that $f\in \d^2(\d)$ and $k = \pi \circ f$ but that $\de$ is not balanced, so that $\delta$ is either flat or purely unbalanced. Since $k$ is $\Ga$-inner, the components of $f$ are inner.

If $\delta$ is flat, then since $C(s) = s^2$ solves $\Car(\delta)$, $C\circ k \in\aut\d$.  This implies that $f^1f^2$ has degree $1$. Since $f^1$ and $f^2$ are inner, it follows that either $f^1$ or $f^2$ is constant. But then $k=\pi \circ f$ maps $\d$ into $\partial G$, contradicting the fact that $k$ solves $\Kob(\delta)$.

If $\delta$ is purely unbalanced then $\deg f^1f^2 = \deg k^2 = 2$. Hence, since the argument in the previous paragraph implies that neither $f^1$ nor $f^2$ is constant, $\deg f^1=\deg f^2 =1$. Thus $m_1,m_2\in\aut\d$. By Theorem \ref{geothm30}, $k$  has  one royal node in $\t$ and one in $\d$. By Lemma \ref{m1=m2}, the quadratic equation $m_1=m_2$  has a root in $\d$ and another in $\t$, and its roots are symmetric with respect to $\t$.  The equation thus has at least 3 distinct roots, which implies that $m_1$ and $m_2$ coincide, and so that $\de$ is royal.  This is a contradiction.  Hence, if $k$ factors through the bidisc then $\de$ is balanced.

Suppose now that $\de$ is balanced and $k=\pi\circ (m_1,m_2)$, where $m_1,m_2\in\d(\d)$.  Since $k$ is $\Ga$-inner, $m_1$ and $m_2$ are inner, and degree considerations show that $m_1,m_2\in\aut\d$.  Thus $k$ is a rational $\Ga$-inner function of degree $2$ and is a complex C-geodesic of $G$.  By Proposition \ref{georoyal}, $k$ has a royal node in $\t$.

By Lemma \ref{m1=m2}, the  royal nodes of $k$ are the solutions of $m_1=m_2$ in $\d^-$ and the roots of this equation in $\c$ are symmetric with respect to $\t$.
There are exactly three possible cases for the roots of the quadratic equation $m_1=m_2$, given that there is a root in $\t$.

(1) There are infinitely many roots.  Then $m_1$ coincides with $m_2$ and $k(\d) = \royal$.
By Theorem \ref{geothm30},  $\de$ is royal.

(2) The equation $m_1=m_2$ has exactly two solutions in $\t$.
Then $k$ has exactly two royal nodes in $\t$.
By Theorem \ref{geothm30},  $\de$ is purely balanced.

(3) The equation $m_1=m_2$ has exactly one solution in $\t$, which is of multiplicity two. Then
$k$ has a unique royal node, which lies in $\partial\calr$.
By Theorem \ref{geothm30},  $\de$ is exceptional.
\end{proof}

 The term `balanced' in Definition \ref{extdef10} was suggested by a similar notion for the bidisc. 
\begin{definition}\label{defbalD2}
A  {\em balanced disc} in $\d^2$ is a set of the form $\{(w,m(w)): w\in\d\}$ for some $m\in\aut\d$.
\end{definition}
\index{balanced disc}
 According to \cite[Definition 1.17]{agmc_vn}, a \nd discrete datum $\lambda=(\lambda_1,\lambda_2)$ in $\d^2$ is said to be \emph{balanced} if the hyperbolic distance between the first co-ordinates of $\lambda_1$ and $\lambda_2$ is equal to the hyperbolic distance between the second co-ordinates of $\lambda_1$ and $\lambda_2$; otherwise $\la$ is {\em unbalanced}. 
It was noted in \cite{agmc_vn} that, for a \nd discrete datum $\la$ in $\d^2$, the solution to $\Kob(\la)$ is essentially unique if and only if $\la $ is balanced, and in that case the unique geodesic of $\d^2$ that contacts $\la$ is a  {\em balanced disc}.  Any geodesic in $\d^2$ of the form $(m_1,m_2)(\d)$, where $m_1,m_2 \in\aut\d$ is a balanced disc.  The following lifting property of certain geodesics in $G$ will be exploited in Section \ref{Gnpep}.

\begin{proposition}\label{prop3.13} 
Let $\delta$ be a \nd datum in $G$ that is either balanced or flat.  There exists a balanced disc $D$ in $\d^2$ such that  $\cald_\de=\pi(D)$.
\end{proposition}

\begin{proof}
Suppose that $\de$ is a balanced datum.  By Proposition \ref{prop3.12}, $\cald_\de$ is of the form $\pi(D)$ where $D=(m_1,m_2)(\d)$ for some $m_1,m_2 \in\aut\d$.  This $D$ is a balanced disc in $\d^2$.

Consider the case of a flat datum $\de$.
Note that
\[
\{0\}\times \d = \calf_0
\]
is a flat geodesic.  If $D= \{(z,-z): z\in\d\}$ then $D$ is a balanced disc in $\d^2$, and 
\[
\pi(D)=\{(0,-z^2): z\in\d\} = \calf_0.
\]
Now consider the general flat geodesic
$\calf_\beta$, where $\beta\in\d$.
According to Theorem \ref{autosG}, there exists $m\in\aut\d$ such that $\tilde m(\calf_0)=\calf_\beta$, where $\tilde m$ is the automorphism of $G$ given by
\[
\tilde m(z+w,zw)=(m(z)+m(w),m(z)m(w)).
\]
Let
\[
D_m=\{(m(z),m(-z)): z\in\d\}.
\]
Then $D_m$ is a balanced disc in $\d^2$, and
\[
\pi(D_m)=\{ (m(z)+m(-z),m(z)m(-z)): z\in\d\} = \tilde m(\pi(D))= \tilde m(\calf_0)=\calf_\beta.
\]
\end{proof}

Note that the image under $\pi$ of a balanced datum in $\d^2$ need not be a balanced datum in $G$ -- it can be flat.

\chapter[Geodesics and sets $V$ with the extension property]{Geodesics and sets $V$ with the norm-preserving extension property in $G$}\label{Gnpep}

We say that a set $A\subseteq \c^2$ is an \emph{algebraic set} if there exists a set $S$ of polynomials in two variables such that
\[
A = \set{\lambda \in \c^2}{ p(\lambda)=0 \mbox{ for all } p\in S}.
\]
\index{set!algebraic}
We say that a set $V\subseteq \c^2$ is an \emph{algebraic set in $G$} if there exists an algebraic set $A$ in $\c^2$ such that $V=A\cap G$.  If $V$ is an algebraic set in $G$ and $f$ is a complex valued function defined on $V$, then we say that \emph{$f$ is holomorphic on $V$} if $f$ can be \emph{locally} extended to a holomorphic function on $\c^2$, that is,  if for each $\lambda \in V$ there exists a neighborhood $U$ of $\lambda$ in $\c^2$ and a holomorphic function $F$ defined on $U$ such that $F(\mu) =f(\mu)$ for all $\mu \in V \cap U$.    

We recall Definition \ref{npep} from the introduction: a subset $V$ of $G$ has the norm-preserving extension property if every bounded holomorphic function $f$ on $V$ has a holomorphic extension to $G$ of the same supremum norm.  An immediate consequence of the definition is the following.
\begin{proposition}\label{connect}
Any subset of $G$ that has the \npep \ is connected.
\end{proposition}
\begin{proof}
If not, choose two distinct components $V_1$ and $V_2$ of $V$.  The norm-preserving extension property of $V$ ensures the existence of a function $F\in\c(G)$ such that $\sup_G |F|=1$,  $F=0$ on $V_1$ and $F=1$ on $V_2$, contrary to the Maximum Principle.
\end{proof}

The following is a slight strengthening of an observation in the introduction.
 \begin{proposition}\label{extprop10a}
 If $R$ is a retract in $G$ then $R$ is an algebraic set in $G$ having the norm-preserving extension property.
 \end{proposition}
\begin{proof}
By Theorem \ref{retthm10} $R$ is a geodesic.  Hence, by Proposition \ref{retprop10}, $R$ is an algebraic set.
If $\rho$ is a retraction of $G$ with range $R$ and $f$ is a bounded holomorphic function on $V$ then $f\circ\rho$ is a holomorphic extension of $f$ to $G$ with the same supremum norm as $f$.
\end{proof}

We initially hoped to prove the converse of Proposition \ref{extprop10a} for $G$.  After all, in the bidisc it {\em is} true that sets with the \npep \  are retracts \cite{agmc_vn}. 
However, it transpires that there are algebraic sets with the \npep \ in $G$ which are not retracts; we call them  {\em anomalous} sets.  They are described in Lemma \ref{extlem40} below.  
\index{set!anomalous}

This chapter investigates some of the geometric consequences of the norm-preserving extension property.  The argument will eventually culminate in Lemma \ref{extlem41}, to the effect that if $V$ is an algebraic set in $G$ that has the norm-preserving extension property and if $V$ is not anomalous then $V$ is a properly embedded finitely connected Riemann surface.

 \section{$V$ and $\Car (\delta)$}\label{V&card}
In this section we prove a technical result about solutions of the Carath\'eodory problem for datums that contact a set $V$ having the \npep.
 \begin{lemma}\label{extlem8}
Let $w_0\in \d$ and let $D=\set{z\in\d}{|z-w_0|>r}$ for some  $r$ such that $0<r<1-|w_0|$.
 For any \nd datum $\de$ in $D$ there exists a holomorphic map $\beta:D\to\d$ such that
\be\label{gtth}
|\beta(\de)|_\d > |\de|.
\ee
 \end{lemma}
We use the notation $|\cdot|_\d$ to emphasize that $\beta(\de)$ is being regarded as a datum in $\d$.
\begin{proof}
We can assume that $w_0=0$.   Consider first the case of a \nd discrete datum $\de=(z_1,z_2)$.  We shall choose $t>0$ and $g\in\c(D)$ such that $g(z_1)=g(z_2)=0$ and the function
$z-tg(z)$
maps $D$ into a disc of radius $R<1$; then the function
\[
\beta(z)=R^{-1}(z-tg(z))
\]
will have the desired properties.

Let
\[
\eps= \frac{1}{2 (1+|z_1|)(1+|z_2|)}.
\]
Consider the  Laurent expansion of the rational function $z/((z-z_1)(z-z_2))$, which is analytic in an annulus containing $\t$.  A suitable partial sum of the expansion is a rational function $h \in \c(D)$ such that
\be\label{chooseh}
\left| h(z) - \frac{z}{(z-z_1)(z-z_2)} \right| <\eps
\ee
for all $z\in \t$, and the only poles of $h$ are at $0$ and $\infty$.  Let
\be\label{gotg}
g(z)=(z-z_1)(z-z_2) h(z) \qquad\mbox{ for } z\in D^-.
\ee
$g$ is analytic in $D$ and $g(z_1)=g(z_2)=0$.  Multiply the inequality \eqref{chooseh} by $|(z-z_1)(z-z_2)|$ to deduce that, for $z\in\t$,
\begin{align*}
|\bar z g(z) - 1| &< \eps \sup_{|z|=1}\left| (z-z_1)(z-z_2) \right| \\
	&\leq \eps (1+|z_1|)(1+|z_2|) \\
	&=\half.
\end{align*}
Hence, for $z\in\t$,
\[
\re (\bar z g(z)) > \half.
\]
Consequently, for $z\in\t$,
\begin{align}\label{outbound}
|z-tg(z)|^2 &= 1-2t\re (\bar zg(z)) +t^2 |g(z)|^2 \notag \\
	&<  1- t +t^2 \sup_\t |g|^2.
\end{align}
Let $M= \sup_{|z|=r} |g(z)|$.  For $|z|=r$ we have 
\[
|z-tg(z)| \leq r+t M.
\]
Hence we can choose $R\in(0,1)$ and $t$ so small that 
 $|z-tg(z)| < R< 1$ for all $z$ in both the inner and outer bounding circles of $D$, and so, by the Maximum Principle, for all $z$ in $D$.
Consideration of the points $z_1,z_2$ shows that  $R>\max\{|z_1|,|z_2|\}$, and it follows that
\[
|\beta(\de)|_\d = |(\beta(z_1),\beta(z_2))|= \left|\left(\frac{z_1}{R},\frac{z_2}{R}\right)\right| > |(z_1,z_2)|= |\de|.
\]

Now consider the \nd infinitesimal datum $(z_1;1)$ in $D$.  The proof is similar: construct $t>0$ and $g\in\c(\d)$ such that $g(z_1)=g'(z_1)=0$ and the function $z-tg(z)$ maps $D$ into a disc of radius $R<1$.  The rest of the construction is as in the discrete case, with $z_2=z_1$.
\end{proof}
 \begin{lemma}\label{extlem10}
Let $V \subseteq G$ have the norm-preserving extension property and let $\delta$ be a \nd datum in $G$ that contacts $V$.  If $C$ is a solution to $\Car(\delta)$ then $C(V)$ is dense in $\d$.
 \end{lemma}
 \begin{proof}
 Let $C$ be a solution to $\Car(\delta)$ and let $\zeta= C(\delta)$.   If $C(V)$ is not dense in $\d$ then there exist $w_0\in \d$ and $r$ satisfying $0<r<1-|w_0|$ such that if $D$ is defined to be $\set{z\in\d}{|z-w_0|>r}$, then $C(V) \subseteq D$ and $\zeta$ is a \nd datum in $D$. By Lemma  \ref{extlem8} there exists a holomorphic map $\beta:D\to\d$ such that $|\beta(\zeta)| > |\zeta|$.   The map $f_0=\beta \circ (C|V)$ is well defined and holomorphic from $V$ to $\d$. Consequently, since $V$ has the norm-preserving extension property, there exists a holomorphic map $F:G \to \d^-$ satisfying $F|V=f_0$. By the Open Mapping Theorem, $F(G) \subset \d$,  and so $F$ is a candidate for $\Car(\delta)$. But $C$ is a solution to $\Car (\delta)$ and
 \begin{align*}
|F(\de)| &=|f_0(\de)| \qquad \mbox{ by Remark \ref{F1=F2}}\\
&=|(\beta \circ C) (\de)|\\
&=|(\beta(C(\de))| \qquad \mbox{ by equation  \eqref{chainrule}}\\
&= |\beta(\zeta)| \\
&>|\zeta|\\
&=|C(\de)|\\
&= \car{\de}.
\end{align*}
This contradiction shows that $C(V)$ is dense in $\d$.
 \end{proof}

 \section{$V$ and balanced datums}\label{V&bal}
Here is another technical result concerning datums that contact a set with the \npep.
 \begin{lemma}\label{extlem20}
 Let $V$ be an algebraic set in  $G$ that has the norm-preserving extension property. Let $\delta$ be a \nd datum in $G$ such that there exist distinct points $\omega_1, \omega_2 \in\t$ for which $\Phi_{\omega_1}, \Phi_{\omega_2}$ solve $\Car(\de)$ (so that $\de$ is either flat, royal or purely balanced).  If $\de$ contacts $V$ then $\cald_\delta \subseteq V$. 
\end{lemma}
\begin{proof} Pick $\omega_1,\, \omega_2 \in \t$ where $\omega_1\ne \omega_2$ and $\Phi_{\omega_1}$ and $\Phi_{\omega_2}$ both solve $\Car (\delta)$. Fix any solution $k$ to $\Kob(\delta)$, choose $m_1, \, m_2\in\aut\d$ such that
\[
m_1 \circ \Phi_{\omega_1} \circ k = \id{\d}\qquad  \text{ and }\qquad  m_2 \circ  \Phi_{\omega_2} \circ k = \id{\d},
\]
and then let $C_1$ and $C_2$ be defined by
\[
C_1=m_1 \circ \Phi_{\omega_1}\qquad \text{ and }\qquad C_2=m_2 \circ \Phi_{\omega_2};
\]
then
\[
C_1\circ k = \id{\d} = C_2\circ k.
\]
Define $C$ by 
\[
C = \half C_1 + \half C_2.
\]
\begin{claim}\label{norclaim10}
If $\tau \in \t$,  $\tau \not= C_2(2\bar\omega_1,\bar\omega_1^2)$, and
$\tau \not= C_1(2\bar\omega_2,\bar\omega_2^2)$, then there exists $s_0 \in \cald_\delta^- \cap V^-$ such that $s_0\not=(2\bar\omega_1,\bar\omega_1^2)$, $s_0\not=(2\bar\omega_2,\bar\omega_2^2)$, and 
\[
C_1(s_0)=C_2(s_0)=\tau.
\]
\end{claim}

 To prove this claim, fix $\tau \in \t \setminus \{C_2(2\bar\omega_1,\bar\omega_1^2),
C_1(2\bar\omega_2,\bar\omega_2^2)\}$.  By construction, $C$ solves $\Car(\de)$.
Therefore, by Lemma \ref{extlem10} there exists a sequence of points $\{s_n\}$ in $V$ such that $C(s_n) \to \tau$.     Passing to a subsequence if necessary we may further assume that $s_n\to s_0 \in \partial G$. Since $C(s_n) \to \tau$, it follows that both $C_1(s_n) \to \tau$ and $C_2(s_n) \to \tau$. But then $C_1(s_n)-C_2(s_n) \to 0$. Consequently, by Lemma \ref{retlem10}, it follows that $P_\delta(s_n) \to 0$. Since $P_\delta$ is continuous, it follows that $P_\delta(s_0)=0$, that is, that $s_0 \in \cald_\delta^-$. Since it is also the case that $s_0 \in V^-$, it follows that $s_0 \in \cald_\delta^- \cap V^-$.

To see that $s_0\not=(2\bar\omega_1,\bar\omega_1^2)$, assume to the contrary that $s_0=(2\bar\omega_1,\bar\omega_1^2)$. Since $C_2=m_2\circ \Phi_{\omega_2}$ is continuous on a neighborhood of $(2\bar\omega_1,\bar\omega_1^2)$, it follows that
\[
C_2(2\bar\omega_1,\bar\omega_1^2)=C_2(s_0)=\lim_{n\to \infty}C_2(s_n)=\tau,
\]
contradicting the assumption that $\tau \not=C_2(2\bar\omega_1,\bar\omega_1^2)$. That $s_0\not=(2\bar\omega_2,\bar\omega_2^2)$ follows in similar fashion. This completes the proof of Claim \ref{norclaim10}.

Now fix $\tau \in \t$ with $\tau \not= C_2(2\bar\omega_1,\bar\omega_1^2)$ and
$\tau \neq C_1(2\bar\omega_2,\bar\omega_2^2)$. Let $s_0$ be as in the claim. Since $s_0 \in \cald_\delta^-$, there exists $\eta \in \t$ such that $s_0=k(\eta)$. But then
\[
\tau = C_1(s_0) = C_1(k(\eta)) = (C_1 \circ k) (\eta) = \eta,
\]
so that
\[
k(\tau) = k(\eta) =s_0 \in V^-.
\]
This proves that
\[
k(\t \setminus \{C_2(2\bar\omega_1,\bar\omega_1^2),
C_1(2\bar\omega_2,\bar\omega_2^2)\}) \subseteq V^-.
\]
By continuity, $k(\t) \subseteq V^-$.  Since $V$ is algebraic there is a set $S$ of polynomials such that
\[
V=G\cap \{s: p(s)=0 \mbox{ for all } p\in S\}.
\]
Hence $p\circ k=0$ on $\t$ for all $p\in S$, and by the Maximum Principle, $p\circ k=0$ on $\d$ for all $p\in S$.   Thus 
 $\cald_\delta = k(\d) \subseteq V$, as was to be proved.
\end{proof}

\begin{remark} \rm
The above proof remains valid if $V$, rather than being assumed to be an algebraic set, is only assumed to be relatively polynomially convex.
\end{remark}

\section{$V$ and flat or royal datums}\label{V&flatorroy}
In this section we show that if a set $V$ having the \npep \ contacts either a flat datum or a royal datum then $V$ has one of three concrete forms.
To do this we first prove a result about the closest point in $G$ to a flat or royal geodesic.
\begin{lemma}\label{lem5.11}
Let $\delta$ be a \nd datum in $G$ such that $\Phi_\omega$ solves $\Car(\de)$ for all $\omega\in\t$ (so that $\de$ is flat or royal).  

(i) If $t \in G \setminus \cald_\delta$ then
\[
\inf_{s \in \cald_\delta} \car{(s,t)}
\]
is attained at some point $s_0\in \cald_\delta$. 
Furthermore,

(ii) there are at least two values of $\omega\in\t$ such that $\Phi_\omega$ solves $\Car((s_0,t))$ (so that the datum $(s_0,t)$ is either flat, royal or purely balanced);

(iii) if  $\cald_\de$ is flat and $t \in \calr$ then $s_0 \in \calr$;

(iv) if  $\cald_\de$ is flat and $t \notin \calr$ then the datum 
$(s_0,t)$ is purely balanced. 
\end{lemma}
\begin{proof}

(i)    Let $k$ solve $\Kob(\de)$ and let $\omega_1 \in \t$ be such that $\Phi_{\omega_1}\circ k=\id{\d}$.  The function $\ph$ on $\cald_\de$ defined by
\[
\ph(s)=\car{(s,t)} 
\]
is continuous and $\ran \ph \subseteq (0,1)$.  Note that
\begin{align}
\ph(s)&=\car{(s,t)} = \sup_{\omega \in \t} \left|(\Phi_{\omega}(s), \Phi_{\omega}(t)) \right| \notag\\
	& \ge \left|(\Phi_{\omega_1}(s), \Phi_{\omega_1}(t) )\right|.  \label{phi_ge_dzw}
\end{align}
Let $w_1= \Phi_{\omega_1}(t)$ and let $s = k(z)$,  $ z \in \d$.  We have
\begin{align}
\ph(s) & \ge \left|(\Phi_{\omega_1} \circ k(z), w_1)\right| \notag\\
	& = \left|(z, w_1 )\right| .  \label{phi_ge_dzw_2}
\end{align}
Therefore
\begin{align}
1-\ph(s)^2 & \le 1 - \left|(z, w_1 )\right|^2 \notag\\
 & = \frac{(1 -|z|^2)(1-|w_1|^2)}{|1-\bar{w_1} z|^2}\notag\\
 & \le  (1 -|z|^2) \frac{1+|w_1|}{1- |w_1|} . \label{phi_ge_dzw_3}
\end{align}
It is clear  that $\ph(s) \to 1$ as $s \to \partial G$ in $\cald_\de$.  
Hence the function $\tilde{\ph}$ defined by
\[
\tilde{\ph}(s)=\left\{\begin{array}{lcl}\ph(s) & \text{if} & s \in  \cald_\de\\
1 &\text{if} & s \in   \partial G\cap\cald_\de^-\\
\end{array}\right.
\]
is continuous  on the closure $\overline{\cald_\de}$ of $\cald_\de$ in $G$, and $\ran \tilde\ph \subseteq (0,1]$.  Note that $k( \d^-) = \overline{\cald_\de}$, and so  $\overline{\cald_\de}$ is compact.
Hence $\tilde\ph$ attains its infimum at some point $s_0\in \overline{\cald_\de}$.  It is clear that $s_0$ is not in $k(\t)$, since $\tilde\ph(\d) \subseteq (0,1)$  and $|\tilde\ph(s)| =1$ for all $s \in k(\t)$. Thus $s_0\in \cald_\de$.

(ii) Suppose that there exists a unique $\omega_0\in\t$ such that
\be\label{a}
|\Phi_{\omega_0}((s_0,t))| = \sup_{\omega \in \t} |\Phi_\omega((s_0,t))|=   \car{(s_0,t)}.
\ee
Let $k$ be a solution to $\Kob(\delta)$ 
and let $s_0 =k(z_0)$. 

Since $\delta$ is assumed to be flat or royal, for each $\omega \in \t$, $\Phi_{\omega}\circ k$ is in $\aut\d$ and therefore has no critical points.  For any direction $w\in\c$,
\begin{align}
 D_w\  |\Phi_{\omega_0}((k(z),t))|^2 \big|_{z=z_0}&=D_w\  |(\Phi_{\omega_0}\circ k(z),\Phi_{\omega_0}(t))|^2\ \big|_{z=z_0} \notag\\
	&=\frac{d}{dy}|B\circ\Phi_{\omega_0}\circ k (z_0+y w)|^2\Big|_{y=0}\notag\\
	&=2 \re \{ w (B\circ\Phi_{\omega_0}\circ k )'(z_0) \overline{B\circ\Phi_{\omega_0}\circ k (z_0)}\}, \label{Dw}
\end{align}
where
\[
B(z)=\frac{z- \Phi_{\omega_0}(t)}{1-\overline{\Phi_{\omega_0}(t)} z}.
\]
Now
\[
|B\circ\Phi_{\omega_0}\circ k(z_0)|= | B\circ \Phi_{\omega_0}(s_0)| = \car{(s_0,t)} \neq 0.
\]
Since, further, $B'$ and $(\Phi_{\omega_0}\circ k)'$ are never zero on $\d$, 
\[
(B\circ\Phi_{\omega_0}\circ k )'(z_0) \overline{B\circ\Phi_{\omega_0}\circ k (z_0)} \neq 0.
\]
By equation \eqref{Dw},
there is some direction $w_0$ such that  $ D_{w_0}|\Phi_{\omega_0}((k(z),t))|^2 \big|_{z=z_0}$ is nonzero.  After scaling $w_0$,  we can assume that 
\be\label{Dw1} 
 D_{w_0}|\Phi_{\omega_0}((k(z),t))|^2  \big|_{z=z_0} =1.
\ee

For $x\in [-\pi,\pi]$ and $y$ in some sufficiently small neighborhood $I$ of $0$, let
\[
f(x,y) = \left| \Phi_{\e^{ix}\omega_0}\left((k(z_0+yw_0),t)\right) \right|^2.
\]
Choose a sequence $(y_j)$ in $I$, with $y_j < 0$, such that $y_j \to 0$.

\begin{claim}  There is a sequence $(x_j)$ in $[-\pi,\pi]$ such that, for each $j$, $f(x, y_j)$ attains its maximum at $x=x_j$, {and} $x_j\to 0$.
\end{claim}
Indeed, suppose not.  Choose, for each $j$, any point $x_j$ at which $f(x,y_j)$ attains its maximum over $[-\pi,\pi]$, then pass to
a subsequence such that $x_j$ converges to some point $x_0$.  By supposition, $x_0\neq 0$, and so $\e^{ix_0}\omega_0 \neq \omega_0$.  We have $k(z_0+y_jw_0) \to s_0$, and therefore
\begin{align}
f(x_j,y_j) &= \max_{-\pi\leq x\leq \pi} f(x, y_j)\notag \\
	&= \max_{-\pi\leq x\leq \pi}   \left| \Phi_{\e^{ix}\omega_0}\left((k(z_0+y_jw_0),t)\right) \right|^2  \label{keepthis} \\
	&= \car{(k(z_0+y_jw_0),t)}^2 \notag \\
	&\to  \car{(s_0,t)}^2. \notag
\end{align}
Hence
\[
f(x_0,0)= \car{(s_0,t)}^2.
\]
Thus $\omega_1= \e^{ix_0}\omega_0$ is a second value of $\omega$ for which $\Phi_\omega$ solves $\Car((s_0,t))$, contrary to hypothesis.  This establishes the Claim.

The function $f$ has the following properties.
\begin{enumerate}[\rm (1)]
\item $f$ is real analytic on $[-\pi,\pi]\times I$;
\item $\frac{\partial f}{\partial y}(0,0)=1$;
\item $f(x,0)$ has a unique global maximum over $[-\pi,\pi]$, which is at $x=0$.
\end{enumerate}
Consider the power series expansion of $f$ about $(0,0)$.  In view of statement $(3)$ there exists $n\geq 1$ and $\rho > 0$ such that 
\[
\frac{d^{2n}}{dx^{2n}} f(x,0)\Big|_{x=0} = -\rho
\]
is the first nonzero derivative of $f(x,0)$ at $x=0$.  There also exist analytic functions $\ph(x), \psi(x,y)$ such that $\psi(0,0)=0$ and
\be\label{series}
f(x,y) = \car{(s_0,t)}^2 + y -\rho x^{2n}+ x^{2n+1}\ph(x) + y\psi(x,y)
\ee
in a neighborhood of $(0,0)$.   Consequently
\[
\frac{\partial f}{\partial  x}(x,y)= -2n\rho x^{2n-1} + (2n+1)x^{2n}\ph(x) + x^{2n+1}\ph'(x)+y\frac{\partial \psi}{\partial x}(x,y)
\]
near  $(0,0)$.  As $x_j$ is a maximiser for $f(x,y_j)$, $x_j$ is a critical point of $f(x, y_j)$.  That is,
\[
0= -2n\rho x_j^{2n-1} +(2n+1) x_j^{2n}\ph(x_j) + x_j^{2n+1} \ph'(x_j)+y_j\frac{\partial \psi}{\partial x}(x_j,y_j).
\]
It follows that
\[
x_j^{2n-1}= O(y_j).
\]
Since $x_j \to 0$,
\[
-\rho x^{2n}_j + x_j^{2n+1}\ph(x_j) = x_j^{2n-1}(-\rho x_j+x_j^2\ph(x_j))=o(y_j),
\]
and, since $\psi(0,0)=0$,
\[
y_j\psi(x_j,y_j)=o(y_j).
\]
  Hence the series expansion \eqref{series} yields
\[
f(x_j,y_j) =\car{(s_0,t)}^2 + y_j + o(y_j),
\]
and therefore (since $y_j < 0$) $f(x_j,y_j) < \car{(s_0,t)}^2$ for some $j$.

However, for all $j$ we have, by the calculation \eqref{keepthis},
\begin{align*}
f(x_j,y_j) &=  \car{(k(z_0+y_jw_0),t)}^2\\
	&\geq \car{(s_0,t)}^2
\end{align*}
by choice of $s_0$ as a closest point in $\cald_\de$ to $t$.  This is a contradiction, and so part (i) of
the lemma is proved.

(iii) Suppose  $t \in \calr$, say $t = (2w, w^2)$ for some $w \in \d$.
Without loss of generality we may assume that $\cald_\de$ is the flat geodesic  $\calf_0 = \{(0,z): z\in \d\}$.  This is because the automorphisms of $G$ fix $\calr$, permute the flat geodesics (by Proposition \ref{autosG}) and preserve the Carath\'eodory distance.  We claim that $(0,0)$ is the closest point in $\calf_0$ to $t$.

We must show that, for every $z \in \d$,
\[
|(t, (0,0))|_{\rm car} \le |(t, (0,z))|_{\rm car}. 
\]
Now
\begin{align*}
|(t, (0,0))|_{\rm car}&= \sup_{\omega \in \t} \left|\left(\Phi_\omega (2w,w^2),\Phi_\omega (0,0) \right)\right|\\
	&= \sup_{\omega \in \t} |(-w,0)|\\
         &=|w|.
\end{align*}
Furthermore
\begin{align*}
|(t, (0,z))|_{\rm car}&= \sup_{\omega \in \t} \left|\left(\Phi_\omega (2w,w^2),\Phi_\omega (0,z) \right)\right|\\
	&= \sup_{\omega \in \t} |(-w,\omega z)|\\
         &= \sup_{\omega \in \t} \left|\frac{\omega z + w} {1 + \bar{w}\omega z}\right|.
\end{align*}
Suppose $z \neq 0$. Choose $ \omega \in \t$ and  $r >0$, such that $ \omega z = r w$. Then 
\begin{align*}
|(t, (0,z))|_{\rm car}& \ge \left|\frac{\omega z + w} {1 + \bar{w}\omega z}\right|\\
&=\left|\frac{(1+r) w} {1 + r |w|^2} \right|\\
& > |w|\\
&= |(t, (0,0))|_{\rm car}.
\end{align*}
Hence $(0,0)$ is a closest point to $t$ in $\cald_\de$ (in fact it is the unique closest point). Thus $s_0 \in \calr$.\\

 (iv) Since $t \notin \calr$, by (ii) the datum $(s_0,t)$ is either flat or purely balanced. If  the datum $(s_0,t)$ is  flat then
the $\cald_\delta = \cald_{(s_0,t)}$ and so $t \in \cald_\delta$. This is a contradiction to the assumption that $t \in G \setminus \cald_\delta$.
Therefore the datum $(s_0,t)$ is  purely balanced.

\end{proof}

\begin{lemma}\label{extlem40}
Let $V$ be an algebraic set in $ G$ having the norm-preserving extension property.  If $\delta$ is a flat datum in $G$ that contacts $V$ then either $V=\cald_\delta$,  $V=\calr\cup\cald_\de $  or $V=G$.
\end{lemma}
\begin{proof}
Since $\de$ contacts $V$,  it follows from Lemma \ref{extlem20} that $\cald_\de \subseteq V$.

Suppose that $V \neq \cald_\de$. Then there exists a point $v_0 \in V$ such that $v_0\notin \cald_\de$. 
Consider two cases: (1)  $v_0 \notin \calr$,
and (2)  $v_0\in\calr$.

Case (1): $v_0\notin \calr$.   By Lemma \ref{lem5.11} there exists a point $s_0 \in \cald_\de$ which is the closest point of $\cald_\de$ to 
 $v_0$
and a pair $\omega_1,\omega_2$ of distinct points in $\t$ such that both $\Phi_{\omega_1}$ and $\Phi_{\omega_2}$ solve $\Car(\gamma)$, where $\gamma = (s_0,v_0)$. According to Definition \ref{extdef10} the datum $\gamma$
is either flat,  purely balanced or royal.  Now $\ga$ is not royal, since $v_0\notin\calr$.  Nor is $\ga$ flat; indeed, if  $\gamma$ is flat, then $\cald_\ga$ and $\cald_\de$ are flat geodesics through $s_0$ and hence coincide, contrary to the fact that $v_0 \in \cald_\ga \setminus \cald_\de$.
Hence $\ga$ is purely balanced, which is to say that there are exactly two values of $\omega$ for which $\Phi_\omega$ solves $\Car(\ga)$.    Consequently, by Lemma \ref{extlem20}, $\cald_\gamma \subseteq  V$.  
 Corollary \ref{prop3.13}  now  implies that there is a balanced disc $D$ in the bidisc such that $\pi(D)=\cald_{\gamma}$.

Consider the case that $\cald_\de$ is the flat geodesic given by equation \eqref{defFbeta} with $\beta=0$, that is,
\begin{align*}
\cald_\de &= \calf_0 =\{(0,z): z\in \d\}.
\end{align*}
By inspection,
\[
	\cald_\de=\pi(B)
\]
where $B$ is the balanced disc in $\d^2$ given by
\be\label{defB}
B=\{(z, -z): z\in\d\}.
\ee
Let $\la_0 =(z_0,-z_0) \in B$ be such that $\pi(\la_0)=s_0$.  We can assume that also $\la_0\in D$, for certainly $D$ contains one of the (at most two) points of $\pi^{-1}(s_0)$, and we may replace $D$ by the balanced disc $\{(\la^2,\la^1): (\la^1,\la^2) \in D\}$ if necessary. Thus $\la_0 \in D\cap B$.

Choose $\mu_0\in D$ such that $\pi(\mu_0)=v_0$.  Since $v_0 \notin \calr$, $\mu_0$ does not lie in the diagonal set $\Delta$ in $\d^2$, defined by $\Delta=\{(z,z):z\in\d\}$.    Since $\la_0, \mu_0$ belong to the balanced disc $D$ in $\d^2$,
\[
|(z_0, \mu_0^1)| = |(-z_0, \mu_0^2)|= |(z_0, -\mu_0^2)|.
\]
The locus of points $z\in\d$ such that
\[
|(z, \mu_0^1)| =  |(z, -\mu_0^2)|
\]
 is a geodesic $(z_t)_{t\in\r}$ through $z_0$ in the Poincar\'e disc (this statement is obvious if $\mu_0^1=r, \ \mu_0^2=-r$ for some $r>0$, and the general case follows upon application of a suitable automorphism of $\d$).  Thus
\[
|(z_t, \mu_0^1)| =  |(z_t, -\mu_0^2)| = |(-z_t, \mu_0^2)|.
\]
Let $\la_t= (z_t, -z_t)$.  Then $(\la_t)$ is a curve in $B$ such that $(\la_t, \mu_0)$ is a balanced pair for all $t$.  Let the unique geodesic  of $\d^2$ passing through $\la_t$ and $\mu_0$ be $D_t$ (see the comment preceding Corollary \ref{prop3.13}).
 $D_t$ is the balanced disc having the form 
\be\label{formDt}
D_t= \{(w,m_t(w)): w\in \d\}
\ee
where $m_t$ is the unique automorphism of  $\d$ such that 
\be\label{mtprop}
m_t(\mu_0^1)=\mu_0^2 \quad \mbox{ and } \quad m_t(z_t)= -z_t.
\ee
Clearly $m_t$ depends continuously on $t\in\r$.

We claim that if $t_1\neq t_2$ then $D_{t_1}\cap D_{t_2}=\{\mu_0\}$.  Suppose, to the contrary, that the intersection contains a point $\nu \neq \mu_0$. 
Then, for some $w_1,w_2\in\d$,
\[
(w_1, m_{t_1}(w_1)) = \nu =(w_2, m_{t_2}(w_2)) \neq \mu_0.
\]
Thus $w_1=w_2$ and so
\[
m_{t_1}(w_1)=  m_{t_2}(w_1)
\]
and either $w_1\neq \mu_0^1$ or $m_{t_1}(w_1) \neq \mu_0^2$.

If $w_1\neq \mu_0^1$ then $m_{t_1}, m_{t_2}$ are automorphisms of $\d$ taking the same values at the distinct points $w_1$ and $\mu_0^1$. Hence $m_{t_1}=m_{t_2}$ and so $D_{t_1}=D_{t_2}$.  Thus $D_{t_1}$ is a balanced disc containing the two distinct points $\la_{t_1}, \la_{t_2}$.  The balanced disc $B$ also contains these two points, and so, by uniqueness, $D_{t_1}=B$.  It follows that $\mu_0 \in B$, and so, by application of $\pi$, we have $v_0 \in \cald_\de$, a contradiction.  We deduce that  $w_1=\mu_0^1$ and $m_{t_1}(w_1) \neq \mu_0^2$.  That is,
$m_{t_1}(\mu_0^1) \neq \mu_0^2$, contrary to equation \eqref{mtprop}.  Hence $D_{t_1}\cap D_{t_2}=\{\mu_0\}$.

  Since $D$ is the balanced disc in $\d^2$ containing $\la_0$ and $\mu_0$,  again by uniqueness, $D_0=D$.  Hence $\pi(D_0)=\cald_{\gamma}$ and so 
\be\label{unexcep}
\pi(D_0) \mbox{  is a purely balanced geodesic in } G.   
\ee
Now $D_0=\{(w,m_0(w)):w\in\d\}$.   By Proposition \ref{prop3.10}, $m_0$ has two distinct fixed points in $\t$.  It follows that there is a neighborhood $I$ of $0$ in $\r$ such that $m_t$ has two distinct fixed points in $\t$ for all $t\in I$.  Then, for $t\in I$, $\pi(D_t)$ is a purely balanced geodesic in $G$ containing the distinct points $v_0=\pi(\mu_0)$ and $\pi(\la_t)$, both of which belong to $V$.  Hence
\[
\pi(D_t) = \cald_{\de_t} \quad \mbox{ where } \de_t = (v_0,\pi(\la_t)).
\]
The \nd discrete datum $\de_t$ contacts $V$.  By  Lemma \ref{extlem20}, $\cald_{\de_t} \subseteq V$ for all $t\in I$.

Now, since $\la_t \in B$, we have $\pi(\la_t)\in \calf_0 \cap \cald_{\de_t}$ for all $t\in\r$.   Hence, for $t\in I$ and $\beta$ in a neighborhood $N$ of zero, 
\be\label{nearlythere}
\calf_\beta \mbox{ meets } \cald_{\de_t} \quad \mbox{ and }\quad \cald_{\de_t} \subset V.
\ee
Since $v_0\notin \calf_0$, we may ensure (replacing $N$ by a smaller neighborhood of $0$ if necessary) that $v_0\notin \calf_\beta$ for all $\beta\in N$.  

Since $\cald_{\de_t}=\pi(D_t)$, statement \eqref{nearlythere} implies that there exists $\la_{\beta t} \in D_t\setminus \{\mu_0\}$ such that $\pi(\la_{\beta t}) \in \calf_\beta \cap \cald_{\de_t}\subseteq V$ for $\beta\in N, t\in I$.  Since distinct $D_t$s meet only at $\mu_0$, for fixed $\beta\in N$, the points $\la_{\beta t}, \ t\in I$, are pairwise distinct.  Since each point of $G$ has at most two preimages in $\d^2$ under $\pi$, it follows that the set $\{\pi(\la_{\beta t}): t\in I\}$ is an uncountable subset of $\calf_\beta \cap V$.

We can now show that $V=G$.

Suppose that $p$ is a polynomial that vanishes on $V$.  Then,  for  $\beta\in N$,  $p$ vanishes at the uncountably many points $\pi(\la_{\beta t}), \  t\in I$, in the one-dimensional disc $\calf_\beta$, and hence $p|\calf_\beta=0$.  Since the union of the discs $\calf_\beta$, for $\beta\in N$, is a neighborhood of zero in $\c^2$, we have $p=0$.  Thus every polynomial that vanishes on $V$ is the zero polynomial.   Therefore, since $V$ is an algebraic set in $G$,  $V=G$.

We have shown that $V=G$ in the case that $\cald_\de$ is the flat geodesic $\calf_0$.   If $D_\de$ is the flat geodesic $\calf_\beta$, for some $\beta\in\d$, then the same conclusion can be deduced from the case $\beta=0$ by the application of an automorphism of $G$, with the aid of Theorem \ref{autosG} in the Appendix.  This concludes the proof of the lemma in Case (1).

Case (2): $v_0\in \calr.$    By Lemma \ref{lem5.11}(ii) the closest point $s_0$ to $v_0$ in $\cald_\de$ lies in $\calr$.  Hence the \nd discrete datum $\ga=(s_0,v_0)$ is royal and contacts $V$.  By Lemma \ref{extlem20}, $\cald_\ga\subseteq V$, that is, $\calr\subseteq V$.  Hence $\calr\cup\cald_\de \subseteq V$.

If $V\neq \calr\cup\cald_\de$ then there exists a point $v_1\in V\setminus (\calr\cup\cald_\de)$.  By Case (1), $V=G$.

We have shown that, if $V\neq\cald_\de$, then in Case (1) $V=G$ and in Case (2) either $V=\calr\cup\cald_\de$ or $V=G$.
\end{proof}

\begin{lemma}\label{extlem45}
Let $V$ be an algebraic set in $G$ having the \npep.
If a royal datum contacts $V$ then either $V=\calr, \, V=\calr\cup\cald_\de$ for some flat datum $\de$ or $V=G$.
\end{lemma}
\begin{proof}
By Lemma \ref{extlem20}, $\calr \subseteq V$.  Suppose that $V\neq \calr$.  There exists a point $t_0\in V\setminus \calr$.  By Lemma \ref{lem5.11} there is a point $s_0\in\calr$ such that the datum $\de=(s_0,t_0)$ is either flat, royal or purely balanced.  Since $t_0 \notin \calr$, $\de$ is not royal.  Neither is $\de$ purely balanced, because $\cald_\de$ meets $\calr$ at the point $s_0\in G$, and Theorem \ref{geothm30}(4) applies.
Hence $\de$ is flat.  Again by Lemma \ref{extlem20}, $\cald_\de\subseteq V$.  Hence, by Lemma \ref{extlem40}, either $V= \calr\cup\cald_\de $ or $V=G$.
\end{proof}

\chapter[Anomalous sets with the extension property]{Anomalous sets $\calr\cup\cald$ with the norm-preserving extension property in $G$}\label{anomalous}

Lemmas \ref{extlem40} and \ref{extlem45} imply that, among sets that contact a flat or royal datum, the only possible anomalous algebraic sets having the \npep \ are sets of the form $\calr\cup\cald$ for some flat complex geodesic $\cald$.
This chapter will be devoted to the proof that such sets do indeed have the \npep.

\begin{theorem}\label{anom.thm10}
If 
 \[
 V=\calr\cup\cald
 \]
 for some flat geodesic $\cald$ of $G$ then $V$ has the norm-preserving extension property.
 \end{theorem}
 Section \ref{def&lem} contains a number of definitions and lemmas to be used in the proof of Theorem \ref{anom.thm10} and the proof of the theorem is executed in Section \ref{proofthm10}.
 \section{Definitions and lemmas}\label{def&lem}
 \begin{definition}\label{anom.def10}
 Let $U$ be a domain and assume that $V \subseteq U$. We say that a function $f:V\to \c$ is \emph{Herglotz on $V$} if $f$ is analytic on $V$ and $\re f(s) >0$ for all $s\in V$. 
\index{Herglotz function}
We say that $V$ has the \emph{Herglotz-preserving extension property} if every function $f$ that is Herglotz on $V$ has an extension to a function that is Herglotz on $U$.
 \end{definition}
\index{extension property!Herglotz-preserving }
 \begin{lemma}\label{anom.lem10}
  Let $U$ be a domain and assume that $V \subseteq U$.  Then $V$ has the Herglotz-preserving extension property if and only if $V$ has the norm-preserving extension property.
 \end{lemma}
 \begin{proof}
 The formula
 \[
 \ph (z) =\frac{1+z}{1-z}
 \]
 defines a conformal automorphism of  the Riemann sphere which maps $\d$ onto the right halfplane $\set{w\in \c}{\re w >0}$.  Consequently, if $V \subseteq U$ and $f$ is an analytic function on $V$, then 
 \[
  f \text{ is Herglotz on } V    \iff |\ph^{-1}\circ f(s)|<1 \mbox{  for all } s\in V.
 \]

 Since this statement holds when $V=U$ as well, the lemma follows immediately.
 \end{proof}
\begin{definition}\label{anom.def20}
If $f$ is a complex valued function on $\calr\cup\calf_0$ we define functions on $\d$ by the formulas
\be\label{anom1}
f_\calf (z) = f(0,z)\quad \text{ and }\qquad f_\calr (z) = f(2z,z^2).
\ee
\end{definition}
\index{$f_\calf$}
\index{$f_\calr$}
\begin{lemma}\label{anom.lem15}
Let $f$ be a complex valued function defined on $\calr\cup\calf_0$.  Then $f$ is analytic on $\calr\cup\calf_0$ if and only if $f_\calf$ and $f_\calr$ are analytic on $\d$ and $f_\calf(0) = f_\calr(0)$.
\end{lemma}
\begin{proof}
Clearly, if $f$ is analytic on $\calr\cup\calf_0$, then $f_\calf$ and $f_\calr$ are analytic on $\d$ and $f_\calf(0) = f_\calr(0)$. Conversely, assume that $f_\calf$ and $f_\calr$ are analytic on $\d$ and
\[
f_\calf(0) =\eta = f_\calr(0).
\]
It is immediate from the first formula in \eqref{anom1} that $f$ is analytic at each point $s=(0,z) \in \calf_0 \setminus \{(0,0)\}$. Likewise, the second formula in \eqref{anom1} implies that $f$ is analytic at each point $s=(2z,z^2) \in \calr\setminus \{(0,0)\}$. To see that $f$ is analytic at $s=(0,0)$, define $F$ by the formula
\[
F(s) = f_\calr\left(\tfrac 12 s^1\right) +f_\calf\left(s^2-\tfrac 14 (s^1)^2\right)-\eta
\]
and observe that $F$ is a holomorphic extension of $f$ to a neighborhood of $(0,0)$.
\end{proof}
\begin{definition}\label{anom.def30}
For $\tau=(\tau_1,\tau_2) \in \t^2$ we define $f_\tau:\calr\cup\calf_0 \to \d$ by the formula
\[
f_\tau (s) = \left \{ \begin{array}{ccl} \tau_1 z  &  \text{ if }  & s=(0,z) \in \calf_0, \\
	 \tau_2 z &  \text{ if } & s=(2z,z^2) \in \calr.  \end{array} \right.
\]
\end{definition}
\begin{lemma}\label{anom.lem20}
The function $f_\tau$ is analytic on $\calr \cup \calf_0 $ for any $\tau\in\t^2$. Furthermore, if a complex-valued function $\Psi_\tau$ is defined on $G$ by the formula
\be\label{anom2}
\Psi_\tau(s) = -\tau_2\Phi_{\omega}(s),\qquad s \in G,
\ee
where $\omega=-\bar\tau_2\tau_1$, then $\Psi_\tau$ is a norm-preserving analytic extension of $f_\tau$ to $G$.
\end{lemma}
\begin{proof}
By Lemma \ref{anom.lem15} $f_\tau$ is analytic on $\calr\cup\calf_0$ and, clearly,
\[
\sup_{s\in \calr\cup\calf_0} |f(s)| = 1.
\]
Also, as $\Phi_\omega$ is a Carath\'eodory extremal function on $G$,
\be\label{anom5}
\sup_{s\in G} |\Psi_\tau(s)| = \sup_{s\in G} |\Phi_\omega(s)| =1.
\ee
It  is a simple to show from Definition \ref{defPhi} that $f_\tau = \Psi_\tau|  (\calr\cup\calf_0)$.  This fact also implies that $f_\tau$ is analytic on $\calr\cup\calf_0$.
\end{proof}
\begin{lemma}[A Herglotz Representation Theorem for $\calr\cup\calf_0$]  \label{anom.lem30}
A function $h:\calr\cup\calf_0\to\c$ is a Herglotz function on $\calr\cup\calf_0$ satisfying $h(0,0) = 1$  if and only if there exists a probability measure $\mu$ on $\t^2$ such that
\be\label{anom10}
h(s) = \int_{\t^2} \frac{1+f_\tau(s)}{1-f_\tau(s)}\ d\mu (\tau)
\ee
for all $s \in \calr\cup\calf_0$.
\end{lemma}
\begin{proof}
Lemma \ref{anom.lem20} asserts that $f_\tau$ is analytic on $\calr\cup\calf_0$  for each $\tau \in \t^2$. Consequently, as $\sup_{\calr\cup\calf_0}|f_\tau| \le 1$ and $f_\tau (0,0) =0$, it follows that the function
\[
s \mapsto  \frac{1+f_\tau(s)}{1-f_\tau(s)}, \qquad s \in \calr\cup\calf_0,
\]
is a Herglotz function on $\calr\cup\calf_0$ that assumes the value 1 at $(0,0)$  for each $\tau \in \t^2$. Hence, if $h$ is defined by equation \eqref{anom10}, then $h$ is Herglotz function on $\calr\cup\calf_0$ satisfying $h(0,0) = 1$.

Now assume that $h$ is a Herglotz function on $\calr\cup\calf_0$ satisfying $h(0,0) = 1$. By Lemma \ref{anom.lem15} it follows that the functions $h_\calf$ and $h_\calr$ are analytic on $\d$. In addition $h_\calf(0) =1$, $h_\calf(0)=1$, and both $h_\calf$ and $h_\calr$ have positive real part on $\d$. Therefore, by the classical Herglotz Representation Theorem there exist probability measures (i.e. finite positive Borel measures with total variation 1) $\mu_\calf$ and $\mu_\calr$ on $\t$ such that
\be\label{anom20}
h_\calf (z) = \int_{\tau_1 \in \t} \frac{1+\tau_1 z}{1-\tau_1 z}\ d\mu_\calf (\tau_1)
\ee
and
\be\label{anom30}
h_\calr (z) = \int_{\tau_2 \in \t} \frac{1+\tau_2 z}{1-\tau_2 z}\ d\mu_\calr (\tau_2)
\ee
for all $z\in \d$. We define a probability measure $\mu$ on $\t^2$ by $d\mu(\tau) = d \mu_\calf(\tau_1) d\mu_\calr (\tau_2)$.

Observe that if $s=(0,z) \in \calf_0$, then Definition \ref{anom.def20} and formula \eqref{anom20} imply that
\[
h(s) = h_\calf(z) =\int_{\tau_1 \in \t} \frac{1+\tau_1 z}{1-\tau_1 z}\ d\mu_\calf (\tau_1).
\]
Therefore, since $\mu_\calr$ has total variation 1, Fubini's Theorem implies that
\begin{align*}
h(s) &= \int_{\tau_2 \in \t} h(s)\ d\mu_\calr (\tau_2)\\ \\
&=  \int_{\tau_2 \in \t}\Big(\int_{\tau_1 \in \t} \frac{1+\tau_1 z}{1-\tau_1 z}\ d\mu_\calf (\tau_1)\Big)\ d\mu_\calr (\tau_2)\\ \\
&=\int_{\tau \in \t^2} \frac{1+\tau_1 z}{1-\tau_1 z}\ d\mu (\tau).
\end{align*}
But since $s=(0,z) \in \calf_0$,
\[
\frac{1+\tau_1 z}{1-\tau_1 z} = \frac{1+f_\tau(s)}{1-f_\tau(s)}.
\]
where $f_\tau$ is as in Definition \ref{anom.def30}. Therefore, equation \eqref{anom10} holds when $s \in \calf_0$.

If $s=(2z,z^2) \in \calr$, then by repeating the argument in the previous paragraph using the representation \eqref{anom30} for $h_\calr$ one shows that equation \eqref{anom10} holds when $s\in \calr$ as well.
\end{proof}

\section{The proof of  the norm-preserving extension property for $\calr\cup\cald$ }\label{proofthm10}

We first observe that Theorem \ref{autosG}  guarantees that the group of automorphisms of $G$ acts transitively on the flat geodesics in $G$ and fixes $\calr$. Therefore, it suffices to prove Theorem \ref{anom.thm10} in the special case when $\cald = \calf_0$.

Observe next that Lemma \ref{anom.lem10} implies that it suffices to prove that $\calr\cup\calf_0$ has the Herglotz-preserving extension property. However, if $U$ is a general domain, $V$ is a subset of $U$, and $f$ is a Herglotz function on $V$, then

\[
h=\frac{f- i \im f(0,0)}{\re f(0,0)}
\]
is a Herglotz function on $V$ satisfying $h(0,0) =1$. Therefore, as
\[
f=(\re f(0,0))h +i \im f(0,0),
\]
to establish the Herglotz-preserving extension property for $V$, it suffices to show that every Herglotz function $h$ on $V$ satisfying $h(0,0) =1$ has a Herglotz extension to $U$. Hence, the proof of Theorem \ref{anom.thm10} will be complete if we can show that every Herglotz function $h$ on $\calr\cup\calf_0$ satisfying $h(0,0)=1$ has a Herglotz extension to $G$.

Fix a Herglotz function $h$ on $\calr\cup\calf_0$ satisfying $h(0,0)=1$. By Lemma \ref{anom.lem30} there exists a probability measure on $\t^2$ such that equation \eqref{anom10} holds. Define a function $g$ on $G$ by the formula
\be\label{anom40}
g(s) = \int_{\tau \in \t^2} \frac{1+\Psi_\tau (s)}{1-\Psi_\tau(s)}\ d\mu(\tau)
\ee
where for each $\tau \in \t^2$, $\Psi_\tau$ is the function defined by the formula \eqref{anom2}.

 First, note that $\Psi_\tau$ is analytic on $G$ for each $\tau \in \t^2$ and that $\tau \mapsto \Psi_\tau$ is a continuous map from $\t$ into $\c(G)$ (endowed with the topology of locally uniform convergence).   Therefore,
\be\label{anom45}
 g \text{ is a well defined analytic function on } G.
\ee

Next, note that since equation \eqref{anom5} implies that, for each $\tau \in \t^2$,
\[
\re \frac{1+\Psi_\tau (s)}{1-\Psi_\tau(s)} >0
\]
for all $s\in G$, and $\mu$ is a probability measure, necessarily the function $g$ defined in equation \eqref{anom40} satisfies
\be\label{anom50}
\re g(s)> 0\ \ \ \text{ for all }s \in G.
\ee

Finally, note that Lemma \ref{anom.lem20} asserts that $f_\tau =\Psi_\tau| (\calr\cup\calf_0)$ for each $\tau \in \t^2$. Consequently,
\[
 \frac{1+f_\tau(s)}{1-f_\tau(s)} = \frac{1+\Psi_\tau (s)}{1-\Psi_\tau(s)}
\]
for each $\tau \in \t^2$ and each $s \in \calr\cup\calf_0$. Therefore, using equations \eqref{anom10} and \eqref{anom40}, for each $s \in \calr\cup\calf_0$, we have
\begin{align*}
h(s) &= \int_{\t^2} \frac{1+f_\tau(s)}{1-f_\tau(s)}\ d\mu (\tau)\\ \\
&=\int_{ \t^2} \frac{1+\Psi_\tau (s)}{1-\Psi_\tau(s)}\ d\mu(\tau)\\ \\
&= g(s).
\end{align*} 
This completes the proof of Theorem \ref{anom.thm10}. \hfill $\Box$

\chapter{$V$ and a circular region $R$ in the plane}\label{Vflat}

In the light of Lemma \ref{extlem40} the study of algebraic sets $V$ in $G$ that have the norm-preserving extension property can be reduced to the case where $V$ does not have contact with any flat datum.  Our goal in this section is to show, under these assumptions, that $V$ is a properly embedded planar Riemann surface of finite type [meaning that the boundary of $V$ has finitely many connected components].

 \begin{lemma}\label{extlem5}
 Let $V$ be an algebraic set in $G$ with the norm-preserving extension property.  If $s$ is an isolated point in $V$ then $V=\{s\}$.
 \end{lemma}
This statement follows from Proposition \ref{connect}.

Recall that the map $\beta: G\to \d$ was defined in Proposition \ref{propbeta} by the formula
\be\label{defbetabis}
\beta(s) = \frac{s^1-\overline{s^1}s^2}{1-|s^2|^2} \quad \mbox{ for } s\in G.
\ee

\begin{lemma}\label{extlem41}
Let $V$ be an algebraic set in $ G$ having the norm-preserving extension property. 
If $V$ is not a singleton  and   $V$  does not have contact with any flat datum, then  $V$ is a Riemann surface properly embedded in $G$ and $\beta|V$ is a homeomorphism onto its range.
\end{lemma}
\begin{proof}
Clearly, as $V$ is assumed not to  have contact with any flat datum, $V\neq G$. Therefore, if $V$ is not a singleton, it follows from Lemma  \ref{extlem5} that $V$ has no isolated points and so is a one-dimensional algebraic set in $G$. We first prove that $V$ has no singular points, and as a consequence that $V$ is a one-dimensional complex manifold properly embedded in $G$.  We then show that $\beta|V$ is an open mapping.  As Proposition \ref{propbeta} implies that $\beta|V$ is injective, it follows that $\beta|V$ is a homeomorphism onto its range.

{\em $V$ has no singular points.}   Fix $s_0 \in V$ and choose a neighborhood $U_0$ of $s_0$ and a holomorphic function $\zeta$ on $U_0$ such that
 \be\label{ha1}
 V\cap U_0 = Z_\zeta,
 \ee
where $Z_\zeta$ denotes the zero set of $\zeta$ in $U_0$. Let  $\beta_0=\beta(s_0)$.
Introduce local coordinates $(z,w)$ at $s_0$ by setting
\[
s=s_0 + z(\bar{\beta_0},1) +w(1,-\beta_0),
\]
and define $\eta$ by
\[
\eta(z,w) =\zeta(s).
\]
Now, since  $(\bar{\beta_0},1)$ points in the flat direction at $s_0$,  there does not exist $\eps>0$ such that $\eta$ is identically zero on the set $\set{(z,0)}{|z|<\eps}$; else $V$ contacts the flat infinitesimal datum $(s_0,(\bar{\beta_0},1))$, contrary to hypothesis. Therefore, $\eta(z,0)$ has a zero of finite positive order, $\ell$ say, at $z=0$. Consequently, by the Weierstrass Preparation Theorem, there exist analytic functions of one variable $b_1,\ldots, b_\ell$, defined on a neighborhood $\eps\d$ of zero, for some $\eps>0$, and vanishing at zero, and a holomorphic function $h(z,w)$ on $(\eps\d)^2$ such that if $P_w$ is the pseudopolynomial
\[
P_w(z) = z^\ell +b_1(w)z^{\ell-1}+\ldots + b_\ell(w),
\]
then
\be\label{ha2}
\eta(z,w) = P_w(z)h(z,w)\ \  \text{ and }\ \ h(z,w)\not=0
\ee
on $(\eps\d)^2$. Evidently, it will be the case that $s_0$ is not a singular point of $V$ if $P_w$  has a unique irreducible factor, which has degree $1$.

To see that $P_w$ has a unique irreducible factor, assume to the contrary that $Q_w$ and $R_w$ are distinct irreducible monic factors of $P_w$
 in the ring 
\[
R=\c(\eps\d)[z]
\]
of pseudopolynomials.   Since the zeros of a polynomial depend continuously on its coefficients, there exists $\de>0$ such that the zeros of the polynomial $P_w(\cdot)$ lie in $\d^-$ whenever $|w|<\de$.  It follows that the zeros of $Q_w(\cdot)$ and $R_w(\cdot)$ also lie in $\d^-$ when $|w| < \de$.

If the restrictions of $Q_w, R_w$ to $(\eps\d)\times n^{-1}\d$ are equal, for some $n> 1/\de$, then in fact $Q_w=R_w$ on $(\eps\d)^2$.  Since $Q_w$ and $R_w$ are distinct, for each $n>1/\de$, there exists $w_n\in n^{-1}\d$ such that $Q_{w_n}(\cdot) \neq R_{w_n}(\cdot)$.  As $Q_{w_n}(\cdot), \,  R_{w_n}(\cdot)$ are distinct monic polynomials, their zero sets are not identical, and so there exist complex numbers $x_n$ and $y_n$   such that  $x_n \neq y_n$ and 
\[
Q_{w_n}(x_n) =0=R_{w_n}(y_n).
\]
Since $x_x, y_n \in\d^-$, by passing to a subsequence if necessary, we may assume that the sequences $\{x_n\}$ and $\{y_n\}$ converge. Since, for all $n$, $Q_{w_n}(x_n) =0$, we have
$P_{w_n}(x_n)= 0$ and thus
\[
0 = \lim_{n \to \infty}
P_{w_n}(x_n)
\]
\[= \lim_{n \to \infty}x_n^\ell + \lim_{n \to \infty} b_1(w_n)x_n^{\ell-1}+\ldots + \lim_{n \to \infty}b_\ell(w_n)= (\lim_{n \to \infty}x_n)^\ell.
\]
Therefore
 $x_n\to 0$ and, similarly, $y_n\to 0$.     Finally, define $s_n,t_n \in V$ by setting
\[
s_n =s_0 + x_n(\bar{\beta_0},1) +w_n(1,-\beta_0)
\]
and
\[
t_n =s_0 + y_n(\bar{\beta_0},1) +w_n(1,-\beta_0).
\]
With the above constructions we have $s_n\to s_0$ and 
\[
 \frac{t_n-s_n}{\norm{t_n-s_n}} = \frac{y_n-x_n}{\|y_n-x_n\|}\frac{(\bar\beta_0,1)}{\norm{(\bar\beta_0,1)}}.
\]
After passage to a further subsequence one infers that there exists $\chi\in\t$ such that
\[
\frac{t_n-s_n}{\norm{t_n-s_n}} \to \chi \frac{(\bar\beta_0,1)}{\norm{(\bar\beta_0,1)}},
\]
which is a flat direction at $s_0$.  Thus   $V$ is contacted by the flat infinitesimal datum $(s_0, (\bar\beta_0,1))$, contrary to assumption.
 This completes the proof that $P_w$ does not have a distinct pair of irreducible factors.

To see that  each of the irreducible factors of $P_w$ has degree $1$, assume to the contrary that $Q_w$ is an irreducible factor and $\deg Q_w =d \ge 2$. Since $Q_w$ is irreducible, there exists $\eps>0$ such that for each $w$ with $0<|w|<\eps$ the equation $Q_w(z) =0$ has exactly $d$ distinct roots.  Thus we may choose sequences $\{w_n\}$, $\{x_n\}$, and $\{y_n\}$ such that $w_n \to 0$ and
\[
x_n\not=y_n\ \ \text{ and }\ \ Q_{w_n}(x_n) =0=Q_{w_n}(y_n)
\]
for all $n$. By the argument of the previous paragraph, these sequences lead to a contradiction to the assumption that $V$ does not have contact with a flat datum.  

We have shown that  $\eta(z,w)=(z+b_1(w))^k h(z,w)$  for some positive integer $k$.  Thus
\be\label{b1w}
V\cap U_0=Z_\zeta= \{s\in U_0: z+b_1(w)=0\}
\ee
and $V\cap U_0 = f(N)$ for some neighborhood $N$ of $0$ in $\c$, where
\be\label{locparamV}
f(w)=s_0-b_1(w)(\bar\beta_0,1)+w(1,-\beta_0).
\ee
Thus $s_0$ is not a singular point of $V$.  Hence $V$ has no singular points.

{\em $\beta|V$ is an open mapping}.  As $V$ is a one dimensional analytic set without singular points, it follows that $V$ is a Riemann surface. 
 Since algebraic sets are closed, the injection map of $V$ into $G$ is a proper embedding. There remains to show that the function $\beta|V$ is a homeomorphism onto its range. Clearly, $\beta|V$ is continuous. Furthermore, as $V$ is assumed not to have contact with any flat geodesic, it is clear from Proposition \ref{propbeta} that $\beta|V$ is injective. We conclude the  proof that $\beta$ is a homeomorphism by showing that $\beta|V$ is an open mapping.

Firstly, for $f$ as in equation \eqref{locparamV},
\[
df_0 =(-b_1'(0)\bar{\beta_0}+1,\, -b_1'(0)-\beta_0)\ dw_0.
\]
Secondly, the total complex differential of $\beta$ at a general point $s\in G$ is given by
\begin{align}\label{thisone}
d\beta&=\frac{\partial\beta}{\partial s^1} ds^1 +\frac{\partial\beta}{\partial \bar{s^1}}d\bar{s^1}
+\frac{\partial\beta}{\partial s^2} ds^2
+\frac{\partial\beta}{\partial \bar{s^2}} d\bar{s^2}\notag \\
&=\frac{1}{1-|s^2|^2} ds^1
 -\frac{s^2}{1-|s^2|^2}d\bar{s^1}
+\frac{-\bar{s^1}+s^1\bar{s^2}}{(1-|s^2|^2)^2} ds^2
+s^2\frac{s^1 - \bar{s^1}s^2}{(1-|s^2|^2)^2} d\bar{s^2}\notag \\
&=\frac{1}{1-|s^2|^2}
\Big(ds^1 -s^2d\bar{s^1}
-\bar{\beta} ds^2 +s^2\beta d\bar{s^2}\Big) \notag \\
&=\frac{1}{1-|s^2|^2}
\Big(ds^1-\bar{\beta} ds^2  - s^2{(d\bar{s^1}-{\beta} d\bar{s^2} )}\Big).
\end{align}
Thirdly, observe that if $f=(f^1,f^2)$ is given by equation \eqref{locparamV}, then 
\[
(df^1)_0=(-b_1'(0)\bar{\beta_0}+1) dw_0 \quad \mbox{ and }\quad (df^2)_0=(-b_1'(0)-\beta_0) dw_0,
\]
 and so
\begin{align}\label{penult}
(df^1-\bar{\beta_0} df^2)_0 &=(-b_1'(0)\bar{\beta_0}+1)-\bar{\beta_0}(-b_1'(0)-\beta_0) dw_0\notag\\
&=(1+|\beta_0|^2) dw_0.
\end{align}
Similarly
\begin{align}\label{ult}
(d\bar{f^1}-{\beta_0} d\bar{f^2})_0 
&=(1+|\beta_0|^2) d\bar w_0.
\end{align}
Therefore, by virtue of equations \eqref{thisone}, \eqref{penult} and \eqref{ult},
\begin{align}\label{dbcf}
d(\beta \circ f)_0 &=C (dw_0 - s_0^2d\bar w_0)
\end{align}
where
\[
C=\frac{1+|\beta_0|^2}{1-|s_0^2|^2} > 0.
\]
 Let $u(x,y), v(x,y)$ be the real and imaginary parts of $\beta\circ f(x+iy)$ for $x,y\in\r$.

It follows from equation \eqref{dbcf} that the Jacobian determinant of $\beta\circ f$ at $0$  is
\[
\frac{\partial(u,v)}{\partial(x,y)}(0)=  C^2-|Cs_0^2|^2 = C^2(1-|s_0^2|^2) > 0.
\]
By the Inverse Function Theorem, $\beta \circ f$ is invertible in a neighborhood of $s_0$ in $V$. As the point $s_0 \in V$ is arbitrary, $\beta|V$ is an open mapping, as was to be proved.

We have shown that $V$ is a one-dimensional manifold.  By Proposition \ref{connect}, $V$ is connected, and so $V$ is a Riemann surface.
\end{proof}
\begin{lemma}\label{extlem50}
If $V$ is an algebraic set in $G$  
then $\beta(V)$ is a semialgebraic set and
 $\partial \beta(V)$ has a finite number of components.
\end{lemma}
\index{semialgebraic set}
\begin{proof}
$G$ is the set of points $s\in\c^2$ such that
\[
|s^1-\overline{s^1}s^2|^2 < (1-|s^2|^2)^2,
\]
and therefore $G$ is a semialgebraic set.  Since $V=G\cap A$ for some algebraic set $A$, it follows that 
$V$ is a semialgebraic set (see, for example, \cite{bcr}).  It is clear from the definition of $\beta$ that $\beta$ is a regular mapping on $G$ (that is, the real and imaginary parts of $\beta$ are rational functions whose denominators do not vanish on $G$).  Hence $\beta$ is a semialgebraic map, and so, by \cite[Proposition 2.2.7]{bcr}, $\beta(V)$ is a semialgebraic set.  It follows from \cite[Proposition 2.2.2]{bcr} that  the boundary $\partial\beta(V)$ of $\beta(V)$ is a semialgebraic set.  By \cite[Theorem 2.4.5]{bcr}, $\partial\beta(V)$ has finitely many connected components.
\end{proof}

A {\em circular region} is defined to be a region in $\c$ whose boundary consists of a finite number of disjoint \nd circles.   
\index{circular region}
These domains are canonical multiply connected domains, in the sense that every finitely connected planar domain is conformally equivalent to a circular region  \cite[Theorem 7.9]{conway}.  A domain is {\em nondegenerate} if none of its boundary components is a singleton set. 
\index{domain, nondegenerate}

\begin{lemma}\label{extlem60}  
Let $V$ be an algebraic set in $ G$ having the norm-preserving extension property. 
If $V$ does not have contact with any flat datum, then either $V$ is a singleton or there exists a   circular region $R$ in the plane and a bijective holomorphic mapping $\sigma: R\to V$ such that $\sigma'(z)\neq 0$ for all $z\in R$.
\end{lemma}
\begin{proof}
Suppose that $V$ is not a singleton.
By Lemma \ref{extlem41}, $V$ is a Riemann surface and
 $V$ is homeomorphic to $\beta(V)$, an open subset of $\d$.  Now any open planar Riemann surface is conformally equivalent to a plane region $R$ \cite[Chapter III, Paragraph 4]{AS}. By Lemma \ref{extlem50}, $\partial \beta(V)$ has a finite number of components, and so $\beta(V)$ is finitely connected.  Since $R$ is homeomorphic to $\beta(V)$, \ $R$ is a finitely connected plane region.  We claim that $R$ is nondegenerate.
Suppose, to the contrary, that there is some point $z_0$ that is a singleton boundary component of $R$.
Let $\si=(a,b): R \to V$ be a conformal map onto the Riemann surface $V$.      Since every bounded holomorphic function on $R\setminus \{z_0\}$ continues holomorphically to $R$,
\[
s_0 \df \lim_{z\to z_0} (a(z),b(z)) 
\]
exists.  Furthermore, since $z_0\in \partial R, \, s_0\in\partial V \subseteq \partial G$.  Hence there exists $\omega\in\t$ such that $|\Phi_\omega(s_0)|=1$, for otherwise $|\Phi_\omega(s_0)| < 1$ for all $\omega\in\t$, which implies that $s_0\in G$, by Proposition \ref{elG}.   The holomorphic extension to $R \cup \{z_0\}$ of $\Phi_\omega\circ \sigma$ has modulus at most $1$  and attains its maximum modulus at the interior point $z_0$, contrary to the Maximum Modulus Principle.  Hence the boundary of $R$ has no singleton components, and so, by  \cite[Theorem 7.9]{conway}, we can assume that $R$ is a circular domain.
Then $\si:R\to V$ is the required conformal map.
\end{proof}

\chapter{Proof of the main theorem}\label{mainproof}

This chapter is devoted to the proof of Theorem \ref{main}.
For convenience we repeat the statement.

\begin{theorem}\label{mainbis}
 $V$ is an algebraic subset of $G$ having the norm-preserving extension property if and only if   either $V$ is a retract in $G$ or $V$ is the union of $\calr$ and  a flat geodesic in $G$. 
\end{theorem}
By Proposition \ref{extprop10a} and   Theorem \ref{anom.thm10}, if $V$ is a retract in $G$ or $V=\calr\cup \cald$, where $\cald$ is a flat geodesic in $G$,  then $V$ is  an algebraic subset of $G$ with the norm-preserving extension property.  We shall prove the converse.

Accordingly, for the remainder of the chapter, fix  a set $V\subseteq G$ such that
\be\label{sh1}
V \mbox{ is an algebraic set in } G, \mbox{ and }
\ee
\be\label{sh2}
V \mbox{ has the norm-preserving extension property. }
\ee

We wish to show that  $V$ is a retract in $G$   or $V=\calr\cup \cald$, where  $\cald$ is a flat geodesic in $G$.  This conclusion certainly follows if $V$ is a singleton set
or if $V=G$.  We may therefore assume that
\be\label{sh3}
V \mbox{ is not a singleton set, and }
\ee
\be\label{sh4}
V\neq G.
\ee
If $V$ has contact with a flat datum $\de$, then by Lemma \ref{extlem40},
either $V=G$ or $V=\cald_\de$, both of which are retracts in $G$, or $V=\calr\cup \cald_\de$. If $V$ has contact with a royal datum $\de$, then by Lemma \ref{extlem45},
either $V=\calr$, $V$ is the union of $\calr$ and  a flat geodesic in $G$   or $V=G$. We may therefore assume further
that 
\be\label{sh5}
V \mbox{ does not have contact with any flat or royal datum.}
\ee
By Lemma \ref{extlem60} there exists a circular region $R$ in the plane and  a bijective holomorphic mapping $\sigma: R \to V$ such that $\sigma'(z) \neq 0$ for all $z\in R$.
\begin{lemma}\label{purelybal}
If  $V$ has contact with a purely balanced datum $\de$ in $G$ then {\rm (1)} $V=\cald_\de$ and {\rm (2)} $V$ is a retract in $G$.
\end{lemma}
\begin{proof}
Suppose that $V$ has contact with a purely balanced datum $\de$ in $G$.  By Lemma \ref{extlem20}, $\cald_\de\subseteq V$.  Choose a solution $k$ of $\Kob(\de)$, so that $\cald_\de= k(\d)$. 
 Then $k$ is a rational $\Ga$-inner function and so {\em a fortiori} proper as a map from $\d$ to $V$, 
and hence $f\df\si^{-1}\circ k$ is a proper injective holomorphic map from $\d$ to $R$.   
By the Open Mapping Theorem, $f$ is a homeomorphism from $\d$ to $f(\d)$.  We claim that $f(\d)=R$.

Suppose that $f(\d)\neq R$.  Choose a point $w\in R\setminus f(\d)$, a point $u\in f(\d)$ and a continuous path $\gamma$ from $u$ to $w$ in $R$.  Let $z_0=\ga(t_0)$ where $t_0=\sup\{t:\ga(t)\in f(\d)\}$.  Then $z_0\in (\partial f(\d)) \cap R$.

Pick a sequence $(\la_n)$ in $\d$ such that $f(\la_n)\to z_0$.  If $(\la_n)$ has a subsequence which converges to a limit $\mu$ in $\d$ then $f(\mu)=z_0$ and hence $z_0\in f(\d)$, a contradiction.  Hence $|\la_n| \to 1$, and so, by the propriety of $f$, $f(\la_n)$ tends to $\partial R$, and so  $z_0\in\partial R$, contrary to the fact that $z_0\in R$.  Hence $f(\d) = R$, and so $k(\d)=\si\circ f(\d)=\si(R)=V$.
Thus $\cald_\de=V$, and so $V$ is a retract in $G$ (see equation \eqref{ret10}).
\end{proof}

In the light of Lemma \ref{purelybal} we may further assume that
\be\label{sh6}
V \mbox{ does not have contact with any purely balanced datum.}
\ee

 Let 
\be\label{bdryR}
\partial R= \partial_0\cup\partial_1\cup \dots \cup \partial_n
\ee
be a decomposition of $\partial R$ into its connected components, where $\partial_0$ is the boundary of the unbounded component of the complement of $R^-$.  We may take $\partial_0$ to be the unit circle.  

\begin{lemma}\label{Vrtr}
If $n=0$ then $V$ is a nontrivial retract in $G$ and is a complex geodesic of $G$.
\end{lemma}
\begin{proof}
Since $n=0$ we have $R= \d$. By assumption 
$V$ is an algebraic subset of $G$ having the norm-preserving extension property, and thus the map
 $\sigma^{-1}: V \to \d$ has a holomorphic extension $f: G \to \bar{\d}$. By the Open Mapping Theorem, $f(G) \subset \d$, and so $\sigma \circ f$ is a holomorphic map from $G$ to $V$ which is the identity on $V$. That is, $V$ is a retract in $G$.  By statements \eqref{sh3} and \eqref{sh4},  $V$ is neither a singleton nor the whole of $G$, and so $V$ is a nontrivial retract.  By Theorem \ref{retthm10}, $V$ is a complex geodesic in $G$.
\end{proof}

{\it The rest of the proof consists of a demonstration that $n=0$}. 

Suppose $n \neq 0$. Write $\si=(a,b)$, so that $a$ and $b$ are holomorphic functions on $R$ and, for all $z\in R$, $a'(z)$ and $b'(z)$ are not both zero.

\section{Preliminary lemmas}\label{prelim}

For $\omega\in\t$ let $f_\omega=\Phi_\omega \circ \si:R\to\d$.  Define $\cala\subseteq \d(R)$ by
\[
\cala=\{f_\omega:\omega\in\t\}.
\]
\begin{lemma}\label{lem20a}
If $\omega_1\neq\omega_2$ then $f_{\omega_1}\neq f_{\omega_2}$.
\end{lemma}
\begin{proof}
Assume, to the contrary, that $\omega_1, \omega_2$ are distinct point in $\t$ but $f_{\omega_1}=f_{\omega_2}$.
Straightforward manipulation of the equation
\[
\frac{2\omega_1 b-a}{2-\omega_1 a}=\frac{2\omega_2 b-a}{2-\omega_2 a}
\]
yields the relation $a^2=4b$, so that $V=\si(R) \subseteq \calr$, contrary to the statement \eqref{sh5}.
\end{proof}
\begin{lemma}\label{lem30}
$\cala$ is a universal set for the Carath\'eodory Problem on $R$. Furthermore, if $\zeta$ is a \nd datum in $R$, then there exists a {\em unique} $f \in \cala$ that solves $\Car(\zeta)$.
\end{lemma}
\begin{proof}
If $\zeta$ is a datum in $R$ then $\sigma(\zeta)$ is a datum in $V$. Furthermore, since $\sigma$ is biholomorphic, the holomorphic invariance of the Carath\'eodory metric implies that
\[
|\zeta|_R = |\sigma(\zeta)|_V.
\]
Here we use a self-explanatory variant on the notation $\car{\cdot}$ which indicates the domain in question.
Also, as $V$ has the norm-preserving extension property,
\[
|\sigma(\zeta)|_V = |\sigma(\zeta)|_G.
\]
Therefore, if we view $\si(\zeta)$ as a datum in $G$ and we choose $\omega \in \t$ such that $\Phi_\omega$ solves $\Car(\sigma(\zeta))$, then $f_\omega=\Phi_\omega \circ \sigma$ solves $\Car(\zeta)$.  

To see that $f_\omega$ is unique, assume to the contrary that $f_{\omega_1}$ and $f_{\omega_2}$ solve $\Car(\zeta)$ and $\omega_1\not=\omega_2$. Then $\Phi_{\omega_1}$ and $\Phi_{\omega_2}$ both solve $\Car(\sigma(\zeta))$.  This means that $\si(\zeta)$ is a flat, royal or purely balanced datum that contacts $V$, contrary to statements \eqref{sh5} and \eqref{sh6}.
\end{proof}

In the light of Lemma \ref{Vrtr} it is clear that the proof of Theorem \ref{main} will be complete if we establish that $n=0$.
Assume to the contrary that 
\be\label{ngeq1}
n \ge 1.
\ee

While $\cala$ is a universal set for the Carath\'eodory Problem on $R$, by Lemma \ref{lem30}, it is not necessarily the case that each $f\in\cala$ solves $\Car (\zeta)$ for some \nd datum $\zeta$ in $R$. Accordingly, we define a set ${E} \subseteq \t$ by
\[
{E} = \set{\omega\in\t}{f_\omega \text{ solves } \Car(\zeta) \text{ for some \nd datum } \zeta \text{ in } R}
\]

It is well known that the Carath\'eodory extremal functions on $R$ (the {\em Ahlfors functions}) are $(n+1)$-valent inner functions that analytically continue to a neighborhood of $R^-$ (for example, \cite[Chapter 5, Theorem 1.6]{fisher}).  In particular the functions $f_\omega, \, \omega\in E$ enjoy these properties.  Of course $f_\omega$ is nonconstant for $\omega\in E$.

\begin{lemma}\label{lem35}
There exist two datums $\zeta_1$ and $\zeta_2$ in $R$ such that $f_{\omega_1}$ solves $\Car(\zeta_1)$, $f_{\omega_2}$ solves $\Car(\zeta_2)$, and $\omega_1 \not= \omega_2$.
\end{lemma}
\begin{proof}
Fix a datum $\zeta_1$ in $R$ and let $f_{\omega_1}$ solve $\Car(\zeta_1)$. As $f$ is $(n+1)$-valent, there exist $z_1,z_2 \in R$ such that $z_1\not=z_2$ and $f_{\omega_1}(z_1)=f_{\omega_2}(z_2)$. Let $\zeta_2 = (z_1,z_2)$. As $|f_{\omega_1}(\zeta_2)|=0$, $f_{\omega_1}$ cannot possibly solve $\Car(\zeta_2)$. Therefore, the lemma follows by choice of $\omega_2$ such that $f_{\omega_2}$ solves $\Car(\zeta_2)$.
\end{proof}

\begin{lemma}\label{lem40}
If  $n\geq 1$ then $E$ is an infinite set.
\end{lemma}
\begin{proof}
Assume to the contrary that $E$ is finite.

Fix a nondegenerate discrete datum $\zeta_1=(z_0,z_1)$ in $R$ and let $f_{\omega_1}$ solve $\Car(\zeta_1)$. Since $f_{\omega_1}$ is $(n+1)$-valent and $n\geq 1$, there exists $z_2 \in R$ with $z_2\not=z_1$ and $f_{\omega_1}(z_1)=f_{\omega_1}(z_2)$. Construct a continuous curve $\gamma:[1,2]\to R$ with $\gamma(1)=z_1$ and $\gamma(2)=z_2$ and define a curve of discrete datums $\delta_t$ in $R$ by the formula
\[
\zeta_t = (z_0,\gamma(t)), \qquad 1\le t \le 2.
\]
Since Lemma \ref{lem30} guarantees that for each $t\in[1,2]$ there exists a unique $\omega \in E$ such that $f_\omega$ solves $\Car(\zeta_t)$, it follows that there exists a well defined function $\Omega:[1,2] \to E$ such that
\[
f_{\Omega(t)} \text{ solves } \Car(\zeta_t) \text{ for all } t\in[1,2].
\]
Also, observe that by construction, $\Omega(1) = \omega_1$.

We claim that for each $\tau \in E$, $\Omega^{-1}(\{\tau\})$ is an open subset of $[1,2]$. To prove this claim, fix $\tau_0=\Omega(t_0) \in E$ and note that since the solution to $\Car(\zeta_{t_0})$ is unique,
\be\label{5}
|f_\tau (\zeta_{t_0})| < |f_{\tau_0}(\zeta_{t_0})| \text{ for all } \tau \in E\setminus \{\tau_0\}.
\ee
It follows by continuity, that equation \eqref{5} holds, with $t_0$ replaced by $t$, for all $t$ in a neighborhood $I$ of $t_0$. But then,
\[
f_{\tau_0} \text{ solves } \Car(\zeta_t) \text{ for all } t\in I,
\]
that is, $I \subseteq \Omega^{-1}(\{\tau_0\}$. This proves that $\Omega^{-1}(\{\tau\})$ is an open subset of $[1,2]$ for each $\tau \in E$.

Since $\Omega$ is continuous and $[1,2]$ is connected,  $\Omega$ is constant. Hence
\[
\Omega(2) = \Omega(1) = \omega_1
\]
 and $f_{\omega_1}$ solves $\Car(\zeta_2)$. But $\zeta_2$ is nondegenerate while $|f_{\omega_1}(\zeta_2)| =0$.
\end{proof}

\section{$\sigma: R\to V$ is analytic on $R^-$}\label{analytic}

\begin{lemma}\label{lem50}
The map $\sigma$ extends holomorphically to a neighborhood of $R^-$.
\end{lemma}
\begin{proof}
By Lemma \ref{lem40} there exist distinct points $\omega_1,\omega_2 \in E$.
Let
\be\label{f1f2}
f_1 =f_{\omega_1}= \frac{2\omega_1 b-a}{2-\omega_1a}\ \ \text{ and }\ \  f_2 =f_{\omega_2}= \frac{2\omega_2 b-a}{2-\omega_2a}.
\ee
Solve these equations for $a$ and $b$ in terms of $f_1,f_2$:
\begin{align*}
a&=2\frac{\omega_1f_2 - \omega_2 f_1}{\omega_2-\omega_1 +\omega_1\omega_2 (f_2-f_1)}\\
b&=\frac{f_2-f_1 +(\omega_2-\omega_1)f_1f_2}{\omega_2-\omega_1 +\omega_1\omega_2 (f_2-f_1)}.
\end{align*}
Since $f_1$ and $f_2$ are holomorphic on a neighborhood of $R^-$, $a$ and $b$ extend to meromorphic functions defined on a neighborhood of $R^-$. But $a$ and $b$ are bounded on $R$, and so extend to holomorphic functions on a neighborhood of $R^-$.
\end{proof}
\begin{lemma}\label{lem60}
$\sigma(\partial R)$ is contained in the distinguished boundary $b\Gamma$ of the closure $\Ga$ of $G$, and therefore
\[
|a|\leq 2, \quad |b| = 1\quad \mbox{ and }\quad a=\bar a b\quad\mbox{ on } \partial R.
\]
\end{lemma}
\begin{proof}
We are using the characterization of $b\Ga$ given in Proposition \ref{bGam}.
First notice that as $\sigma\in G(R)$, $|a(z)| <2$ for all $z\in R$. By Lemma \ref{lem50}, $\si$ extends continuously to $\partial R$, and therefore
$|a(z)| \le 2$ for all $z \in \partial R$. There remains to prove that $a=\bar ab$ and $|b|=1$ on $ \partial R$.

Let $f_1$ and $f_2$ be as in Lemma \ref{lem50}, let $E_1 = \set{z \in \partial R}{a(z)=2\bar\omega_1}$ and  $E_2 = \set{z \in \partial R}{a(z)=2\bar\omega_2}$. Clearly, Lemma \ref{lem50} guarantees that $E_1$ and $E_2$ are finite. As
\[
f_1 =\frac{2\omega_1 b-a}{2-\omega_1a}
\]
is inner, it follows that $|\Phi_{\omega_1}(a(z),b(z))|=1$ for all 
$z \in \partial R \setminus E_1$. Since, for such $z$,  $(a(z),b(z)) \in \Ga\setminus \{(2\bar\omega_1,\bar\omega_1^2)\}$, Proposition \ref{modPhi1} shows that
\be\label{10}
\omega_1 \big(a(z)-\overline {a(z)}b(z)\big) = 1-|b(z)|^2
\ee

for all $z \in \partial R \setminus E_1$. But Lemma \ref{lem50} implies that $a$ and $b$ are continuous on $\partial R$. Therefore, equation \eqref{10} holds for all $z \in \partial R$. In similar fashion we deduce that
\be\label{20}
 \omega_2 \big(a(z)-\overline {a(z)}b(z)\big) = 1-|b(z)|^2\quad \mbox{ for all } z\in\partial R.
\ee
Since $\omega_1 \neq \omega_2$, it follows immediately from equations \eqref{10} and \eqref{20} that
\be\label{30}
a(z)-\overline {a(z)}b(z) =0\ \ \text{and }\ \ |b(z)|=1
\ee
for all $z \in \partial R$.  Thus $\si(z)\in b\Ga$ for all $z\in\partial R$.
\end{proof}
The following lemma gives an additional relationship between $a$ and $b$ that must hold whenever $\omega\in E$.
\begin{lemma}\label{lem70}
If $\omega\in E$ and $a(\eta)= 2\bar\omega$ for some $\eta\in\partial R$, then $b(\eta)=\bar\omega^2$.  In consequence
\[
a(\eta)^2=4b(\eta).
\]
\begin{proof}
Suppose that $b(\eta)\neq \bar\omega^2$.  Then $\Phi_\omega$ is continuous at $\si(\eta)$, and
\[
f_\omega=\frac{2\omega b-a}{2-\omega a}
\]
is bounded on a neighborhood of $\si(\eta)$.  Since $2-\omega a =0$ at $\eta$, necessarily $2\omega b-a=0$ at $\eta$, so that $b(\eta)=\bar\omega^2$.
\end{proof}
\end{lemma}
\begin{lemma}\label{lem80}
The set 
\[
F\df\{\omega\in\t:2\bar\omega\in a(\partial R)\}
\]
is finite.
\end{lemma}
\begin{proof}
If not then there is an infinite sequence $(\omega_i)$ of distinct points in $\t$ and a sequence $(\eta_i)$ of distinct points in $\partial R$ such that $a(\eta_i)=2\bar\omega_i$ for each $i$.  Lemma \ref{lem70} then implies that $a(\eta_i)^2=4b(\eta_i)$ for all $i$.  Hence, by Lemma \ref{lem50}, $a(\eta)^2=4b(\eta)$ for all $\eta\in R$, so that $\si(R)$ is contained in the royal variety $\calr$.  Then each datum in $G$ that contacts $V$ also contacts $\calr$, contrary to assumption \eqref{sh5}.
\end{proof}
\begin{lemma}\label{lem90}
There exists $\omega\in E$ such that $2\bar\omega\notin a(\partial R)$.
\end{lemma}
\begin{proof}
Lemma \ref{lem40} asserts that $E$ is infinite, while Lemma \ref{lem80} states that the subset $F$ of $E$ is finite.  Therefore $E\setminus F$ is nonempty, that is, there exists $\omega\in E$ such that $2\bar\omega \notin a(\partial R)$.
\end{proof}
\section{Degree considerations}\label{degrees}
For $r=0, \ldots, n$ and $\phi$ a nonvanishing continuously differentiable complex-valued function on $\partial_r$ we define
\[
\#_r (\phi) = \frac{1}{2\pi i}\int_{\partial_r} \frac{\phi'(z)}{\phi(z)}\ dz.
\]
Here, in  the decomposition \eqref{bdryR} of the boundary of $R$, $\partial_r$ is oriented counterclockwise if $r=0$ and clockwise if $r>0$.

If $\phi$ is defined on $\partial R$ we set
\[
\# (\phi) = \sum_{r=0}^n\#_r (\phi).
\]
We note that if $\phi$ is a nonconstant inner function on $R$ that extends to be holomorphic on a neighborhood $U$ of $R^-$, then $\#_r(\phi) \ge 1$ for each $r$ so that  $\#(\phi) \ge n+1$. Furthermore, since the cycle $\partial_0 +\dots + \partial_n$ is clearly homologous to $0$ in $U$, we may apply the Argument Principle to deduce that in this case $\#(\phi)$ is equal to the number of zeros of $\phi$ in $R$ (for example, \cite[Chapter 4, Subsection 5.2]{ahlfors}).

  \begin{lemma}\label{lem100}
\be\label{50}
\#(b) = n+1.
\ee
\end{lemma}
\begin{proof}
By Lemma \ref{lem90}, there is an $\omega \in E$ such that  $2\bar\omega \not\in a(\partial R)$.  Thus  $\omega a(z)\neq 2$ for all $z\in \partial R$. It follows from Lemma \ref{lem60}  that $|a(z)| \le 2$ for all $z\in\partial R$. Hence the curves
\[
\frac{2-\overline{\omega a}}{2-\omega a}(z), \qquad z\in \partial_r,
\]
are well defined and do not meet the closed negative real axis for 
$r=0,\dots,n$. Therefore
\[
\#\left(\frac{2-\overline{\omega a}}{2-\omega a}\right) = 0.
\]
Since
\be\label{formfo}
f_\omega = \frac{2\omega b-a}{2-\omega a}=\frac{2\omega b-\bar a b}{2-\omega a}=\omega b \frac{2-\overline{\omega a}}{2-\omega a},
\ee
we have $  \#(f_\omega)=\#(b) $.
Now $f_\omega$ is not a constant function, for otherwise $\Phi_\omega \circ \si$ is a constant of unit modulus on $R$, and hence $\si(R) \subseteq \partial G$, which is false.  Thus  the Ahlfors function $f_\omega$ for $R$ has $n+1$ zeros counting multiplicity, so that 
\be\label{indexfo}
\# (f_\omega) = n+1.
\ee
Therefore  $\#(b)=n+1$.
\end{proof}
Notice that the equation \eqref{formfo} is valid for all $\omega\in\t\setminus F$, and therefore the index formula \eqref{indexfo} is also valid for such $\omega$.
  \begin{lemma}\label{lem110}
The set $F$ of Lemma \ref{lem80} is empty.
\end{lemma}
\begin{proof} Assume to the contrary that $\omega_0 \in F$, say $a(z_0)=2\bar\omega_0$ where $z_0\in\partial R$. Since $F$ is finite, there exists a sequence $\{\omega_i\}_{i\geq 1}$ in $\t \setminus F$
such that $ \omega_i \to \omega_0$. Consequently, $f_{\omega_i} \to f_{\omega_0}$ uniformly on compact subsets of $R$. Furthermore, the denominator of  $f_{\omega_i}$, $2-  {\omega_i}a$, converges uniformly on a neighborhood $U$ of $\bar{R}$ to the denominator of  $f_{\omega_0}$. Hence, by Hurwitz's Theorem,  $2-\omega_i a$ has a zero at a point $w_i\in U$, and since $\omega_i \notin F$, we have $w_i\notin R^-$.  We can assume that $w_i \to z_0$.  Thus
$f_{\omega_i}$ has a pole at a point $w_i\in U\setminus \bar{R}$. 
By the Schwarz Reflection Principle, for $i\geq 0$, $f_{\omega_i}$ has a zero at a point $z_i \in R$, where $z_i \to z_0$.  
Now  $\omega_i\notin F$ for each $i\geq 1$,
and so the formula \eqref{indexfo} obtains.  Hence $f_{\omega_i}$ is an inner function on $R$ of minimal degree $n+1$.
Since cancellation occurs in the limit $f_{\omega_0}$ of $(f_{\omega_i})$, it follows that $\#(f_{\omega_0}) < n+1$, 
and therefore $\#(f_{\omega_0})=0$, which is to say that $f_{\omega_0}$ is constant. 
However, since $\omega_0\in E$, there is a \nd datum $\de$ in $R$ such that $f_{\omega_0}$ solves $\Car(\de)$, which implies that  $f_{\omega_0}$ is nonconstant, a contradiction.
Hence $F=\emptyset$.
\end{proof}

{\em Conclusion}.  It remains to show that $n=0$.  Suppose not.

By  Lemma \ref{lem110}, the set $F$ is empty. Therefore, on $\partial R$, $|a| < 2$  and
\[
| 4b - (4b - a^2)| =| a^2| < 4 = |4b|.
\]
Hence, by Rouch\'e's Theorem, $4b-a^2$ has the same number of zeros as $b$, which, in view of Lemma \ref{lem100}, is to say that
\[
\#(4b-a^2) =\#(b) = n+1.
\]
Since $n \ge 1$, this implies that $4b-a^2$ has at least two zeros in $R$, and hence that $V=\si(R)$ meets $\calr$ at least twice. Thus there exists
a royal datum in $G$ which contacts $V$, contrary to statement \eqref{sh5}.

We have established that $n=0$.  This completes the proof of Theorem \ref{main}.

\chapter{Sets in $\d^2$ with the symmetric extension property}\label{appD2}

Motivation for the study of the \npep \ in a domain of holomorphy is described in the introduction of \cite{agmc_vn}.  One motive is to understand the set of solutions of a Nevanlinna-Pick interpolation problem on $\d^2$.  Another is to prove refinements of the inequalities of von Neumann \cite{neu} and And\^o \cite{and63}.  In this section we shall apply the preceding results to obtain some statements relevant to interpolation problems on $\d^2$ having a symmetry property; in the next section we shall apply our results to the theory of spectral sets of commuting pairs of operators.

In \cite[Theorem 1.20]{agmc_vn} Agler and McCarthy described all sets in the bidisc that have the \npep.  The following statement is contained their result.

\begin{theorem}\label{jim&john}
An algebraic set $V$ in $\d^2$ has the \npep \ if and only if $V$ has one of the following forms.
\begin{enumerate}[\rm (1)]
\item $V=\{\la\}$ for some $\la\in\d^2$;
\item $V=\d^2$;
\item $V=\{(z, f(z)):z\in\d\}$ for some $f\in\d(\d)$;
\item $V=\{(f(z),z):z\in\d\}$ for some $f\in\d(\d)$.
\end{enumerate}
\end{theorem}
The authors also generalized the \npep \ in \cite[Definition 1.2]{agmc_vn} as follows. 
\begin{definition}\label{Aextprop}
 Let $\Omega$ be a domain of holomorphy, $V$ be a subset of $\Omega$ and $A$ be a collection of bounded holomorphic functions on $V$.  Then $V$ is said to have the {\em $A$-extension property} (relative to $\Omega$) if, for every $f\in A$, there is a bounded holomorphic function $g$ on $\Omega$ such that $g|V=f$ and
\[
\sup_\Omega |g| = \sup_V |f|.
\]
\end{definition}
\index{extension property! $A$-}
Let $V$ be a symmetric algebraic set in $\d^2$ (`symmetric' meaning that $(\la^1,\la^2)\in V$ implies that $(\la^2,\la^1) \in V$). 
\index{set!symmetric in $\d^2$}
\index{function!symmetric}
 Let $H^\infty_\sym(V)$ denote the algebra of bounded holomorphic functions $g$ on $V$ which are symmetric, in the sense that $g(\la^1,\la^2)= g(\la^2,\la^1)$ for all $(\la^1,\la^2)  \in V$. \label{HinfV}
\index{$H^\infty_\sym(V)$}
 We say that $V$ has the {\em symmetric extension property} if $V$ has the $H^\infty_\sym(V)$-extension property.
\index{extension property!symmetric}
In this section we shall describe all symmetric algebraic sets in $\d^2$ that have the symmetric extension property.  It is striking that there are three species of set in $\d^2$ that have the symmetric extension property but do not resemble any of the types in Theorem \ref{jim&john}.

The symmetric extension property in $\d^2$ is closely related to the \npep \ in $G$.  We shall denote by $t$ the transposition map $t(\la^1,\la^2)=(\la^2,\la^1)$.
\index{transposition map}
\begin{lemma}\label{exprops}
A symmetric subset $V$ of $\d^2$ has the symmetric extension property if and only if $\pi(V)$ has the \npep \ in $G$.
\end{lemma}
\begin{proof}
Note that since $V$ is symmetric, $V=\pi^{-1}(\pi(V))$.

Suppose that $V\subseteq \d^2$ has the symmetric extension property.  Let $f$ be a bounded holomorphic function on $\pi(V)$; then $f\circ \pi$ is a bounded symmetric function on $V$.  Moreover, $f\circ\pi$ is holomorphic on $V$.  For consider any $\la\in V$.  Since $f$ is holomorphic, there is a neighborhood $U$ of $\pi(\la)$ in $G$ and a holomorphic function $g$ on $U$ which agrees with $f$ on $U\cap \pi(V)$.  Then $g\circ\pi$ is a holomorphic function on the neighborhood $\pi^{-1}(U)$ of $\la$ which agrees with $f\circ\pi$ on $\pi^{-1}(U)\cap V$.  Thus $f\circ\pi$ is a symmetric bounded {\em holomorphic} function on $V$.  Since $V$ has the symmetric extension property, there exists a bounded holomorphic function $f_1$ on $\d^2$ which agrees with $f\circ\pi$ on $V$ and satisfies
\[
\sup_{\d^2} |f_1| = \sup_V |f\circ \pi|.
\]
Let $\tilde f= \half (f_1+ f_1\circ t)$.  Then $\tilde f$ is a symmetric bounded holomorphic function on $\d^2$ which agrees with $f\circ\pi$ on $V$ and satisfies  
\[
\sup_{\d^2} |\tilde f| = \sup_V |f\circ \pi|.
\]
There exists a holomorphic function $F$ on $G$ such that $\tilde f=F\circ\pi$.  Then $F$ agrees with $f$ on $\pi(V)$ and
\[
\sup_G |F|= \sup_{\d^2} |\tilde f| = \sup_V |f\circ\pi|= \sup_{\pi(V)} |f|.
\]
Hence $\pi(V)$ has the \npep \  in $G$.

Conversely, suppose that $\pi(V)$ has the \npep \ in $G$.  
We must show that $V$ has the symmetric extension property.  To this end consider any bounded symmetric holomorphic function $f$ on $V$.  
There exists a unique function $F$ on $\pi(V)$ such that $f= F\circ \pi$.  We claim that $F$ is holomorphic on $\pi(V)$.  

Consider any $s\in\pi(V)$ and let $\mu\in V$ be a preimage of $s$ under $\pi$. 
  Since $f$ is holomorphic on $V$ there exists a neighborhood $U$ of $\mu$ in $\d^2$ and a holomorphic function  $g$ on $U$ such that $f$ and $g$ agree on $U\cap V$.  We claim that both $U$ and $g$ can be taken to be symmetric.  Indeed,  if $t(\mu)=\mu$ then define $U_1$ to be $U\cap t(U)$ and let $g_1=\half(g+g\circ t)$; then $U_1$ is a symmetric neighborhood of $\mu$ and $g_1$ is a symmetric holomorphic function on $U_1$ such that $g_1$ and $f$ agree on $ U_1\cap V$.  If $t(\mu)\neq \mu$, let $U_0$ be an open ball centred at $\mu$ and contained in $U$ such that $U_0$ is disjoint from $ t(U_0)$, let $U_1=U_0\cup t(U_0)$ and extend $g$ to a symmetric holomorphic function on $U_1$ by $g_1(\la)=g\circ t(\la)$ for $\la\in t(U_0)$.  In either case $U_1$ is a symmetric neighborhood of $\mu$ and $g_1$ agrees with $f$ on $U_1\cap V $.

  The symmetric function $g$ determines a holomorphic function $H$ on the neighborhood $\pi(U)$ of $s$ such that $g=H\circ \pi$ on $U$.  Then $H$ is a holomorphic function on the neighborhood $\pi(U)$ of $s$ such that $H$ and $F$ agree on $\pi(U)\cap\pi(V)$.  Thus $F$ is bounded and holomorphic on   $\pi(V)$.

Since $\pi(V)$ has the \npep, there exists a bounded holomorphic function $\tilde F$ on $G$ which extends $F$ and satisfies
\[
\sup_G |\tilde F|= \sup_{\pi(V)} |F|= \sup_{V} |f|.
\]
Then $\tilde F \circ \pi$ is a bounded holomorphic function on $\d^2$ that extends $f$ and has the same supremum norm as $f$.
Hence $V$ has the symmetric extension property in $\d^2$.
\end{proof}
\begin{lemma}\label{piValg}
If $V$ is a symmetric algebraic set in $\d^2$ then $\pi(V)$ is an algebraic set in $G$.
\end{lemma}
\begin{proof}
  Suppose that 
\[
V= \{ \la\in\d^2: f(\la)=0 \mbox{ for all } f\in S\}
\]
for some set $S \subseteq \c[\la^1,\la^2]$.  Since $V$ is symmetric, for $\la\in\d^2$ we have $\la\in V$ if and only if $t(\la)\in V$.  Hence $\la\in V$ if and only if both $(f+f\circ t)(\la)=0$ and $(f-f\circ t)^2(\la)=0$ for all $f\in S$.  On replacing $S$ by the set
\[
\{ f+f\circ t : f\in S\} \cup \{(f-f\circ t)^2: f\in S\},
\]
we may assume that $S$ is a set of {\em symmetric} polynomials, and so can be expressed in the form 
\[
S=\{ F\circ \pi: F\in S^\flat\}
\]
for some set $S^\flat\subseteq \c[s^1,s^2]$.  Then 
\[
\pi(V)=\{s\in G: F(s)=0 \mbox{ for all } F\in S^\flat\}.
\]
Hence $\pi(V)$ is an algebraic set.
\end{proof}
Recall that the notion of a balanced disc in $\d^2$ was given in Definition \ref{def_bal_datum_D2}.  It is a subset $D$ of $\d^2$ having the form $D=\{(z,m(z)):z\in\d\}$ for some $m\in\aut\d$.   
\begin{theorem}\label{allsymexprop}
A symmetric algebraic set $V$ in $\d^2$ has the symmetric extension property if and only if 
one of the following six alternatives holds.
\begin{enumerate}[\rm (1)]
\item  $V=\{\la,t(\la)\}$ for some $\la\in\d^2$;
\item $V=\d^2$;
\item $V= D\cup t(D)$ for some balanced disc $D$ in $\d^2$ such that $D^-$ meets the set $\{(z,z): z\in\t\}$;
\item   $V=V_\beta$ for some $\beta\in\d$,  where
\be\label{defVbeta}
V_\beta \df \{(z,w)\in\d^2:  z+w= \beta + \bar\beta zw\};
\ee
\item $V=\Delta\cup V_\beta $ for some $\beta\in\d$, where $\Delta=\{(z,z): z\in\d\}$;
\item $V=V_{m,r}$ for some $r\in(0,1)$ and $m\in\aut\d$, where
\be\label{Vmr}
V_{m,r}\df\{(z,w)\in\d^2: H_r(m(z),m(w))=0\}
\ee 
and
\be\label{defHr}
H_r(z,w) \df  2zw(r(z+w)+2-2r) - (1+r)(z+w)^2 + 2r(z+w).
\ee
\end{enumerate}
Moreover, the six types of sets $V$ in {\rm (1)} to {\rm (6)} are mutually exclusive.
\end{theorem}
\index{theorem!symmetric algebraic sets in $\d^2$ with the symmetric extension property}
\begin{proof}
Let $V$ be a symmetric algebraic set in $\d^2$.  By Lemma \ref{piValg}, $\pi(V)$ is an algebraic set in $G$.
By Lemma \ref{exprops}, $V$ has the symmetric extension property if and only if 
$\pi(V)$ has the \npep \ in $G$.  By Theorem \ref{main}, $\pi(V)$ has the \npep \  if and only if either $\pi(V)$ is a retract in $G$ or $\pi(V)=\calr\cup\cald$ for some flat geodesic $\cald$ of $G$.  

Suppose that $V$ has the symmetric extension property. Consider the two cases (i) $\pi(V)$ is a retract and (ii)  $\pi(V)=\calr\cup\cald$ for some flat geodesic $\cald$.

In case (ii), by Proposition \ref{calfbeta}, $\cald = \calf_\beta$ for some $\beta\in\d$, and so $V=\pi^{-1}(\calr\cup \calf_\beta)= \Delta\cup V_\beta$, as in alternative (5) of the theorem.

In case (i), by Theorem \ref{main2}, the retracts in $G$ are the singleton sets, $G$ itself and the complex geodesics in $G$. If $\pi(V)$ is either a singleton or $G$ then alternative (1) or (2) respectively holds.

The remaining possibility in case (i) is that $\pi(V)$ is a complex geodesic in $G$, so that $\pi(V)=\cald_\de$ for some \nd datum $\de$, of one of the five types described in Definition \ref{extdef10}.   If $\de$ is flat then $\cald = \calf_\beta$ for some $\beta\in\d$, and $V=V_\beta$, so that alternative (4) holds.
By Proposition \ref{prop3.13}, if $\de$ is balanced then there exists a balanced disc $D$ in $\d^2$ such that $\pi(D)=\pi(V)$, and hence $V=D \cup t(D)$.   Moreover, if $k$ solves $\Kob(\de)$, then $k(z)=(z+m(z),zm(z))$ for some $m\in\aut\d$,  and the disc $D$ is given by $\{(z,m(z)):z\in \d\}$.
Furthermore, since $k$ has a royal node in $\t$, the disc $D^-$ meets $\{(z,z):z\in\t\}$. Thus alternative (3) holds.

There remains the case that $\pi(V)$ is a purely unbalanced geodesic.  By Theorem \ref{formgeos}, then $\pi(V) \sim k_r(\d)$ for some $r\in (0,1)$, where
\be\label{formkrbis}
k_r(z) = \frac{1}{1-r z}\left(2(1-r)z, z(z-r)\right) \quad \mbox{ for all } z\in\d.
\ee
Let $h_r(z)=k_r(-z)$.  Then $h_r(\d)=k_r(\d)$ and one finds that $\Phi_1\circ h_r=\id{\d}$.  We may apply Proposition \ref{retprop10} to write $h_r(\d)$ as the zero set of a polynomial.  We have $h_r= \tilde q_r/q_r$, where
\[
q_r(z)=1+rz, \qquad \tilde q_r(z)= z(z+r).
\]
As in equation \eqref{defP}, let
\begin{align*}
P_r(s)&= (2-\omega s^1)^2 \; \frac{ \tilde q_r\circ \Phi_\omega(s) - s^2 q_r\circ\Phi_\omega(s)}{s^2-\bar\omega^2} \\
	&= 2s^2(rs^1+2-2r) - (1+r)(s^1)^2 + 2rs^1.
\end{align*}
By Proposition \ref{retprop10}, $k_r(\d)= G\cap P_r^{-1}(0)$.  Hence $\pi(V) = G\cap  \tilde m^{-1}(P_r^{-1}(0))$ for some $m\in\aut\d$.
That is, there exist $r\in(0,1)$ and $m\in\aut\d$ such that $(z,w) \in V$ if and only if
\[
P_r ( m(z)+m(w),m(z)m(w)) =0.
\]
Since the function $H_r$ is defined in equation \eqref{defHr} so that $H_r(z,w)=P_r(z+w,zw)$, we have $V=V_{m,r}$, and alternative (6) of the theorem holds.

We have shown that if $V$ has the symmetric extension property then $V$ is of one of the forms (1) to (6).  We prove the converse.  If either of alternatives (1) or (2) holds, then $\pi(V)$ is a singleton or $G$ respectively and so $\pi(V)$ has the \npep.  If (5) holds then $\pi(V)= \calr\cup \calf_\beta$, and so, by Theorem \ref{main}, $\pi(V)$ has the \npep \ in $G$. 

We claim that, in the remaining cases (3), (4) and (6), $\pi(V)$ is a complex geodesic in $G$.

If (3) holds then  $D=\{(z,m(z)):z\in\d\}$ for some $m\in\aut\d$ having a fixed point in $\t$.   The map  $k(z)=(z+m(z),zm(z))$ is a complex C-geodesic and so $\pi(V)=\pi(D)=k(\d)$ is a geodesic of degree $2$.

If (4) holds and $V=V_\beta$ for some $\beta \in\d$, then $\pi(V)$ is the flat geodesic $\calf_\beta$.

 If (6) holds and $V=V_{m,r}$ for some $m\in\aut\d$ and $r\in (0,1)$, then $\pi(V)\sim k_r(\d)$, where $k_r$ is given by equation \eqref{formkrbis}, and so, by Theorem \ref{formgeos}, $\pi(V)$ is a (purely unbalanced) geodesic in $G$.  Thus $\pi(V)$ has the \npep \ in $G$. 

In each of the six cases $\pi(V)$ has the \npep \ in $G$, and therefore
$V$ has the symmetric extension property in $\d^2$.

To see that the classes (1) to (6) of sets $V$ in $\d^2$ with the symmetric extension property are mutually exclusive, one can simply consider the corresponding sets $\pi(V)$.  Our reasoning shows that, in cases (1) to (6), the corresponding set $\pi(V)$ is respectively (1) a singleton (2) $G$ (3) an exceptional, purely balanced or royal geodesic (4) a flat geodesic (5) $\calr\cup \calf_\beta$ for some $\beta$ and (6) a purely unbalanced geodesic.
\end{proof}

\chapter{Applications to the theory of spectral sets}\label{AvonN}

Two fundamental results in the harmonic analysis of operators on Hilbert space \cite{nf} are the inequalities of von Neumann and And\^o \cite{neu,and63}, which can be expressed by the statements
\begin{enumerate}[\rm (1)]
\item the closed unit disc is a spectral set for every contraction, and
\item the closed unit bidisc is a spectral set for every pair of commuting contractions.
\end{enumerate}
In this chapter an {\em operator} means a bounded linear operator on a Hilbert space, and a {\em contraction} means an operator of norm at most $1$.
\index{operator}
\index{contraction}
A {\em spectral set} for a commuting $n$-tuple $T$ of operators is a set $V\subseteq \c^n$ such that $\si(T) \subseteq V$ and, for every holomorphic function $f$ in a neighborhood of $V$,
\[
\|f(T)\| \leq \sup_V |f|.
\]
\index{set!spectral}
The terminology {\em spektralische Menge} is due to von Neumann \cite{neu}.  Spectral sets for commuting tuples of operators are commonly defined to be closed subsets of $\c^n$, but in \cite{agmc_vn} the notion was broadened considerably, with the goal of a refinement of And\^o's and von Neumann's inequalities.

\begin{definition}\label{subordinate} 
Let $V$ be a subset of $\c^n$ and $T$ be an $n$-tuple of commuting operators on a Hilbert space. $T$ is {\em subordinate} to $V$ if the spectrum $\sigma(T)$ is a subset of $V$ and every holomorphic function on a neighborhood of $V$ that vanishes on $V$ annihilates $T$.
\end{definition}
\index{subordinate}
Clearly, if $T$ is subordinate to $V$ and $g$ is the restriction to $V$ of a holomorphic function $f$ on a neighborhood of $V$ then we may uniquely define $g(T)$ to be $f(T)$, where $f(T)$ is defined by the Taylor functional calculus.  We denote by $H^\infty(V)$ the algebra of functions $f|V$ where $f$ is bounded and holomorphic in some neighborhood $U_f$ of $V$.  Thus, if $T$ is subordinate to $V$ then the map $g \mapsto g(T)$ is a functional calculus for $H^\infty(V)$.
\index{$H^\infty(V)$}

\begin{definition}\label{Aspecset} 
Let $V\subseteq\c^n$, let $A\subseteq H^\infty(V)$ and let $T$ be an $n$-tuple of commuting operators. $V$ is an {\em $A$-spectral set} for $T$ if $T$ is subordinate to $V$ and, for every $f\in A$,
\be\label{specsetcond}
\|f(T)\| \leq \sup_V |f|.
\ee
\end{definition}
\index{set!$A$-spectral}
Another formulation of Ando's inequality is that $\d^2$ is a spectral set for any commuting pair of contractions whose joint spectrum is contained in $\d^2$.  Isolating the role of $\d^2$ in this statement and generalizing it to arbitrary subsets of $H^\infty(V)$, in \cite[Definition 1.12]{agmc_vn} the authors introduced the following notion.
\begin{definition}\label{A-von_N_set_D2} 
Let $V \subseteq \c^2$ and let $A \subseteq {H}^\infty(V)$. Then $V$ is an {\em $A$-von Neumann set } if the inequality
\[
\| f(T)\| \le \sup_{V} |f|
\]
holds for all $f \in A$ and all pairs $T$ of commuting contractions which are subordinate to $V$, in the sense of Definition \ref{subordinate}. 
\end{definition}
\index{set!$A$-von Neumann}
Thus $V$ is an $A$-von Neumann set if $V$ is an $A$-spectral set for every pair $T$ of commuting contractions which is subordinate to $V$.

One of the main results of \cite{agmc_vn}, Theorem 1.13, states that if $V\subseteq \d^2$ and $A\subseteq H^\infty(V)$ then $V$ is an  $A$-von Neumann set if and only if  $V$ has the $A$-extension  property relative to $\d^2$.
In the case that $V$ is a symmetric subset of $\d^2$ and $A$ is the algebra $H^\infty_\sym(V)$ (see page \pageref{HinfV}),
Theorem \ref{allsymexprop} enables us to give an explicit description of the $A$-von Neumann sets in $\d^2$.

\begin{theorem}\label{allsymexprop-appl}
Let $V$ be a symmetric algebraic set in $\d^2$. Then $V$ is an $H^\infty_\sym(V)$-von Neumann set if and only if $V$ has one of the  the six forms {\rm (1)} to {\rm(6)} in Theorem \ref{allsymexprop}.
\end{theorem}
\index{theorem!symmetric algebraic sets in $\d^2$ which are $H^\infty_\sym(V)$-von Neumann sets}
\begin{proof} By \cite[Theorem 1.13]{agmc_vn},  $V$ is an  $H^\infty_\sym(V)$-von Neumann set in $\d^2$ if and only if  $V$ has the $H^\infty_\sym(V)$-extension  property relative to $\d^2$.  By Theorem \ref{allsymexprop}, the latter is so if and only if $V$ has one of the six forms (1) to (6) in the theorem.
\end{proof}

The $A$-von Neumann sets of Definition \ref{A-von_N_set_D2} are very much tied to the bidisc.  One can define a similar notion for other subsets of $\c^2$.
Let us illustrate with the symmetrized bidisc.
\begin{definition}\label{gaAvnset}
A pair $T$ of commuting bounded linear operators is a {\em $\Gamma$-contraction} 
if $\Gamma$ is a spectral set for $T$. 

Let $V\subseteq G$ and let $A\subseteq H^\infty(V)$.  Then $V$ is a {\em $(G, A)$-von Neumann set} if
$V$ is an $A$-spectral set for every $\Ga$-contraction $T$ subordinate to $V$.
\end{definition}
\index{$\Ga$-contraction}
\index{set!$(G,A)$-von Neumann}

\begin{theorem}\label{A-ext=A-von_N}  
Let $V\subseteq G$ and let $A\subseteq H^\infty(V)$. Then
$V$ is a $(G,A)$-von Neumann set if and only if $V$ has the $A$-extension property relative to $G$. 
\end{theorem}
\index{theorem!$(G,A)$-von Neumann sets are sets with $A$-extension property relative to $G$}
\begin{proof} Suppose that $V$ has the $A$-extension property relative to $G$. Consider any $f \in A$. There exists $g \in {H}^\infty(G)$ such that $g|V=f$ and
\be\label{spect_ineq_A_G}
\sup_G |g| = \sup_V |f|.
\ee
Let $T$ be a $\Gamma$-contraction subordinate to $V$:  then
\[
\|g(T)\| \leq \sup_G |g|,
\]
and so,by equation  \eqref{spect_ineq_A_G},
\be \label{f-Gamma-cont}
\| f(T)\| = \| g(T)\| \leq \sup_{\Gamma} |g |= \sup_{V} |f |.
\ee
This inequality holds for all $f\in A$, and so $V$ is an $A$-spectral set for $T$.
Since this is true for every $\Ga$-contraction $T$ subordinate to $V$, it follows that $V$ is a $(G,A)$-von Neumann set.

Conversely, suppose that $V$ is a $(G,A)$-von Neumann set.
We claim that $\pi^{-1}(V)$ is  an  $A\circ \pi$-von Neumann set contained in $\d^2$. Consider a commuting pair of contractions $T$
subordinate to $\pi^{-1}(V)$. Then $\pi(T)$ is a $\Ga$-contraction and 
\[
\sigma(\pi(T))=\pi(\si(T)) \subseteq \pi(\pi^{-1}(V))=V.
\]
 Let us show that  $\pi (T)$ is subordinate to $V$.
Consider $g$ holomorphic on a  neighborhood of $V$ and such that $g|V=0$. Then $g\circ\pi$ is holomorphic on a  neighborhood of $\pi^{-1}(V)$
and is zero on $\pi^{-1}(V)$. Since $T$ is 
subordinate to $\pi^{-1}(V)$, $ g\circ\pi(T) =0$ and hence $\pi (T)$ is subordinate to $V$.

By assumption, $V$ is a $(G,A)$-von Neumann set, and so $V$ is an $A$-spectral set for $\pi(T)$.  Therefore
\be\label{spect_ineq_A}
\| f\circ \pi(T)\| \le \sup_{V} |f |= \sup_{\pi^{-1}(V)} |f\circ \pi|
\ee
for all $f \in A$. Thus, for all $F \in A \circ \pi$, 
\be\label{spectA}
\| F(T)\| \le \sup_{\pi^{-1}(V)} |F|
\ee
whenever $T$ is a  commuting pair of contractions 
subordinate to $\pi^{-1}(V)$.
That is,  $\pi^{-1}(V)$ is  an  $A\circ \pi$-von Neumann set in $\d^2$.

By  \cite[Theorem 1.13]{agmc_vn}, a subset $W$ of $\d^2$ is an $A\circ \pi$-von Neumann set if and only if  $W$ has the $A \circ \pi$-extension  property relative to $\d^2$. On applying this result to $W=\pi^{-1} (V)$ we deduce that $\pi^{-1} (V)$ has the $A \circ \pi$-extension  property relative to $\d^2$.  It follows that $V$ has the $A$-extension property relative to $G$. Indeed, consider $f \in A \subseteq{H}^\infty(V)$.
Since $\pi^{-1} (V)$ has the $A \circ \pi$-extension  property in $\d^2$,
there exists a bounded holomorphic function $g$ on $\d^2$ which agrees with $f\circ\pi$ on $\pi^{-1} (V)$ and satisfies
\[
\sup_{\d^2} |g| = \sup_{\pi^{-1} (V)} |f\circ \pi| = \sup_{V} |f|.
\]
Let $\tilde g= \half (g+ g\circ t)$.  Then $\tilde g$ is a symmetric bounded holomorphic function on $\d^2$ which agrees with $f\circ\pi$ on $\pi^{-1} (V)$ and satisfies  
\[
\sup_{\d^2} |\tilde g| = \sup_{\pi^{-1} (V)} |f\circ \pi|.
\]
There exists a holomorphic function $F$ on $G$ such that $\tilde g=F\circ\pi$.  Then $F$ agrees with $f$ on $V$ and
\[
\sup_G |F|= \sup_{\d^2} |\tilde g| = \sup_{\pi^{-1} (V)} |f\circ\pi|= \sup_{V} |f|.
\]
Hence $V$ has the $A$-extension property relative to $G$.
\end{proof}

The same theorem is stated in \cite{bhatta2}.

In the event that $A=H^\infty(V)$ for a subset $V$ of $G$, we can describe all $(G,A)$-von Neumann sets.

\begin{theorem}\label{spectral_sets_G} 
Let $V$ be an algebraic subset of $G$. Then
 $V$ is a  $(G,{H}^\infty(V))$-von Neumann set in $G$ if and only if either $V$ is a retract in $G$ or $V=\calr\cup \cald$ for some flat geodesic $\cald$ in $G$,  where $\calr$ is the royal variety. 
\end{theorem}
\index{theorem!algebraic sets in $G$ which are $(G,{H}^\infty(V))$-von Neumann sets}
\begin{proof} By Theorem \ref{A-ext=A-von_N}, $V$ is a  $(G,{H}^\infty(V))$-von Neumann set if and only if $V$ has the ${H}^\infty(V)$-extension property relative to $G$. By Theorem \ref{main}, the latter holds if and only if  either $V$ is a retract in $G$ or $V=\calr\cup \cald$ for some flat geodesic $\cald$ in $G$. 
\end{proof}

\chapter[Anomalous sets with the extension property in  other domains]{Anomalous sets with the norm-preserving extension property in some other domains}\label{SpecBall}

At little extra cost we can deduce from Theorems \ref{main2} and  \ref{main} that there are anomalous sets having the \npep\ in some other domains besides $G$ -- to wit, domains containing $G$ as a retract.    In particular, the statement holds for the $2\times 2$ spectral ball, the tetrablock and the pentablock, defined below.

\begin{lemma}\label{OGretract}
If $\Omega$ is a domain in $\c^d$ such that $G$ can be embedded in $\Omega$ as a retract then there exist sets which have the \npep\  but are not retracts in $\Omega$. 
\end{lemma}
\begin{proof}
$G$ can be embedded in $\Omega$ as a retract if and only if there exist $\iota\in\Omega(G)$ and $\kappa\in G(\Omega)$ such that $\kappa$ is a left inverse of $\iota$.

Suppose such maps $\iota$ and $\kappa$ exist. 
  Let $\cald$ be a flat geodesic in $G$ and let $V=\calr\cup \cald$.  By Theorem \ref{main}, $V$ is a subset of $G$ that has the \npep.  By Lemma \ref{npepretr} below, $\iota(V)$ has the \npep\ in $\Omega_2$.

However, $\iota(V)$ is not a retract of $\Omega$.  For suppose that $\rho$ is a retraction in $\Omega$ with range $\iota(V)$.
Let 
\[
\rho_1= \kappa\circ\rho\circ\iota: G\to G;
\]
then
\[
\rho_1\circ\rho_1= \kappa\circ\rho\circ(\iota\circ\kappa)\circ\rho\circ\iota.
\]
Now $\iota\circ\kappa$ acts as the identity mapping on $\iota(G)$, and so {\em a fortiori} on $\ran \rho$.  Hence
\begin{align*}
\rho_1\circ\rho_1&= \kappa\circ\rho\circ\rho\circ\iota \\
	&=\kappa\circ\rho\circ\iota \\
	&=\rho_1.
\end{align*}
Thus $\rho_1$ is a retraction on $G$ with range $\kappa\circ\iota(V)$, which is $V$.
Therefore $V$ is a nontrivial retract in $G$.  This contradicts Theorem \ref{main2}, since $V$ is clearly not a geodesic (it is not homeomorphic to a disc).  Hence $\iota(V)$ is not a retract of $\Omega_2$.
\end{proof}
\begin{lemma}\label{npepretr}
Let $U_1, U_2$ be domains and suppose that $\iota\in U_2(U_1), \, \kappa\in U_1(U_2)$ satisfy $\kappa\circ\iota=\id{U_1}$.  If a subset $V$ of $\;U_1$ has the \npep\ in $U_1$ then $\iota(V)$ has the \npep\ in $U_2$.
\end{lemma}
\begin{proof}
Consider a bounded holomorphic function $f$ on $\iota(V)$.  Then $f\circ \iota$ is bounded and holomorphic on $V$, and so has a norm-preserving extension $g$ to $U_1$.  Thus $g|V=f\circ \iota$ and
\[
\sup_{U_1} |g|=\sup_V |f\circ \iota| = \sup_{\iota(V)} |f|.
\]
Now $g\circ \kappa$ is bounded and holomorphic on $U_2$, and
\[
\sup_{U_2} |g\circ\kappa|=\sup_{U_1} |g| =\sup_{\iota(V)} |f|.
\]
Moreover, for any point of $\iota(V)$, say $\iota(z)$ where $z\in V$,
\[
g\circ\kappa(\iota(z))=g\circ\kappa\circ\iota(z)=g(z)=f\circ\iota(z).
\]
Hence $g\circ\kappa$ is a norm-preserving extension of $f$ to $U_2$.
\end{proof}

The {\em $n\times n$ spectral ball} is the set
\[
\Omega_n \df \{ A\in \c^{n\times n}: r(A) < 1\}
\]
where $r(\cdot)$ denotes the spectral radius of a square matrix.  
\index{spectral ball}
\index{$\Omega_n$}
The domain $\Omega_n$ is of interest principally in the theory of invariant distances and metrics \cite{jp,nipf,nithzw,andrist,andkut}.
Holomorphic interpolation from $\d$ to $\Omega_n$, the `spectral Nevanlinna-Pick problem', is a test case of the $\mu$-synthesis problem of $H^\infty$ control (see for example \cite{bft,Y}).
\index{spectral Nevanlinna-Pick problem}

The symmetrized bidisc was introduced as a tool for the study of $\Omega_2$, and so it is natural to consider the implications of our main theorems for the function theory of $\Omega_2$.

There are many ways in which $G$ can be embedded as a retract in $\Omega_2$.  Let $\kappa:\c^{2\times 2} \to \c^2$ be defined by $\kappa(A)=(\tr A, \det A)$.  Clearly a matrix $A\in\c^{2\times 2}$ belongs to $\Omega_2$ if and only if $\kappa(A) \in G$.  For any holomorphic function
\be\label{FF}
F:G\to GL_2(\c)
\ee
define a holomorphic map $\iota_F:G \to \Omega_2$ by
\be\label{iotaF}
\iota_F(s)= F(s)^{-1}\bbm 0&1\\ -s^2 & s^1 \ebm F(s).
\ee
Then $\kappa\circ\iota_F=\id{G}$, and so $\iota_F$ is injective and $\iota_F\circ\kappa$ is a retraction in $\Omega_2$.  Thus $\iota_F(G)$ is a retract in $\Omega_2$.

The {\em tetrablock} is the domain
\[
\mathcal{E} \df \{ (a_{11},a_{22}, \det A): A=\bbm a_{ij}\ebm_{i,j=1}^2, \ \|A\| <1\}
\]
in $\c^3$.  There are many equivalent characterizations of this domain \cite{awy}.
\index{$\mathcal E$}
\index{tetrablock}
It is easy to show that the map $\iota: G \to \mathcal E$ given by 
\[
\iota(s)=(\half s^1,\half s^1,s^2)
\]
is a holomorphic injection with left inverse
\[
\kappa (z)= (z^1+z^2, z^3)
\]
that maps $\mathcal E$ to $G$.  Thus $E$ contains $G$ as a retract.

The {\em pentablock} is the domain
\[
\mathcal{P}= \{ (a_{21}, \tr A, \det A):  A=\bbm a_{ij}\ebm_{i,j=1}^2, \ \|A\| <1\}
\]
in $\c^3$. 
\index{$\mathcal P$}
\index{pentablock}
The pentablock also admits numerous characterizations \cite{penta}.  The map $\iota: G\to \mathcal P$ given by
\[
\iota(s)=(0,s^1,s^2)
\]
has a left inverse $\kappa:\mathcal{P}\to G$ given by $\kappa(z^1,z^2,z^3)=(0,z^2,z^3)$.   Thus $G$ is embeddable as a retract in $\mathcal P$.

The tetrablock and the pentablock are both of interest in the theory of invariant distances \cite{jp} and in the analysis of special cases of the $\mu$-synthesis problem \cite{Y,penta}.

The following statement is an immediate consequence of Lemma \ref{OGretract}.
\begin{theorem} \label{3domains}
The $2\times 2$ spectral ball $\Omega_2$, the tetrablock and the pentablock all contain sets having the \npep\ which are not retracts in the respective domains.
\end{theorem}
\begin{remark} For a domain $U$ in $\c^N$, $ N \ge 3$, one cannot expect  all nontrivial retracts to be complex geodesics. For example, let $\Omega$ be one of the following domains: the $2\times 2$ spectral ball $\Omega_2$, the tetrablock and the pentablock. Then 
the image of $G$ under $\iota$ in $\Omega$  is a retract which is not a complex geodesic (since it has complex dimension $2$). 
\end{remark}

\appendix
\chapter{Some useful facts about the symmetrized bidisc}\label{appendix}
In this appendix we gather together some established properties of the symmetrized bidisc $G$ relevant to the paper, with appropriate citations.   We also prove analogs for infinitesimal datums of some  statements in the literature about discrete datums.  An alternative source for many of these results is \cite[Chapter 7]{jp}.

\section{Basic properties of $G$ and $\Gamma$}\label{basic}
There are numerous criteria for a point of $\c^2$ to belong to $G$.
\begin{proposition}\label{elG}
The following statements are equivalent for any $s=(s^1,s^2) \in\c^2$.
\begin{enumerate}[\rm (1)]
\item $s\in G$ (that is, there exist $z, w\in\d$ such that $s=(z+w,zw)$);
\item $|s^1 - \overline{s^1} s^2| < 1-|s^2|^2$;
\item $|\Phi_\omega(s)| < 1$ for all $\omega\in\t$;
\item there exists $\beta\in \d$ such that $s^1=\beta +\bar\beta s^2$;
\item  $2|s^1-\overline{s^1}s^2| + |(s^1)^2-4s^2| < 4-|s^1|^2$.
\end{enumerate}
\end{proposition}
The proof of this proposition can be found in \cite[Theorem 2.1]{AY04}.
The functions $\Phi_\omega, \, \omega\in\t$, in (3) are those introduced in Definition \ref{defPhi}.

A simple calculation shows that, for any $s\in G$, the corresponding $\beta$ in (4) is unique and is
 given by
\be\label{appformbeta}
\beta = \frac{s^1-\overline{s^1}s^2}{1-|s^2|^2}.
\ee
Recall from Section \ref{cgeosdtmsG} that $\Ga$ denotes the closure of $G$ and $b\Ga$ the distinguished boundary of $\Ga$.  It means that $b\Ga$ is the Shilov boundary of the uniform algebra of continuous complex-valued functions on $\Ga$ that are analytic in $G$.  
\index{$b\Ga$} 
The following descriptions of $b\Ga$ are to be found in \cite[Theorem 2.4, equation (2.1)]{AY04}.
\begin{proposition}\label{bGam}
\begin{align*}
b\Ga&= \{(z+w,zw): |z|=1=|w|\} \\
	&=\{s\in\c^2: |s^1| \leq 2, \, |s^2|=1 \mbox{ and } s^1=\overline{s^1} s^2\}.
\end{align*}
\end{proposition}

A {\em $\Ga$-inner function} is a holomorphic map $h:\d\to G$ such that, for almost all $\la\in\t$ with respect to Lebesgue measure, the radial limit
$\lim_{r\to 1-} f(r\la)\in b\Ga$.  Rational $\Ga$-inner functions were defined in Definition \ref{Gainner}.
\begin{proposition}\label{phiok}
If $k$ is a nonconstant $\Ga$-inner function and $\omega\in\t$ then $\Phi_\omega\circ k$ is an inner function.
\end{proposition}
\begin{proof}
For if $k^*$ is the radial limit function of $k$ then, for almost all $\tau\in\t$ there exist $z,w \in\t$, not both equal to $\bar\omega$, such that $k^*(\tau)=(z+w,zw)$.  Then
\begin{align*}
\left|\Phi_\omega\circ k^*(\tau)\right| &= \left|\frac{2\omega zw-z-w}{2-\omega(z+w)}         \right|\\
	&=\left| \omega zw \frac{2-\bar\omega\bar w-\bar\omega\bar z}{2-\omega z-\omega w}\right| \\
	&= 1.
\end{align*}
\end{proof}
The functions $\Phi_\omega, \, \omega\in\t$, are central to the theory of $G$.  Their most significant property is the following.
\begin{theorem}\label{Phiunivl}
The functions $\Phi_\omega, \ \omega\in\t$, constitute a universal set for the Carath\'eodory problem on $G$
\end{theorem}
This statement is in  \cite[Theorem 1.1 and Corollary 4.3]{AY04}.  

It is easy to see that, for any $\omega\in\t$, the function $\Phi_\omega$ extends continuously to $\Ga\setminus \{(2\bar\omega,\bar\omega^2)\}$ (but not to $\Ga$).  The range of the extended function is $\d^-$. It is useful to know which points are mapped into $\t$.
\begin{proposition}\label{modPhi1}
For any $\omega\in\t$ and $s\in\Ga\setminus \{(2\bar\omega,\bar\omega^2)\}$,
\[
|\Phi_\omega(s)|=1 \quad\mbox{ if and only if }\quad \omega(s^1-\overline{s^1}s^2)= 1-|s^2|^2.
\]
\end{proposition}
The proof is a simple calculation and is given in \cite[Theorem 2.5]{AY04}.

\begin{theorem} \label{crit-canc=royal} Let $h =(s^1,s^2)$ be a nonconstant rational $\Gamma$-inner function and let  $\ups$ be a finite Blaschke product. Then   $\Phi_{\ups} \circ h $
has a cancellation at $\zeta$ if and only if the following conditions are satisfied: $\zeta \in \t$, $\zeta$ is a royal node for $h$ and  $\ups(\zeta)=\frac{1}{2}\overline{ s^1(\zeta)}$.
Moreover $\Phi_{\ups} \circ h$ has at most one cancellation at any royal node $\zeta$.
\end{theorem}

The proof of this theorem  is given in \cite[Theorem 7.12]{ALY13}.

\section{Complex C-geodesics in $G$}\label{C-geod}
Recall from Section \ref{cgeos} that a map $k\in G(\d)$ is a complex C-geodesic if $k$ has a holomorphic left inverse.  Such maps can be concretely described.
The following statement is \cite[Lemma 1.1 and Theorem 1.2]{AY06}.  
\begin{theorem}\label{descgeos}
Let $k: \d \to G$
be a complex C-geodesic of $G$.  Then $k^2$ is a Blaschke product of degree $1$
or $2$, $k$ extends to a continuous function on $\d^-$ which maps $\t$
into $b\Gamma$ and there exist  a M\"obius function $\upsilon$ and an 
$\omega_0 \in \t$ such that $\Phi_{\omega_0}\circ k = \upsilon$.
Furthermore,
$$
k^2\circ \upsilon^{-1}(\bar\omega_0) = \bar\omega_0^2
$$
and, for all $\la \in \d$,
$$
k^1(\la) = 2\frac{\omega_0 k^2(\la) - \upsilon(\la)}{1 - \omega_0 \upsilon(\la)}.
$$
{\em A fortiori}, $k$ is a rational $\Ga$-inner function.
\end{theorem}

The following statement is a simple corollary.
\begin{proposition}\label{cancels} 
Let $\de$ be a \nd non-flat datum in $G$.  If $\Phi_\omega$ solves $\Car(\de)$,  $k$ solves $\Kob(\de)$ and $k$ is normalised so that $\Phi_\omega \circ k = \id{\d}$, then  $k^2$ is a Blaschke product of degree $2$ and $k(\bar\omega)=(2\bar \omega,\bar\omega^2)$.  
\end{proposition}
\begin{corollary}\label{Cgeoiff}
A map $k\in G(\d)$ is a complex C-geodesic of $G$ if and only if there is a \nd datum $\de$ in $G$ such that $k$ solves $\Kob(\de)$.
\end{corollary}

The following uniqueness result for complex C-geodesics (stated in the introduction as Theorem \ref{extprop30}) plays a vital role in the paper.
\begin{theorem}\label{Kob_ess_unique}
If $\delta$ is a \nd  datum in $G$ then the solution to $\Kob (\delta)$ is essentially unique.
\end{theorem}
\begin{proof}
For discrete datums this statement is contained in \cite[Theorem 0.3]{AY06}; we prove the infinitesimal case.
 Let $\delta= (s_1, v)$ be a \nd infinitesimal  datum in $G$.
According to the definition \eqref{defkob},
\[
\kob{\de} = \inf \{ |\zeta|: \zeta  \; \text{is a datum in } \d \; \text{and there exists} \; h \in G(\d)\;\text{such that} \; h(\zeta) = \delta \}.
\]
Note that, for a datum $ \zeta = (z_1, w)$ in $\d$ and $h \in G(\d)$, 
\[
h(\zeta) = (h(z_1), D_w h(z_1)) = (h(z_1),w h'(z_1)).
\]
Thus
\[
h(\zeta) = \delta \Leftrightarrow h(z_1)= s_1, \quad w (h^1)'(z_1)=v^1 \mbox{ and } \; w (h_2)'(z_1)=v^2.
\]
Hence, for a datum $\delta= (s_1, v)$ in $G$,
\begin{align*}
\kob{\delta} &= \inf_{z_1 \in \d, w \in \c} \left\{
\frac{|w|}{1-|z_1|^2}: \mbox{ there exists } h \in G(\d)\mbox{ such that } \right. \\
	& \hspace*{1.5cm} \left.   h(z_1)= s_1, w (h^1)'(z_1)=v^1 \mbox{ and }  w (h^2)'(z_1)=v^2 \right\}.
\end{align*}
Pick an infinitesimal datum $\zeta = (z_1, w) $ in $\d$ such that $|\zeta|=\kob{\de}$; then $\zeta$ is nondegenerate.
Let $h=(s,p)$ solve  $\Kob(\delta)$ and satisfy
 $ h(\zeta) = \delta$. Let
$\Phi_\omega$ solve $\Car (\delta)$. Since $\kob{\cdot} =\car{\cdot} $ on $G$,  $\Phi_\omega \circ h \in \aut \d$. Say $\Phi_\omega \circ h = q$. Then, by
\cite[Lemma 1.1]{AY06}, $p$ is a Blaschke product of degree at most 2 and 
\be \label{spq}
s = 2 \frac{\omega p -q}{1- \omega q}.
\ee
Hence, if $\tau = q^{-1}(\bar\omega) \in \t$, then $(\omega p- q)(\tau) =0$, and so 
$p(\tau) = \bar\omega ^2$. Thus  $p$ is a  Blaschke product of degree at most 2 which satisfies
\[
 p(z_1)= s_1, \quad  w (p)'(z_1)=v^2 \quad\; \text{and} \; p(\tau) = \bar\omega ^2.
\]
These properties determine $p$ uniquely.

Since $q$ is an automorphism of $\d$ such that  
\[
q (\zeta) = \Phi_\omega \circ h (\zeta) = \Phi_\omega (\delta),
\]
$q$ is uniquely determined. Thus $s$ is also uniquely determined
by equation \eqref{spq}.  Hence there is a unique solution $h$ of $\Kob(\de)$ for which $h(\zeta)=\de$.  It follows that the solution of $\Kob(\de)$ is essentially unique.
\end{proof}

\section{Automorphisms of $G$}\label{autom}
 For every $m\in\aut\d$ we define a map $\tilde m\in G(G)$ by 
\be\label{defmtilde}
\tilde m(z+w,zw)= (m(z)+m(w), m(z)m(w)) \qquad \mbox{ for all } z,w\in\d.
\ee
It is easy to see that $\tilde m$ is an automorphism of $G$ and that $m\mapsto \tilde m$ is a homomorphism from $\aut\d$ to $\aut G$.  In fact it is an isomorphism.
\begin{theorem}\label{autosG}
\begin{enumerate}[\rm (1)]
\item Every automorphism of $G$ is of the form $\tilde m$ for some 
$m$ in  $\aut\d$.

\item Every automorphism of $G$ leaves the royal variety $\calr$ invariant.

\item The automorphisms of $G$ act transitively on the set of flat geodesics of $G$.
\end{enumerate}
\end{theorem}
\index{automorphisms of $G$}
Statement (1) is \cite[Theorem 4.1]{AY08}, proved by an operator-theoretic method; a proof which uses Cartan's Classification Theorem is given in \cite{jp04} and \cite{jp}.  Given statement (1), statements (2) and (3) are easy to verify.
\begin{corollary}\label{corautos}
Every automorphism of $G$ extends continuously to a bijective self-map of $\Ga$.
\end{corollary}
It follows from Theorem \ref{autosG} and the fact that every automorphism of $\d$ extends continuously to a bijective self-map of $\d^-$.

The automorphisms of $G$ are linear fractional maps.  If $m=\tau B_\al$ for some $\tau\in\t, \, \al\in\d$ then, for $s\in G$,
\be\label{formmtilde}
\tilde m(s)=\frac{1}{1-\bar\al s^1+\bar\al^2 s^2}\left(\tau(-2\al+(1+|\al|^2)s^1-2\bar\al s^2), \tau^2(s^2-\al s^1+\al^2)\right).
\ee
\index{$\tilde m$}
\section{A trichotomy theorem}\label{Trich}

Since the classification of datums in Definition \ref{extdef10} depends on the number of functions $\Phi_\omega$ that solve the corresponding Carath\'eodory problem, it is important to know what the possibilities are for this number.
\newtheorem{trichot}[theorem]{Trichotomy theorem}
\index{trichotomy}
\begin{trichot}\label{trichotomy} Let $\de$ be a \nd datum in $G$. Exactly one of the following assertions is true.
\begin{enumerate}[\rm (1)]
\item There is a unique $\omega \in \t$ such that  $\Phi_\omega$ solves $\Car(\de)$; 
\item  there are exactly two points $\omega_1, \omega_2 \in \t$ such that  $\Phi_{\omega_1}$,  $\Phi_{\omega_2}$
solve $\Car(\de)$; 
\item for all $\omega \in \t$, $\Phi_\omega$ solves $\Car(\de)$.
\end{enumerate}
Moreover, all three possibilities do occur.
\end{trichot}
\begin{proof}
In the case that $\de$ is a discrete datum, the statement is \cite[Theorem 1.6]{AY06}. We sketch a proof for a infinitesimal datum $\de=(s_1,v)$: it is a combination of the arguments in \cite[Section 4]{AY04} and \cite[Theorem 1.6]{AY06}, where fuller versions of the reasoning can be found.

For $\kappa \ge 0$ let $S_\kappa, P_\kappa$ be the
operators on the Hilbert space $\mathbb{C}^2$ (with its
standard inner product) given by
\begin{equation}\label{defSPOP}
S_\kappa =
\left[\begin{array}{cc} s_1^1 & \kappa v^1\\
0&s_1^1\end{array}\right],
\qquad P_\kappa =
\left[\begin{array}{cc} s_1^2 & \kappa v^2\\
0& s_1^2\end{array}\right].
\end{equation}
$S_\kappa, P_\kappa$ are commuting operators and $\si(S_\kappa,P_\kappa)=\{s_1\}\subset G$. Moreover the spectral radius of $S_\kappa$ is $|s_1^1|$, which is less than $2$.  For any $f\in \d(G)$,
\[
f(S_\kappa,P_\kappa)= \bbm f(s_1)& \kappa D_vf(s_1) \\ 0 & f(s_1) \ebm.
\]
Hence
\begin{align*}
\|f(S_\kappa,P_\kappa)\| \leq 1\quad &\Leftrightarrow \quad  \frac{\kappa |D_vf(s_1)|}{1-|f(s_1)|^2} \leq 1 \\
	&\Leftrightarrow \quad \kappa |f(\de)| \leq 1.
\end{align*}
Recall that $G$ is said to be a {\em spectral domain} for the commuting pair $(S,P)$ if $\si(S,P) \subset G$ and 
\[
\|f(S,P)\| \leq \sup_G |f|
\]
for every $f\in \c(G)$.  It follows that $G$ is a spectral domain for $(S_\kappa,P_\kappa)$ if and only if $\kappa \car{\de} \leq 1$.

On the other hand, by \cite[Theorem 3.2]{AY04}, $G$ is a spectral domain for $(S_\kappa,P_\kappa)$ if and only if $\|\Phi_\omega(S_\kappa,P_\kappa)\|\leq 1$ for all $\omega\in\t$, or equivalently, if and only if
\be\label{pencil}
(2-\omega S_\kappa)^*(2-\omega S_\kappa)- (2\omega P_\kappa-S_\kappa)^*(2\omega P_\kappa-S_\kappa)\geq 0
\ee
for all $\omega\in\t$.  The inequality \eqref{pencil} can be written
\[
\re \{\omega (S_\kappa- S_\kappa^* P_\kappa)\} \leq 1-P_\kappa^*P_\kappa.
\]
In the event that $\|P_\kappa\| < 1$ (that is, $\kappa |v^2| < 1-|s_1^2|^2$) we define
\[
T_\kappa= (1-P_\kappa^*P_\kappa)^{-\half} (S_\kappa- S_\kappa^* P_\kappa)(1-P_\kappa^*P_\kappa)^{-\half};
\]
in the case that $\|P_\kappa\|=1$, we take $T_\kappa$ to be the unique operator on $\ran (1-P_\kappa^*P_\kappa)$ such that 
\[
(1-P_\kappa^*P_\kappa)^{\half} T_\kappa (1-P_\kappa^*P_\kappa)^{\half}=  S_\kappa- S_\kappa^* P_\kappa.
\]
Now the condition \eqref{pencil} can be written
\[
\re\{\omega T_\kappa\} \leq 1 \qquad \mbox{ for all } \omega\in\t,
\]
which is equivalent to the statement that the numerical range $W(T_\kappa)$ is contained in $\d^-$.
For the extremal choice $\kappa=1/\car{\de}$, $W(T_\kappa) $ touches the unit circle at one or more points $\omega$, and these values of $\omega$ are precisely those for which $\Phi_\omega$ solves $\Car(\de)$.  Since $T_\kappa$ has rank one or two, $W(T_\kappa)$ is either an ellipse (with its interior), a point or a line segment.  It follows that $W(T_\kappa) \cap \t$ consists of either a single point, a pair of points or the whole of $\t$.  Thus the three possibilities for the solutions $\Phi_\omega$ of $\Car(\de)$ are as described in the theorem. 

Examples to show that all three possibilities in Theorem \ref{trichotomy} do occur may be found in the proof of \cite[Theorem 1.6]{AY06}.
\end{proof}

\section{Datums for which all $\Phi_\omega$ are extremal}\label{flatroy}
For the proof of the Pentachotomy Theorem in Section \ref{5types} it is required to characterize the datums $\de$ for which $\rho_\de$ is constant on $\t$, or in other words, for which $\Phi_\omega$ is extremal for $\Car(\de)$ for all $\omega\in\t$.
The next result is a re-statement of \cite[Theorem 5.5]{AY04}.

\begin{theorem}\label{flatandroyal}
Let $\de=(s_1,s_2)$ be a \nd discrete datum in $G$.  The following statements are equivalent.
\begin{enumerate}[\rm (1)]
\item $\Phi_\omega$ solves $\Car(\de)$ for every $\omega\in\t$;
\item  $s_1,s_2$ either both lie in the royal variety $\calr$ or both lie in the set
\[
\calf_\beta =\{(\beta+\bar\beta z,z): z\in\d\}
\]
for some  $\beta\in\d$.
\end{enumerate}
\end{theorem}

We also need the infinitesimal version of Theorem \ref{flatandroyal}.
\begin{theorem}\label{flatroyalesimal}
Let $\de=(s_1,v)$ be a \nd infinitesimal datum in $G$.  The following statements are equivalent.
\begin{enumerate}[\rm (1)]
\item $\Phi_\omega$ solves $\Car(\de)$ for every $\omega\in\t$;
\item  either (a) there exists $z\in\d$ such that $s_1=(2z,z^2)$ and $v$ is collinear with $(1,z)$ or (b) there exist $\beta,z \in\d$ such that $s_1=(\beta+\bar\beta z,z)$ and $v$ is collinear with $(\bar\beta,1)$.
\end{enumerate}
\end{theorem}
\begin{proof}
First show that (2)$\Rightarrow$(1). By the definition of  
$\rho_\delta$, for $\omega \in \t$, 
\[
\rho_\delta(\omega) =\left|\Phi_{\omega}(s_1,v) \right|^2 = \left|(\Phi_{\omega}(s_1),D_v\Phi_{\omega}(s_1)) \right|^2.
\]
Suppose (a) holds, that is, $s_1= (2z, z^2)$ and $v=\la(1,z)$ for some $z \in \d$ and $\la\neq 0$.
One can see that $\Phi_{\omega}(s_1) =-z$ and
\begin{align*}
D_{(1,z)}\Phi_{\omega}(s_1) & = \left[\frac{\partial }{\partial s^1}
\frac{2 \omega s^2 - s^1}{2 - \omega s^1} + z \frac{\partial }{\partial s^2}
\frac{2 \omega s^2 - s^1}{2 - \omega s^1} \right]_{|s = (2z, z^2)}\\
	& =\left[
\frac{-2+ 2 \omega^2 s^2 +2 z \omega (2 - \omega s^1)}{(2 - \omega s^1)^2}  \right]_{|s = (2z, z^2)}\\
& = - \frac{1}{2}.
\end{align*}
Thus
\[
\Phi_\omega(\de)=(\Phi_\omega(s_1), D_{\la(1,z)}\Phi_\omega(s_1))= (-z,-\half\la).
\]
Hence
\[
\rho_\delta(\omega) = |(-z,- \half\la)|^2.
\] 
Thus (a) implies that $\rho_\delta$ is constant on $\t$.\\

Suppose (b) holds, that is,  $s_1= (\beta + \bar{\beta}z, z)$ and $v$ is collinear with $(\bar{\beta},1)$ for some $\beta, z \in \d$.
\begin{align*}
\Phi_\omega(s_1) &= \frac{2\omega z - (\beta + \bar{\beta}z)}{2-\omega (\beta + \bar{\beta}z)} \\
& = m_\omega(z),
\end{align*}
where
\[  m_\omega(z)= 
c \frac{z-\alpha}{1- \bar{\alpha}z},\qquad
c =\omega \frac{2 - \bar{\omega }\bar{\beta}}{2-\omega \beta} \in \t \quad \text{ and } \quad
\alpha = \frac{\beta}{2 \omega - \bar{\beta}} \in \d.
\]
Now
\begin{align*}
D_{(\bar{\beta},1)}\Phi_{\omega}(s_1) & = \left[\bar{\beta}\frac{\partial }{\partial s^1}
\frac{2 \omega s^2 - s^1}{2 - \omega s^1} +  \frac{\partial }{\partial s^2}
\frac{2 \omega s^2 - s^1}{2 - \omega s^1} \right]_{|s = (\beta + \bar{\beta}z, z)}\\
	& =\left[\bar{\beta}
\frac{-2(1- \omega^2 s^2)}{(2 - \omega s^1)^2} + \frac{2 \omega}{2 - \omega s^1} \right]_{|s = (\beta + \bar{\beta}z, z)}\\
& = \frac{4 \omega (1 -\re (\omega \beta))}{(2 - \omega(\beta + \bar{\beta}z))^2}.
\end{align*}
It is routine to check that 
\[
m_\omega'(z) = \frac{4 \omega (1 -\re (\omega \beta))}{(2 - \omega(\beta + \bar{\beta}z))^2},
\]
and so 
\[
\Phi_{\omega}(\delta)= (m_\omega(z), m_\omega'(z)).
\]
Therefore, for all $\omega \in \t$,
\[
\rho_\delta(\omega) = |\Phi_{\omega}(\delta)|^2 = |(m_\omega(z), m_\omega'(z))|^2= |m_\omega((z,1))|^2.
\] 
Since the Poincar\'e metric on $\d$ is invariant under $\aut \d$,
\[
\rho_\delta(\omega) =|(z, 1)|^2 = \frac{1}{(1-|z|^2)^2},
\]
which is independent of $\omega \in \t$.
Thus (b) implies that $\rho_\delta$ is constant on $\t$.
Hence, for each $\omega \in \t$, $\Phi_{\omega}$ solves $\Car(\delta)$ in both cases (a) and (b).

Now we show that (1)$\Rightarrow$(2).  Let $\de=(s_1,v)$ be a \nd infinitesimal datum in $G$ and suppose that $\rho_\de$ is constant on $\t$.  Write
\[
s_1=(\beta+\bar\beta p,p)
\]
for some $\beta, p \in\d$.
Calculation shows that 
\[
|\Phi_\omega(\de)| = \frac{1}{1-|p|^2} \left| \frac{-v^1+2v^2\omega+ (pv^1-s_1^1v^2)\omega^2}{\bar\beta-2\omega+\beta\omega^2}  \right|.
\]
Thus the rational function
\be\label{1.1}
F(w)= \frac{-v^1+2v^2w+ (pv^1-s_1^1v^2)w^2}{\bar\beta-2w+\beta w^2}
\ee
has constant modulus, say $C>0$, for $w\in\t$.  Hence
\begin{align}\label{Fident}
C^2&= F(w)\overline{F(w)}\notag \\
	&= F(w)\overline{F(1/\bar w)}
\end{align}
 for  all $w\in \t$.  Since the right hand side of equation \eqref{Fident} is a rational function of $w$, the equation remains true for all $w\in\c$ for which $F(w)$ and $F(1/\bar w)$ are finite.

First we consider the case that $\beta=0$, when $s_1=(0,p)$. Since $|F|$ is constant on $\t$ we have
\[
|-v^1+2v^2w+pv^1w^2| = \mathrm{const}
\]
for $w\in\t$.  It follows that either $v^1=0$ or $v^2=0$ and $p=0$.  In the former case
\[
\de= \left((0,p),(0,v^2)\right)
\]
and in the latter case
\[
\de=\left((0,0),(v^1,0)\right),
\]
and the statement of the theorem holds.

{\it Now consider the case that $\beta\neq 0$.}  We claim that cancellation occurs in the fractional quadratic $F$.
The zeros of the denominator $\bar\beta-2w+\beta w^2$  in the right hand side of equation \eqref{1.1} are
\[
w_\pm=\frac{1\pm \sqrt{1-|\beta|^2}}{\beta}.
\]
Observe that $w_+\bar w_- =1$, that is, $w_+$ and $w_-$ are symmetric with respect to $\t$.  Moreover, $|w_+|>1$ and $|w_-|<1$.

Suppose there is no cancellation in equation \eqref{1.1}. Then the points $w_\pm$ are both poles of $F$. Equation \eqref{Fident} implies that $\frac{1}{\bar w_+}$ and $\frac{1}{\bar w_-}$ are zeros of $F$, that is, $w_\pm$ are zeros of the numerator of $F$, contrary to assumption.

If there is only one cancellation in equation \eqref{1.1}, say at $w_+$ and
not at $w_-$, then $F(w_+)$ is finite and $w_-$ is a pole of $F$.  Since, by equation \eqref{Fident},
\[
C^2= F(w_+)\overline{F(w_-)},
\]
we have $F(w_+) =0$.

There are only two cases to consider:

Case 1:  $\beta\neq 0$ and there exists exactly one cancellation in equation \eqref{1.1}, which is at  $w_+$ (when $F(w_+) =0$)
or at $w_-$ (when $F(w_-) =0$). We shall show that in this case the datum $\delta$ is such that 
 $s_1=(2z,z^2)$ and $v$ is collinear with $(1,z)$ for some $z\in\d$. 

Case 2: $\beta\neq 0$ and there exist 2 cancellations in equation \eqref{1.1}, at both $w_+$ and $w_-$. We shall show that in this case the datum $\delta$ is such that $s_1=(\beta+\bar\beta z,z)$ and $v$ is collinear with $(\bar\beta,1)$ for some $\beta,z \in\d$.

{\it Case 1}. Let $w_0$ be one of the roots of equation
\begin{equation} \label{denomF}
 \bar\beta-2w+\beta w^2=0.
\end{equation}
On taking the complex conjugate of this equation we obtain the system:
\begin{equation}\label{beta-w0}
\left\{\begin{array}{lcl}
 w_0^2 \beta & + \bar\beta  &=2 w_0\\
\beta & + \bar{w_0}^2 \bar\beta  &=2 \bar{w_0},\\
\end{array}\right.
\end{equation}
which can be solved to give  $ \beta = \frac{2\bar{w_0}}{1+|w_0|^2}$.

If there exists exactly one cancellation in equation \eqref{1.1}, which is at  $w_0$, then, by  L'H$\hat{\rm o}$pital's Rule,
\be\label{Fatw0}
F(w_0) = \frac{2v^2+ 2(pv^1-s_1^1v^2)w_0}{-2 +2\beta w_0}
\ee
Therefore, the facts that there exists exactly one cancellation in equation \eqref{1.1} at  $w_0$ and that $F(w_0) =0$ yield the system
\begin{equation}\label{F-w0-1}
\left\{\begin{array}{lcl}
-v^1+2v^2w_0 &+ (pv^1-s_1^1 v^2)w_0^2&=0\\
2v^2         &+ 2(pv^1-s_1^1 v^2)w_0 &=0,\\
\end{array}\right.
\end{equation}
which is equivalent to 
\begin{equation}\label{F-w0-2}
\left\{\begin{array}{lcl}
v^1  &=v^2 w_0\\
(pv^1-s_1^1 v^2)w_0 &=- v^2.\\
\end{array}\right.
\end{equation}
Since  $\de$ is \nd
 $v \neq 0$.   Hence $v^2\neq 0$ and
\begin{equation}\label{F-w0-3}
\left\{\begin{array}{lcl}
v^1  &=v^2 w_0\\
(p w_0-s_1^1)w_0 &=-1.\\
\end{array}\right.
\end{equation}
Recall that 
\[
s_1^1= \beta+\bar\beta p \;\; \text{and} \;\; \beta = \frac{2\bar{w_0}}{1+|w_0|^2}.
\]
Thus the equation $(p w_0-s_1^1)w_0 =-1$ implies that $p = \frac{1}{w_0^2}$, and so $s_1^1 =  \frac{2}{w_0}$. Hence the datum 
$\delta= \left((\frac{2}{w_0},\frac{1}{w_0^2}),(1, \frac{1}{w_0})\right)$ for $\frac{1}{w_0} \in \d$.\\

{\it Case 2}: there exist 2 cancellations in equation \eqref{1.1}, at both $w_+$ and $w_-$. Equation \eqref{Fident} implies that $\frac{1}{\bar w_+}$ and $\frac{1}{\bar w_-}$ are zeros of $F$, that is, $w_\pm$ are zeros of the numerator of $F$. Therefore, there exists
$\lambda \neq 0$ such that the numerator and the denominator of $F$ are connected by the following equation
\be\label{F-N-D}
-v^1+2v^2w+ (pv^1-s_1^1v^2)w^2 = - \lambda (\bar\beta-2w+\beta w^2).
\ee
Since $s_1^1= \beta+\bar\beta p$, we have
\[
-v^1+2v^2w+ (pv^1-(\beta+\bar\beta p) v^2)w^2 = -\lambda \bar\beta+2 \lambda w-\beta \lambda w^2),\; w \in \c.
\]
Therefore, $v^1 =  \lambda \bar\beta$, $v^2 = \lambda$,
$pv^1-(\beta+\bar\beta p) v^2 = -\beta \lambda$.
Thus,
$v^1 =  \bar\beta v^2$, $pv^1-\bar\beta p v^2 =0$.
The datum 
$\delta= \left((\beta+\bar\beta p, p), v^2(\bar\beta,1)\right)$ for some $\beta,p \in\d$.
\end{proof}

\chapter{Types of geodesic: a crib and some cartoons}\label{crib}
The five types of datum $\de$ and their associated complex geodesics $\cald_\de$ play a central role in the paper.
The types are defined in Definition \ref{extdef10} on page \pageref{extdef10}, and the geometric characterization of types
in terms of the intersection of $\cald_\de^-$ with the royal variety $\calr^-$ is contained in Theorem \ref{geothm30} on page \pageref{geothm30}.
We are grateful to a referee for a very good suggestion: to include a `crib sheet' containing the essentials of the terminology, and also cartoons to give the reader a visual representation of the five types.  
\section{Crib sheet}\label{cribsheet}

\flushleft
\begin{tabular*}{15cm}{llll}
{\bf Type of datum $\de$} & {\bf For how many $\omega \in\t$}  & {\bf How many points} & {\bf Balanced?} \\
{\bf and geodesic} $D_\de$ & {\bf is  $\Phi_\omega$ extremal?} &{\bf  are in $\cald_\de^-\cap\calr^-$?}  & \\
 & & & \\
Purely & Unique $\omega_0$, & One in $\calr$, & No \\
unbalanced & $\frac{d^2{\rho_\de}}{dt^2}\left|_{\omega_0}<0\right.$ & one in $\partial\calr$  & \\
& & & \\
Exceptional & Unique $\omega_0$, & One  (double)\footnote{}  & Yes \\
 &  $\frac{d^2{\rho_\de}}{dt^2}\left|_{\omega_0}=0\right.$ & in $\partial\calr$ & \\
  & & & \\
Purely  & Two & Two, & Yes \\
balanced & & both in $\partial\calr$ & \\
 & & & \\
Flat  & All $\omega\in\t$ & One, in $\calr$ & No\\
 & $C(s)=s^2$ solves & & \\
 &  $\Car(\de)$& & \\
 & & & \\
Royal & All $\omega\in\t$ & All of $\calr^- $ & Yes \\
& $C(s)= \half s^1$ solves & & \\
 &  $\Car(\de)$ & & \\
\end{tabular*}
\index{crib sheet}
\index{geodesics!types of}
\index{datum!royal}
\index{datum!flat}
\index{datum!purely balanced}
\index{datum!purely unbalanced}
\index{datum!exceptional}
\vspace* {0.5cm}
{\small{$^1$ In the sense that the multiplicity of the royal point is $2$ -- see Definition \ref{multiplicity}.}}

\vspace* {0.3cm}
\normalsize
Balanced geodesics are those which are parallel to $\calr$, in the sense that they either coincide with or are disjoint from $\calr$ in $G$.

\section{Cartoons}\label{cartoons}
\index{cartoons}
  We draw the intersection with $\r^2$ of a representative of each type of geodesic.  The set $G\cap \r^2$ in the $(s,p)$ plane is the interior of the isosceles triangle with vertices $(\pm 2,1)$ and $(0,-1)$, as in Figure 1.
\begin{figure}[!ht]\label{realgamma}
\includegraphics{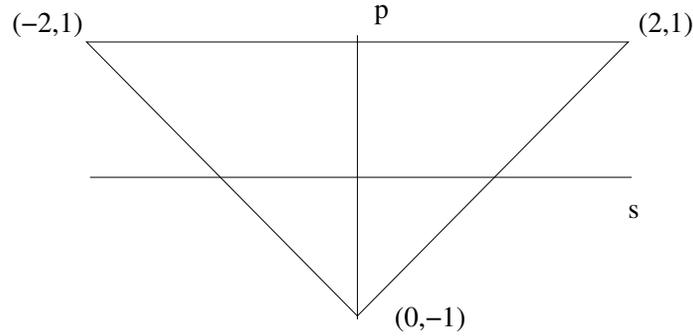}
\caption{The real symmetrized bidisc $G\cap \r^2$}
\end{figure}

\vspace{0.3cm}

The {\bf  royal variety} is the locus $s^2=4p$, as in Figure 2.
\begin{figure}
\label{royalcartoon}
\begin{center}
{\includegraphics{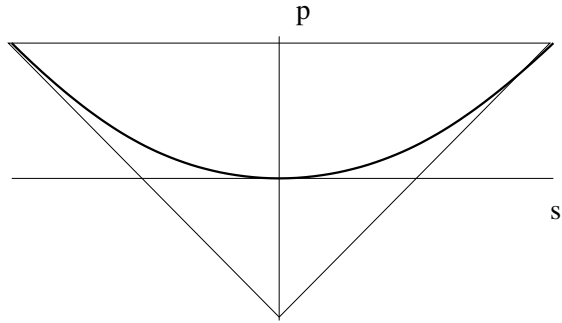}}
\caption{The royal variety}
\end{center}
\end{figure}

\vspace{0.3cm}

{\bf Flat geodesics} have the form $s= \beta +\bar\beta p$ for some $\beta\in\d$.  This locus meets $G\cap\r^2$ if and only $\beta$ is real,
in which case its cartoon is as in Figure 3.   A flat geodesic has a unique royal point, which is in $\calr$.

\begin{figure}
\label{flatcartoon}

\includegraphics{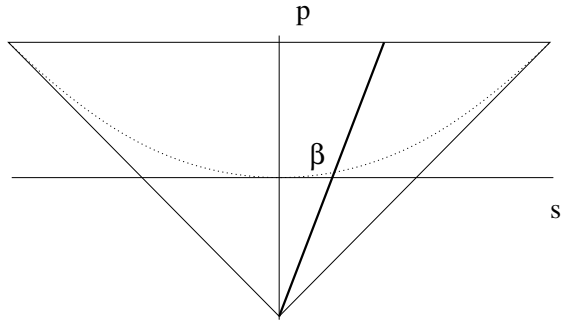}

\caption{A flat geodesic}
\end{figure}

\vspace{0.3cm}

{\bf Purely unbalanced geodesics,} by Theorem \ref{formgeos}, up to an automorphism of $G$, have the form $k_r(\d)$ for some $r\in\d$, where
\[
k_r(z)= \frac{1}{1-rz} \left(2(1-r)z, z(z-r)\right)
\]
for some $r\in(0,1)$.  The geodesic $k_r(\d^-)$ has royal points $(0,0)$ in $\calr$ and $(2,1)$ in $\partial\calr$.  Taking $r=\half$ we obtain the curve 
\[
p=\frac{s(3s-2)}{2(s+2)}, \qquad -\tfrac 23 \leq s \leq 2,
\]
which has graph as in Figure 4.

\begin{figure}
\label{pureunbalcartoon}
\begin{center}
\includegraphics[viewport=30mm 165mm 180mm 240mm, width=8cm, height=37mm,clip] {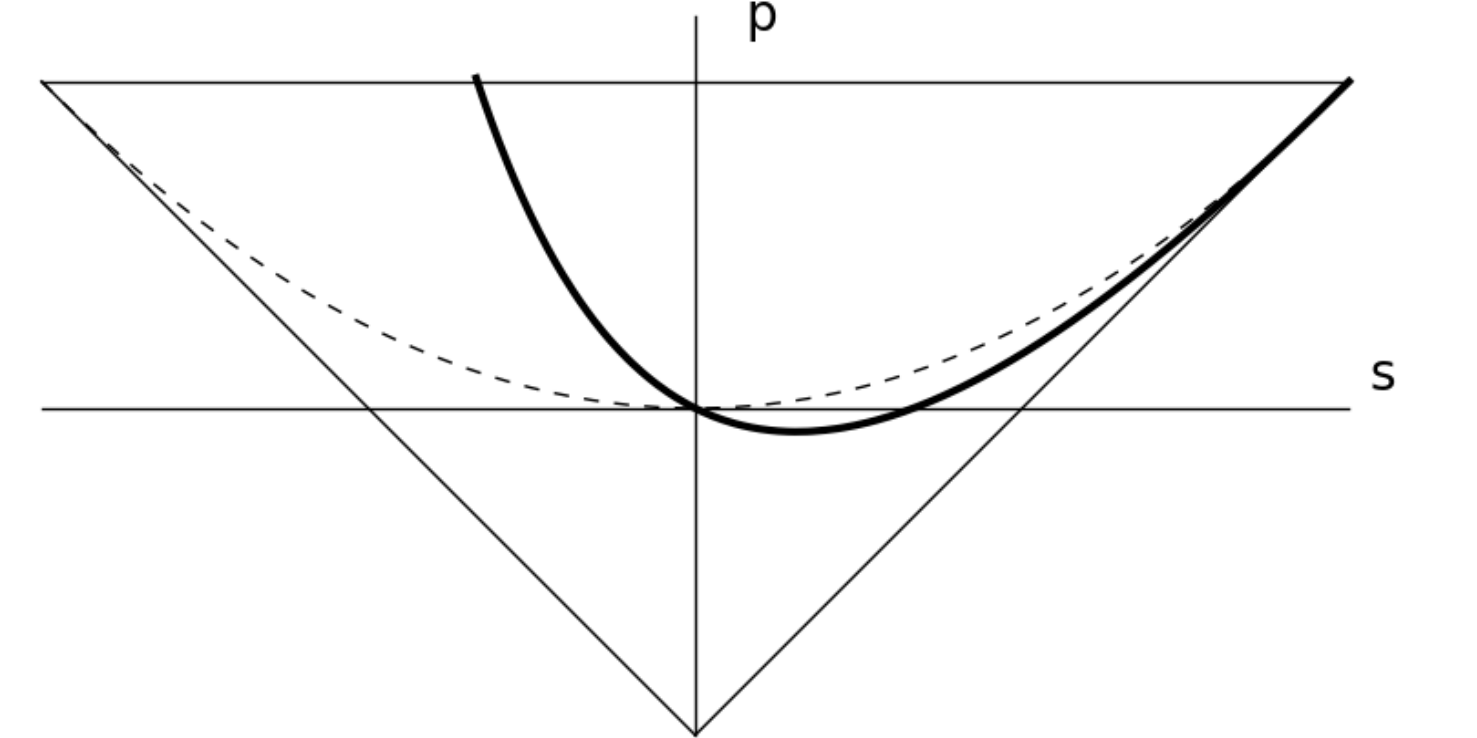}
\caption{A purely unbalanced geodesic}
\end{center}
\end{figure}

\vspace{0.3cm}

{\bf Purely balanced geodesics}, up to automorphisms of $G$, are of the form $g_r(\d)$, for some $r\in (0,1)$, by Theorem \ref{formgeos}, where
\[
g_r(z)=(z+B_r(z), zB_r(z)).
\]
The geodesic $g_r(\d^-)$ has the two royal points $(\pm 2,1)$, both in $\partial\calr$.
In Figure 5 we have taken $r=\half$.

\begin{figure}
\label{purebalcartoon}
\includegraphics{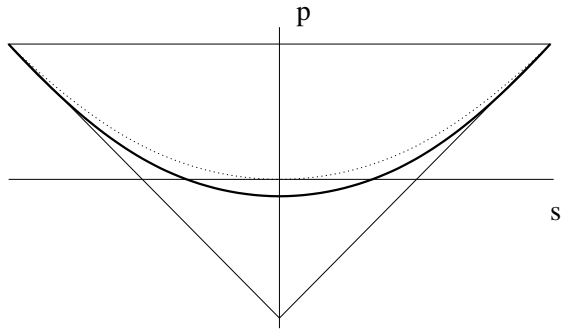}
\caption{A purely balanced geodesic}
\end{figure}

\vspace{0.3cm}

The remaining type of geodesic is the { \bf exceptional type}.    According to Theorem \ref{formgeos}, up to automorphisms of $G$ they have the form $h_r(\d)$ for some $r>0$, where 
\[
h_r(z)= (z+m_r(z), zm_r(z))
\]
and
\[
m_r(z)= \frac{(r-i)z+i}{r+i-iz}.
\]
Any such geodesic $h_r(\d^-)$ has the unique royal point $(2,1)$, and $h_r(z)\in\r^2$ for any $z$ lying on the circle of centre $1-ir$ and radius $r$.  When $r$ is chosen to be $1$ the locus in Figure 6 results.

\begin{figure}
\label{exceptcartoon}
\includegraphics{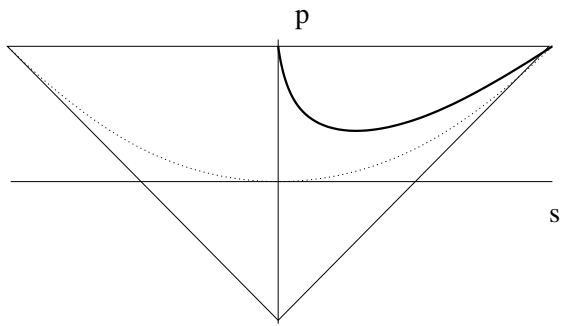}
\caption{An exceptional geodesic}
\end{figure}

\backmatter
\bibliographystyle{amsalpha}
\bibliography{} \label{refs}
\newpage \label{ind}
\printindex 

\end{document}